\PassOptionsToPackage{table}{xcolor}
\documentclass[12pt]{amsart}
\usepackage{amsfonts,amssymb,enumerate,mathtools,enumitem}
\usepackage[usenames,dvipsnames]{color}
\usepackage{hyperref}
\usepackage{comment}
\usepackage{array}
\usepackage[all,ps,cmtip,rotate]{xy} 
\usepackage{tikz}
\usepackage{upgreek}
\usepackage[table]{xcolor}

\usepackage{mathrsfs}
\usepackage{eucal}
\usepackage[normalem]{ulem}

\setlist[itemize]{noitemsep,nolistsep,topsep=-3pt}
\setlist[enumerate]{noitemsep,nolistsep,topsep=-3pt}

\let\mathcal\mathscr

\newcommand{\xrightiso}[1]{ \xrightarrow[{\ \raisebox{0.5ex}[0ex][0ex]{$\sim$}\ }]{#1} }

\def\Z{{\bf Z}}

\def\C{{\bf C}}

\def\A{{\bf A}}

\def\P{{\bf P}}
\def\PP{{\bf P}}

\newcommand{\tS}{\widetilde{S}}

\newcommand{\tsD}{\widetilde{\sD}}
\newcommand{\tY}{\widetilde{\sY}}

\newcommand{\bcA}{\overline\cA}

\newcommand{\hla}{\hat\lambda}
\newcommand{\hg}{\hat{g}}

\newcommand{\st}{\mathsf{t}}

\def\cI{\mathcal{I}}

\def\cA{\mathscr{A}}

\def\cF{\mathscr{F}}

\def\cL{\mathcal{L}}
\def\cO{\mathcal{O}}

\def\cH{\mathscr{H}}
\def\cE{\mathscr{E}}

\def\cC{\mathcal{C}}

\def\cN{{N}}
\def\cK{\mathscr{K}}
\def\cT{{T}}
\def\cQ{\mathcal{Q}}
\def\cR{\mathcal{R}}
\def\cU{\mathcal{U}}

\def\cW{\mathcal{W}}

\def\fS{\mathfrak{S}}

\newcommand{\uvt}{\upvartheta}
\newcommand{\tuvt}{\widehat\upvartheta}

\newcommand{\tf}{\tilde{f}}

\def\bp{\mathbf p}
\def\bq{\mathbf q}
\def\bv{\mathbf{v}}

\def\lra{\longrightarrow}
\def\llra{\hbox to 10mm{\rightarrowfill}}
\def\lllra{\hbox to 15mm{\rightarrowfill}}

\def\llla{\hbox to 10mm{\leftarrowfill}}
\def\lllla{\hbox to 15mm{\leftarrowfill}}

\def\thlra{\relbar\joinrel\twoheadrightarrow}
\def\hra{\hookrightarrow}
\def\lhra{\ensuremath{\lhook\joinrel\relbar\joinrel\rightarrow}}
\def\isom{\simeq}
\def\eps{\varepsilon}
\def\vp{\varpi}

\def\pr{\mathrm{pr}}

 \def\vide{\varnothing}
  \def\emptyset{\varnothing}

\newcommand{\isomto}{\xrightarrow{\,{}_{\scriptstyle\sim}\,}}

\DeclareMathOperator{\codim}{codim}

\DeclareMathOperator{\Coker}{Coker}

\DeclareMathOperator{\Ext}{Ext}

\DeclareMathOperator{\Gr}{\mathrm{Gr}}
\DeclareMathOperator{\IGr}{\mathrm{IGr}}

\DeclareMathOperator{\CGr}{\mathrm{CGr}}

\DeclareMathOperator{\Fl}{Fl}

\DeclareMathOperator{\Hom}{Hom}

\DeclareMathOperator{\Ker}{Ker}

\DeclareMathOperator{\lin}{\underset{\mathrm lin}{\equiv}}

\DeclareMathOperator{\PGL}{PGL}

\DeclareMathOperator{\Spec}{Spec}

\DeclareMathOperator{\rank}{rk}
\DeclareMathOperator{\Sing}{Sing}

\DeclareMathOperator{\Sym}{\mathsf S}

\def\spe{\text{spe}}

\def\llra{\hbox to 10mm{\rightarrowfill}}
\def\lllra{\hbox to 15mm{\rightarrowfill}}

\def\bw#1#2{\textstyle{\bigwedge\hskip-0.9mm^{#1}}\hskip0.2mm{#2}}
\def\sbw#1#2{\small{\bigwedge\hskip-0.9mm^{#1}}\hskip0.2mm{#2}}

\newtheorem{lemm}{Lemma}[section]
\newtheorem{theo}[lemm]{Theorem}
\newtheorem{coro}[lemm]{Corollary}
\newtheorem{prop}[lemm]{Proposition}

\theoremstyle{remark}
\newtheorem{defi}[lemm]{Definition}
\newtheorem{rema}[lemm]{Remark}

\newenvironment{aenumerate}{\begin{enumerate}[label={\textup{(\alph*)}}]}{\end{enumerate}}

\def\hk{hyper-K\"ahler}

\def\id{\mathsf{id}}

\def\ord{\mathrm{ord}}
\def\spe{\mathrm{spe}}

\newcommand{\Bl}{\mathop{\mathrm{Bl}}\nolimits}

\newcommand{\wtY}{\widetilde{\sY}}

\DeclareMathOperator{\OGr}{OGr}
\newcommand{\OGrm}{\OGr'}

\newcommand{\Ap}{{A^\perp}}

 \def\setminus{\smallsetminus}
\def\ssm{\smallsetminus}
\def\cong{\isom}

\newcommand{\rtx}{\tau}

\newcommand{\chG}{\mathrm{D}}
\newcommand{\G}{\mathrm{G}}
\newcommand{\hG}{\widehat{\G}}

\newcommand{\tD}{\widetilde{D}}

\newcommand{\Gp}{\G^+}
\newcommand{\Gs}{\G^\sigma}
\newcommand{\Gt}{\G^\rtx}
\newcommand{\Gst}{\G^{\sigma\rtx}}
\newcommand{\F}{\mathrm{F}}
\newcommand{\Fs}{\F^\sigma}
\newcommand{\Ft}{\F^\rtx}
\newcommand{\Fst}{\F^{\sigma\rtx}}
\newcommand{\Fo}{\mathring{\F}}

\newcommand{\rT}{\mathrm{T}}

\newcommand{\sD}{{\mathsf{Q}}}
\newcommand{\sY}{{\mathsf{Y}}}
\newcommand{\sZ}{{\mathsf{Z}}}

\begin{document}
\title 
{Quadrics on Gushel--Mukai varieties}

\author[O.\ Debarre]{Olivier Debarre}
\address{Universit\'e   Paris Cit\'e and Sorbonne Universit\'e, CNRS,   IMJ-PRG, F-75013 Paris, France}
\email{{\tt olivier.debarre@imj-prg.fr}}
 
\author[A. Kuznetsov]{Alexander Kuznetsov}
\address{Algebraic Geometry Section, 
Steklov Mathematical Institute of Russian Academy of Sciences,
8 Gubkin str., Moscow 119991 Russia;
Laboratory of Algebraic Geometry, National Research University Higher School of Economics, Russia}
\email{{\tt  akuznet@mi-ras.ru}}

\thanks{This project has received funding from the European
Research Council (ERC) under the European
Union's Horizon 2020 research and innovation
programme (ERC-2020-SyG-854361-HyperK)
A.K. was partially supported by the HSE University Basic Research Program.}

\begin{abstract}
We study Hilbert schemes of quadrics of dimension~$k \in \{0,1,2,3\}$ on 
smooth Gushel--Mukai varieties~$X$ of dimension~$n\in\{2,3,4,5,6\}$
by relating them to the relative Hilbert schemes of linear subspaces of dimension~$k + 1$
of a certain family, naturally associated with~$X$, of quadrics of dimension~$n - 1$ over the blowup of~$\P^5$ at a point.
\end{abstract}

\date{\today}

\maketitle


\section{Introduction}
\label{sec:intro}

This paper  
is a continuation of the series of papers~\cite{DK1,DK2,DK4,DK5} 
where we study Gushel--Mukai (GM) varieties of dimension~$n \in \{2,3,4,5,6\}$, 
that is, dimensionally transverse intersections 
\begin{equation*}
X = \CGr(2,V_5) \cap \P(W) \cap Q
\end{equation*}
in~$ \P(\C \oplus \bw2V_5)$ of the cone~$\CGr(2,V_5)$ over the Grassmannian~$\Gr(2,V_5)$
 of $2$-dimensional subspaces 
in a $5$-dimensional vector space~$V_5$,
 a linear space~$\P(W) $ of dimension~\mbox{$n + 4$,}
and a quadric~$Q \subset \P(W)$.\ 
 We describe the Hilbert scheme~$\G_k(X)$ of~$k$-dimensional quadrics on~$X$,
for various values of~$k$ and~$n$.\ 
These include Hilbert squares~$\G_0(X)$ of GM surfaces (subschemes of length~2 are zero-dimensional quadrics),
Hilbert schemes~$\G_1(X)$ of conics on GM threefolds and fourfolds, 
Hilbert schemes~$\G_2(X)$ of quadric surfaces on GM fourfolds, fivefolds, and sixfolds,
and Hilbert schemes~$\G_3(X)$ of~$3$-dimensional quadrics on GM fivefolds and sixfolds.

Some of these Hilbert schemes have been studied before 
and were shown to play important roles for the geometry of GM varieties.\ 
For instance, the Hilbert scheme of conics on a general GM threefold 
was one of the main objects of study in the pioneering work~\cite{lo}
and its geometric significance was demonstrated in~\cite[Theorem~9]{IM07} and~\cite{dim}.\ 
Similarly, in~\cite{IM}, the Hilbert scheme of conics on a general GM fourfold was shown to be related to a \hk\  fourfold.\ 
More recently,  a categorical significance was given to these Hilbert schemes in~\cite{LZ,GLZ}.

In the above papers, various generality assumptions were made.\ 
Here, we suggest instead a general approach that allows us to give a uniform description, 
valid for all smooth GM varieties, in terms of double coverings of Eisenbud--Popescu--Walter (EPW) loci 
and other varieties naturally associated with GM varieties.\ 
Our results are too varied to be explained easily in this introduction, 
so we only state two sample results which were announced in the introduction of~\cite{DK5} 
and we refer the reader to the body of the paper for details (and to Section~\ref{sec:prelim} for the notation).\ 
The first result discusses conics on GM threefolds.

\begin{theo}
\label{thm:g1x3-intro}
Let~$X$ be a smooth ordinary GM threefold and let~$A$ be its associated Lagrangian subspace.\ 
The Hilbert scheme~$\G_1(X)$ of conics on~$X$ is  isomorphic 
to the blowup of the double dual EPW surface~$\wtY^{\ge 2}_\Ap$ at a point.
\end{theo}

This is Theorem~\ref{thm:g1-x3}, which also includes the case of special GM threefolds.

The second sample result discusses quadric surfaces on GM fivefolds.

\begin{theo}
\label{thm:g2x5-intro}
Let~$X$ be a smooth ordinary GM fivefold and let~$A$ be its associated Lagrangian subspace.\ 
The Hilbert scheme~$\G_2(X)$ of quadric surfaces on~$X$ has two disjoint irreducible components
\begin{equation*}
\G_2(X) = \overline{\G^0_2(X)} \sqcup \Gs_2(X),
\end{equation*}
where~$\Gs_2(X) \cong \P^4$, while~$\overline{\G^0_2(X)}$ has dimension~$3$ and is
an \'etale-locally trivial $\P^1$-fibration over the double dual EPW surface~$\wtY^{\ge 2}_\Ap$.
\end{theo}

This is Theorem~\ref{thm:g2-x5}, which also includes the case of special GM fivefolds.\ 

In each of these two cases, the descriptions only involve the (dual) double EPW surfaces.\ 
In general, the complexity of the description of the scheme~$\G_k(X)$ depends on the number
\begin{equation*}
\label{eq:ell}
\ell \coloneqq 2k + 3 - \dim(X).
\end{equation*} 
In particular, if~$\ell \ge 1$, the main component of~$\G_k(X)$, defined in Definition~\ref{def:st-quadrics}, has a fibration structure 
over the double dua EPW variety~$\wtY_\Ap^{\ge \ell}$,  whose definition is recalled in Section~\ref{subsec:ld}.\ 
Moreover, the Hilbert schemes~$\F_{k+1}(X)$ and~$\F_{k+1}(M_X)$ of~$(k+1)$-dimensional linear spaces
on~$X$ and on its Grassmannian hull
\begin{equation*}
M_X \coloneqq \CGr(2,V_5) \cap \P(W)
\end{equation*}
also get involved.\ 
As we will see, the situation simplifies considerably if~$X$ is ordinary 
and if~$\F_{k+1}(X) = \vide$ 
(the latter holds if~$\ell \ge 1$ and~$X$ is general, or if~$\ell \ge 2$).

\subsection{General approach}

Our results on the Hilbert schemes~$\G_k(X)$ all rely on a uniform construction
which we sketch in the rest of this introduction.\

The Grassmannian~$\Gr(2,V_5) \subset \P(\bw2V_5)$ is the intersection 
of the $5$-dimen\-sional space of \emph{Pl\"ucker quadrics} in~$\P(\bw2V_5)$
and this space can be naturally identified with the space~$V_5$;
consequently, the same is true for the cone~$\CGr(2,V_5) \subset \P(\C \oplus \bw2V_5)$
and for the Grassmannian hull~$M_X$ of a GM variety~$X$.\ 
Since~$X = M_X \cap Q$, it follows that~$X$ is the intersection of a~$6$-dimensional space~$V_6$ of quadrics 
containing the space of Pl\"ucker quadrics as a hyperplane~$V_5 \subset V_6$.\ 
Therefore, if~$\Sigma \subset X$ is a quadric of dimension~$k$  
and~$[\Sigma] \in \G_k(X)$ is the corresponding point of the Hilbert scheme,   
the linear span
 \begin{equation*}
\langle \Sigma \rangle \cong \P^{k+1} \subset \P(W) 
\end{equation*}
of~$\Sigma$ is contained in the intersection of a 5-dimensional subspace~$U_5 \subset V_6$ of quadrics containing~$X$
(and~$U_5$ is determined by~$\Sigma$ unless~$X$ contains $\langle \Sigma\rangle$,
that is, $[\langle \Sigma\rangle] \in \F_{k+1}(X)$).\ 
Similarly, $\langle \Sigma \rangle$ is contained in a 4-dimensional space~$U_4 \subset V_5$ 
of Pl\"ucker quadrics containing~$M_X$
(and~$U_4$ is determined by~$\Sigma$ unless~$M_X$ contains $\langle \Sigma\rangle$,
that is, $[\langle \Sigma\rangle] \in \F_{k+1}(M_X)$).\ 
So, this defines a rational map 
\begin{equation}
\label{eq:map-to-b}
\beta \colon \G_k(X) \dashrightarrow B,
\end{equation}
where
\begin{equation}
\label{base}
B \coloneqq \{ ([U_4],[U_5]) \in \Gr(4,V_5) \times \Gr(5,V_6)\ |\ U_4 \subset U_5 \} .
\end{equation}
The projection~$B \to \Gr(5,V_6) \cong \P(V_6^\vee)$ 
is the blowup of the point~$\bp_X \in \P(V_6^\vee)$, called the {\sf Pl\"ucker point}, 
corresponding to the hyperplane~$V_5 \subset V_6$ of Pl\"ucker quadrics, 
and the projection~$B \to \Gr(4,V_5) \cong \PP(V_5^\vee)$ is a $\PP^1$-fibration.\ 
A detailed analysis of the map~\eqref{eq:map-to-b}  (sketched below) eventually leads to a description of the Hilbert scheme~$\G_k(X)$.

Let~$b = ([U_4],[U_5]) \in B$.\
One shows (see Lemma~\ref{lemma:u4-vanishes})
 that the intersection in~$\P(W)$ 
of the quadrics parameterized by the subspace~$U_4 \subset V_6$
is the union of~$M_X$ and a linear subspace~\mbox{$\P(\cW_{[U_4]}) \isom \P^n \subset \P(W)$}, where
\begin{equation}
\label{eq:cw-u4}
\cW_{[U_4]} \coloneqq (\C \oplus \bw2U_4) \cap W.
\end{equation} 
Therefore, the intersection in~$\P(\cW_{[U_4]})$
of the quadrics parameterized by the subspace~\mbox{$U_5 \subset V_6$}
 is a quadric hypersurface~$\cQ_b \subset \P(\cW_{[U_4]})$.\ 
Letting the point~$b$ vary in~$B$, we obtain a vector subbundle~$\cW \subset W \otimes \cO_B$ of rank~$n + 1$
and a quadric fibration
\begin{equation*}
\cQ \subset \P_B(\cW) \lra B
\end{equation*}
of relative dimension~$n - 1$ (see Section~\ref{ss:qf} for details of this construction).

If~$\beta([\Sigma]) = b$ and~$[\langle\Sigma\rangle] \notin \F_{k+1}(M_X)$, then~$[\langle \Sigma \rangle] \in \F_{k+1}(\cQ_b)$.\ 
Therefore, the map~$\beta$ factors through a rational map 
\begin{equation*}
\label{eq:tbeta}
\tilde\beta \colon \G_k(X) \dashrightarrow \F_{k+1}(\cQ/B) = \OGr_B(k+2, \cQ)
\end{equation*} 
 to the relative Hilbert scheme~$\F_{k+1}(\cQ/B)$ 
of linear spaces of dimension~$k+1$ in the fibers of~$\cQ/B$,
or, equivalently, to the Grassmannian of vector subspaces of dimension~$k+2$ 
contained in the fibers of~$\cW/B$ that are isotropic with respect to the quadratic equation of~$\cQ$.\

We show that the map~$\tilde\beta$ is birational and can be factored as the composition~$g_k^{-1} \circ \lambda_k$ 
of two birational morphisms that together with the natural projection~$f \colon \OGr_B(k+2,\cQ) \to B$
fit into a diagram (explained in Section~\ref{ss12})
\begin{equation}
\label{eq:diagram-intro}
\vcenter{\xymatrix{
\G_k(X) \ar[dr]_{{\lambda_k}} \ar@{-->}[rr]^-{\tilde\beta} &&
\OGr_B(k+2,\cQ) \ar[dl]^{g_k} \ar[dr]_f
\\
& \chG_k(X) &&
B.
}}
\end{equation}
An analysis of the maps~$\lambda_k$, $g_k$, and~$f$ will give us a description of~$\G_k(X)$.

\subsection{More details}
\label{ss12}

To define the scheme~$\chG_k(X)$, we consider the tautological vector subbundle~$\cR_{k+2}$ of rank~\mbox{$k + 2$} on the
Grassmannian~$\Gr(k+2,W)$ and the composition 
\begin{equation}
\label{eq:v6-s2r}
V_6 \otimes \cO_{\Gr(k+2,W)} \lra \Sym^2\!W^\vee \otimes \cO_{\Gr(k+2,W)} \lra \Sym^2\!\cR_{k+2}^\vee,
\end{equation}
where the first arrow uses the identification of~$V_6$ with the space of quadrics through~$X$
and the second arrow is tautological.\ 
Then, we show that the map~$\G_k(X) \to \Gr(k + 2, W)$ defined by~$[\Sigma] \mapsto [\langle \Sigma \rangle]$ 
factors through the degeneracy locus
\begin{equation*}
\chG_k(X) \subset \F_{k+1}(\P(W)) = \Gr(k + 2, W)
\end{equation*}
where the composition~\eqref{eq:v6-s2r} has rank at most~$1$.\ 
This defines the map~$\lambda_k$.

More precisely, we check in Proposition~\ref{prop:pik} that~$\lambda_k \colon \G_k(X) \to \chG_k(X)$ is a birational morphism;
in particular, it is an isomorphism away from the zero locus 
\begin{equation*}
\F_{k+1}(X)\subset \chG_k(X) \subset \Gr(k+2,W) 
\end{equation*}
of~\eqref{eq:v6-s2r}.\ 
Moreover, if~$\F_{k+1}(X) $ is empty (which happens for example in the situations of Theorems~\ref{thm:g1x3-intro} and~\ref{thm:g2x5-intro}), 
the map~$\lambda_k$ is an isomorphism and~$\G_k(X) = \chG_k(X)$.

In order to construct the map~$g_k$, we next show (see Lemma~\ref{lem:ogrb-triples}) the equality
\begin{equation*}
\OGr_B(k+2,\cQ) = 
\left\{ ([U_4],[U_5],[R_{k+2}]) \in  B  \times \Gr(k+2,W) 
\,\left|\,
\begin{array}{l}
U_5 \subset \Ker(V_6 \to \Sym^2\!R_{k+2}^\vee)\\[.5ex]
R_{k+2} \subset \cW_{[U_4]}
\end{array}
\right.\right\},
\end{equation*}
where the map~$V_6 \to \Sym^2\!R_{k+2}^\vee$ is the fiber of~\eqref{eq:v6-s2r} at~$[R_{k+2}]$
and~$\cW_{[U_4]}$ is defined in~\eqref{eq:cw-u4}.\ 
It is obvious from this description that the natural projection~$\OGr_B(k+2,\cQ) \to \Gr(k+2,W)$ factors through~$\chG_k(X)$ 
and this defines the map~$g_k  \colon \OGr_B(k+2,\cQ) \to \chG_k(X)$.

Using the observations of Lemmas~\ref{lemma:u4-vanishes} and~\ref{lem:ogrb-triples} (see Proposition~\ref{prop:hpik} for details), 
it is not hard to show that~$g_k$ is an isomorphism away from the zero locus of the morphism
\begin{equation*}
V_5 \otimes \cO_{\Gr(k+2,W)} \hookrightarrow 
V_6 \otimes \cO_{\Gr(k+2,W)} \to 
\Sym^2\!W^\vee \otimes \cO_{\Gr(k+2,W)} \to 
\Sym^2\!\cR_{k+2}^\vee,
\end{equation*}
which is nothing but the Hilbert scheme~$\F_{k+1}(M_X)$ of the Grassmannian hull~$M_X$ of~$X$.\ 
One can also describe the fibers of the map~$g_k$ over~$\F_{k+1}(M_X)$ (see Proposition~\ref{lem:hpik-fibers});
in particular, in the situation of Theorem~\ref{thm:g1x3-intro} (where~$\chG_1(X) = \G_1(X)$),
we prove that the morphism~$g_1 \colon \OGr_B(3, \cQ) \to \G_1(X)$ is the blowup of a point and, 
in the situation of Theorem~\ref{thm:g2x5-intro}  (where~$\chG_2(X) = \G_2(X)$), 
that~$g_2$ is an isomorphism~$\OGr_B(4, \cQ) \isomto \overline{\G^0_2(X)}$
onto the main component~$\overline{\G^0_2(X)}$ of~$\G_2(X)$.

Finally, we analyze the second projection~$f \colon \OGr_B(k+2,\cQ) \to B$.\ 
Its fiber over a point~\mbox{$b \in B$} is the Hilbert scheme~$\F_{k+1}(\cQ_b)$ of linear spaces of dimension~$k+1$
on the quadric~$\cQ_b$ of dimension~$n-1$.\ 
This scheme depends on the rank of the quadric, therefore we start by describing
the rank stratification of~$B$ induced by the family of quadrics~$\cQ$.\ 
We show in Proposition~\ref{proposition:b-stratification} that if~$X$ is ordinary, the stratification of~$B$ 
is induced via the blowup~$B \to \P(V_6^\vee)$ from the dual EPW stratification of~$\P(V_6^\vee)$ (see Section~\ref{subsec:ld}),
and if~$X$ is special, the stratification is the same away from the exceptional divisor~$E \subset B$
of the blowup~$B \to \P(V_6^\vee)$, while on the divisor $E$, it is shifted by~$1$.\ 
This allows us to relate  the Stein factorization of the map~$f \colon \OGr_B(k+2,\cQ) \to B$
to the double dual EPW varieties (see Propositions~\ref{prop:corank-n-k} and~\ref{prop:corank-special}).

In particular, in the situation of Theorem~\ref{thm:g1x3-intro}, 
our arguments imply that~$\OGr_B(3,\cQ)$ is the blowup of the double dual EPW surface~$\wtY^{\ge 2}_\Ap$ at two points
and, in Theorem~\ref{thm:g2x5-intro}, that~$\OGr_B(4,\cQ)$ 
is an \'etale-locally trivial $\P^1$-fibration over~$\wtY^{\ge 2}_\Ap$.\ 
This allows us to complete the proofs of the theorems.

Throughout the paper, we work over the field~$\C$ of complex numbers,
though the same results hold true over any algebraically closed field of characteristic zero.\ 
As a general rule,~$V_k$,~$U_k$, or~$R_k$ denotes a vector space of dimension~$k$.

The organization of the paper is as follows: we begin by recalling in Section~\ref{sec:prelim} 
the necessary facts about the geometry of GM varieties, 
their Grassmannian hulls, Lagrangian data, and associated EPW varieties.\ 
In Section~\ref{sec:hilbert}, we discuss the Hilbert schemes~$\F_{k+1}(X)$ and~$\F_{k+1}(M_X)$ 
of linear spaces on a GM variety~$X$ and its Grassmannian hull~$M_X$
and introduce the Hilbert schemes~$\G_k(X)$.\ 
In Section~\ref{sec:qf}, we construct the vector bundle~$\cW$ on~$B$ 
and the quadratic fibration~$\cQ \subset \P_B(\cW)$ that play prominent roles afterwards.\ 
We also describe the corank stratification of~$B$.\ 
In Section~\ref{sec:og}, we discuss the geometry of the orthogonal Grassmannian~$\OGr_B(p, \cQ)$ and
describe the Stein factorization of its projection to~$B$ and the double coverings that arise in this way.\ 
In Section~\ref{sec:general-another}, we prove general results about the maps~$\lambda_k$ and~$g_k$ from diagram~\eqref{eq:diagram-intro},
and in Section~\ref{sec:explicit}, we combine the results obtained in the previous sections
to give explicit descriptions of the most interesting Hilbert schemes~$\G_k(X)$.
Finally, in Appendix~\ref{sec:encapsulated}, we discuss the local geometry of~$\G_k(X)$
around the locus of quadrics whose linear span is contained in~$M_X$ or in~$X$.


\section{Preliminaries}
\label{sec:prelim}

In this section, we recall some basic geometric facts about GM and EPW varieties.

\subsection{GM varieties}
\label{subsec:gm}

Let~$V_5$ be a vector space of dimension~$5$.\ 
We denote by
\begin{equation*}
\Gr(2,V_5) \subset \P(\bw2V_5)
\qquad\text{and}\qquad 
\CGr(2,V_5) \subset \P(\C \oplus \bw2V_5)
\end{equation*}
the Grassmannian of $2$-dimensional vector subspaces in~$V_5$ in its Pl\"ucker embedding
and the cone over~$\Gr(2,V_5)$.\ 
We denote by~$\bv   \in \P(\C \oplus \bw2V_5)$ the vertex of the cone~$\CGr(2,V_5)$;
it corresponds to the first summand in~$\C \oplus \bw2V_5$.\ 
Both~$\Gr(2,V_5)$ and~$\CGr(2,V_5)$ are the intersections of the $5$-dimensional space 
\begin{equation}
\label{eq:plucker}
V_5 \cong \bw4V_5^\vee \subset \Sym^2(\bw2V_5)^\vee \subset \Sym^2(\C \oplus \bw2V_5)^\vee
\end{equation}
of Pl\"ucker quadrics.

A {\sf Gushel--Mukai} ({\sf GM}) variety is defined as a dimensionally transverse intersection
\begin{equation*}
X = \CGr(2,V_5) \cap \PP(W) \cap Q,
\end{equation*}
where~$W \subset \C \oplus \bw2V_5$ is a vector subspace and~$Q \subset \PP(W)$ is a quadric.\ 
Its dimension is~$n=\dim(W)-5  \le 6$.\ 
In this paper, we only consider \emph{smooth} GM varieties; in particular, they do not contain the vertex of the cone.\ 
For a general introduction to GM varieties, see~\cite{DK1}.

When~$n \ge 3$, the variety $X$ is a Fano variety with Picard number~$1$
and, when~$n = 2$, it is a K3 surface with a polarization of degree~$10$; see~\cite[Theorem~2.16]{DK1}.

Forgetting the quadric~$Q$, we consider the intersection 
\begin{equation*}
M_X \coloneqq \CGr(2,V_5) \cap \P(W),
\end{equation*}
which is called {\sf the Grassmannian hull} of~$X$;
it is intrinsically defined by~$X$.\ 
It is equal to the intersection of the restrictions to~$\P(W)$ of the Pl\"ucker quadrics,
and its dimension is~\mbox{$\dim(X) + 1 = n + 1$}.\ 
Therefore, $X$ is the intersection of the $6$-dimensional space~$V_6$ of quadrics
generated by the space~$V_5$ defined in~\eqref{eq:plucker} and the  quadric~$Q$.\ 
We denote by
\begin{equation}
\label{eq:plucker-point}
\bp_X \in \PP(V_6^\vee)
\end{equation}
the point corresponding to the Pl\"ucker hyperplane $V_5 \subset V_6$ and call it {\sf the Pl\"ucker point}.\ 
The extra quadric~$Q$ corresponds to a point of~$\P(V_6) \ssm \P(V_5)$ and is not canonically associated with~$X$.

A GM variety~$X$ is called {\sf special} if the vertex~$\bv$ of the cone~$ \CGr(2,V_5)$ is contained in~$\P(W)$, 
and {\sf ordinary} otherwise;
every special GM variety of dimension~\mbox{$n \le 5$} 
can be deformed to an ordinary GM variety of the same dimension
(by simply deforming the subspace~$\P(W) \subset \P(\C \oplus \bw2V_5)$ to a subspace not containing the point~$\bv$).\ 
Thus, a general GM variety of dimension~$n \le 5$ is ordinary,
while any GM variety of dimension~$6$ is special.

When~$X$ is smooth and special of dimension~$n \ge 3$, 
 we define the hyperplane~$W_0 \subset W$ as the orthogonal complement of the line~$\C \subset W$ 
(corresponding to the point~$\bv \in \P(W)$) with respect to the quadratic form~$Q$,
 so that
\begin{equation}
\label{eq:w-w0}
W = \C \oplus W_0.
\end{equation}
This decomposition does not depend on the choice of~$Q$ (see~\cite[Proposition~2.30]{DK1}).\ 
If~$X$ is ordinary, we set~$W_0 \coloneqq W$.\ 
In either case, the projection~$\C \oplus \bw2V_5 \to \bw2V_5$ induces an embedding~$W_0 \subset \bw2V_5$
and, if~$X$ is ordinary, it induces on~$X\subset \P(W_0)$ a closed embedding~\mbox{$X \hookrightarrow \Gr(2,V_5)$} 
such that~$X = \Gr(2,V_5) \cap \P(W_0) \cap Q$.

Furthermore, we denote by
\begin{equation}
\label{eq:wp-wp0}
W^\perp \subset (\C \oplus \bw2V_5)^\vee 
\qquad\text{and}\qquad
W_0^\perp \subset \bw2V_5^\vee
\end{equation}
the respective orthogonal complements of~$W$ and~$W_0$;
if~$X$ is special, then~$W_0^\perp = W^\perp$.

If~$X$ is smooth and special of dimension~$n \ge 3$, then  (see~\cite[Proposition~2.30]{DK1})
\begin{equation}
\label{eq:x-x0}
X_0 \coloneqq X \cap \P(W_0) 
\end{equation} 
is a smooth ordinary GM variety of dimension~$n - 1$, 
the projection from the vertex~$\bv$ of the cone~$\CGr(2,V_5)$ is a double covering
\begin{equation}
\label{defmpx}
X \lra M'_X \coloneqq M_{X_0} = \Gr(2,V_5) \cap \P(W_0)
\end{equation}
branched along~$X_0$, and~$M_X$ is the cone over~$M'_X$ with vertex~$\bv$ (see~\cite[Section~2.5]{DK1}).\ 
Note that the space~$W_0^\perp \subset \bw2V_5^\vee$ defined above 
is the space of equations of~$M_X $ in $ \Gr(2,V_5)$ when~$X$ is ordinary, 
or the space of equations of~$M'_X $ in $ \Gr(2,V_5)$ when~$X$ is special.

In the rest of this paper, we will consider GM varieties~$X$ that satisfy the following property 
\begin{equation}\label{hh}
\tag{\mbox{${\mathbb{S}}$}} 
\parbox{.7\textwidth}{either $X$ is a smooth GM variety and~$\dim(X) \ge 3$,\\
or~$X$ is a smooth ordinary GM surface and~$M_X$ is smooth.}
\end{equation}
As we will see in Theorem~\ref{thm:gm-ld}, this property is related to the \emph{strong smoothness} property 
from~\cite[Definition~3.15]{DK1}, hence the notation.

We will often use the following simple consequence of Property~\eqref{hh}.

\begin{lemm}
\label{lem:wperp}
If a GM variety~$X$ satisfies Property~\eqref{hh}, every nonzero skew form in the space~$W_0^\perp \subset \bw2V_5^\vee$ has rank~$4$.
\end{lemm}

\begin{proof}
As we mentioned above, $W_0^\perp$ is the space of equations of~$M_X \subset \Gr(2,V_5)$  if~$X$ is ordinary,
and of~$M'_X \subset \Gr(2,V_5)$ if~$X$ is special.\ 
If~$\dim(X) \ge 3$, the varieties~$M_X$ and~$M'_X$ are smooth by~\cite[Proposition~2.22]{DK1}, 
and if~$\dim(X) = 2$, the variety~$M_X$ is smooth by Property~\eqref{hh}.\ 
Therefore, all skew forms in~$W_0^\perp$ have rank~4 by~\cite[Proposition~2.24]{DK1}.
\end{proof}

Set~$n_0 \coloneqq n - 1$ if~$X$ is special, and~$n_0 \coloneqq n$ if~$X$ is ordinary, 
so that~$\dim(W_0^\perp) = 5 - n_0$.\ 
By Lemma~\ref{lem:wperp}, for every nonzero skew form~$\omega \in W_0^\perp$,
the wedge square~$\omega \wedge \omega$ is a nonzero element of~$\bw4V_5^ \vee \cong V_5$
generating the kernel of~$\omega$;
this implies that the map
\begin{equation}
\label{eq:kappa}
\upkappa \colon \PP(W_0^\perp) \lra \PP(V_5),
\qquad 
\omega \longmapsto \omega \wedge \omega
\end{equation}
is well defined.

\begin{coro}
\label{cor:kappa}
Let~$X$ be a GM variety of dimension~$n$ satisfying Property~\eqref{hh}.\ 
 The map~$\upkappa$ is injective and it is induced by a linear map
\begin{equation*}
\label{eq:tilde-kappa}
\tilde\upkappa \colon \Sym^2(W_0^\perp) \lra V_5
\end{equation*}
which is injective if~\mbox{$n_0 \ge 3$} and surjective if~\mbox{$n_0 = 2$}.\ 
Moreover, if~\mbox{$n_0 = 2$}, the kernel of~$\tilde\upkappa$ is spanned by a nondegenerate quadratic tensor.
\end{coro}

\begin{proof}
The map~$\upkappa$ is quadratic by definition~\eqref{eq:kappa}, hence it is induced by a linear map~$\tilde\upkappa$.

Assume $\upkappa(\omega_1) = \upkappa(\omega_2)$.\ 
The  nonzero skew forms~$\omega_1$ and~$\omega_2$ then have a common kernel vector.\ 
Consequently, an appropriate nontrivial linear combination of these forms is decomposable, 
hence must vanish by Lemma~\ref{lem:wperp}.\
This proves that~$\upkappa$ is injective.

Next, we assume~$n_0 \ge 3$ and prove that~$\tilde\upkappa$ is injective.\
If~$n_0 = 5$, the space~$\Sym^2(W_0^\perp)$ is zero, so there is nothing to prove.\ 
If~$n_0 = 4$, the space~$\Sym^2(W_0^\perp)$ is $1$-dimensional 
and the map is nonzero since~$\upkappa$ is well defined.\ 
If~$n_0 = 3$, the space~$\Sym^2(W_0^\perp)$ is $3$-dimensional 
and~$\upkappa$ is the composition of the second Veronese embedding~$\P^1 \to \P^2$ 
and a linear map~$\P^2 \dashrightarrow \P^4$ induced by~$\tilde\upkappa$.\ 
If the latter is noninjective, either~$\upkappa$ is not regular, 
or it factors through a double cover~$\P^1 \to \P^1$.\
In both cases, this contradicts the injectivity of~$\upkappa$.

Finally, consider the case~$n_0 = 2$.\ 
The space~$\Sym^2(W_0^\perp)$ is $6$-dimensional and it is enough to show 
that the kernel of~$\tilde\upkappa$ contains no tensors of rank~$1$ or~$2$.\ 
If~$\Ker( \tilde\upkappa)$ contains a tensor of rank~$1$,  
$\upkappa$ is not defined at the corresponding point of~$\P(W_0^\perp)$,
and if it contains a tensor of rank~$2$,  
the restriction of~$\upkappa$ to the corresponding~$\P^1 \subset \P(W_0^\perp)$ 
factors through a double covering~$\P^1 \to \P^1$, hence~$\upkappa$ is not injective.\ 
So, both cases are impossible.
\end{proof}

The dual space of~$\Sym^2(W_0^\perp)$ carries a natural family of quadratic forms.\ 
We denote by 
\begin{equation*}
\sD^{\ge k}_{W_0^\perp} \subset \Sym^2\!(W_0^\perp)^\vee
\qquad\text{and}\qquad 
\uvt_{\sD} \colon \tsD^{\ge k}_{W_0^\perp} \lra \sD^{\ge k}_{W_0^\perp} 
\end{equation*}
the corank $\ge k$ locus and its double covering constructed in~\cite[Theorem~3.1]{DK3}.\ 
When~$k = 0$, the covering $\uvt_{\sD}$ is the usual double covering branched along the affine hypersurface~$\sD^{\ge 1}_{W_0^\perp}$.\ 
When~$k = \dim(W_0^\perp) - 1$, it was described in~\cite[Lemma~3.4]{DK3}.\ 
In the next lemma, we describe the case where~$k = 1$ and~$\dim(W_0^\perp) = 3$  
(see~\cite[Lemma~4.1]{Repwcube} for a generalization).
To simplify notation, we replace~$W_0^\perp$ by an abstract 3-dimensional space~$V_3$.

\begin{lemm}
\label{lem:tsd}
The subscheme~$\sD_{V_3}^{\ge 1} \subset \Sym^2\!V_3^\vee$ is the affine symmetric determinantal cubic hypersurface 
and the corresponding double covering is given by the morphism
\begin{equation}
\label{defdc}
\uvt_{\sD} \colon \tsD_{V_3}^{\ge 1} \coloneqq 
\{ \mu \in V_3^\vee\otimes V_3^\vee \mid  \rank(\mu)\le 1 \}  \thlra 
\sD_{V_3}^{\ge 1} \subset  \Sym^2\!V_3^\vee,\qquad 
\mu\longmapsto \mu+\mu^T.
\end{equation}
The ramification locus of~$\uvt_\sD$ is the subvariety~$\{ \mu \mid \mu = \mu^T\}$
and the branch locus is~$\sD_{V_3}^{\ge 2}$.
\end{lemm}

\begin{proof}
The proof is analogous to the proof of~\cite[Lemma~3.4]{DK3}.\ 
First, we identify~\eqref{defdc} with the quotient for the $\Z/2$-action on~$\tsD_{V_3}^{\ge 1}$ 
by permutation of the factors of~\mbox{$V_3^\vee \otimes V_3^\vee$}.\
This induces a direct sum decomposition of the sheaf~$(\uvt_{\sD})_*\cO$ 
with summands corresponding to the invariant and antiivariant parts of the ring of regular functions~$\C[\tsD_{V_3}^{\ge 1}]$,
so that the sheaf of functions of~$\sD_{V_3}^{\ge 1}$ corresponds to the invariant part~$\C[\tsD_{V_3}^{\ge 1}]_+$
and the reflexive sheaf associated with the double cover~\eqref{defdc} corresponds to the antiinvariant part~$\C[\tsD_{V_3}^{\ge 1}]_-$.
Next, we check that the image of the map
\begin{equation*}
q\vert_{\sD_{V_3}^{\ge 1}} \colon V_3 \otimes \cO_{\sD_{V_3}^{\ge 1}} \to V_3^\vee \otimes \cO_{\sD_{V_3}^{\ge 1}}
\end{equation*}
given by the restriction of the universal family of quadrics over~$\Sym^2\!V_3^\vee$
corresponds to the module~$\C[\tsD_{V_3}^{\ge 1}] \cong \C[\tsD_{V_3}^{\ge 1}]_+ \oplus \C[\tsD_{V_3}^{\ge 1}]_-$
over~$\C[\sD_{V_3}^{\ge 1}] = \C[\tsD_{V_3}^{\ge 1}]_+$ with its natural algebra structure
and conclude from this that the double cover from~\cite[Theorem~3.1]{DK3}
coincides with~$\uvt_{\sD}$.
\end{proof}

\begin{rema}
\label{rem:uvt-psd}
The map~\eqref{defdc} being homogeneous, it induces a double covering 
\begin{equation}
\label{eq:ptsd}
\uvt_{\P\sD} \colon \P(\tsD_{V_3}^{\ge 1}) \lra  \P(\sD_{V_3}^{\ge 1}) \subset \P(\Sym^2\!V_3^\vee) 
\end{equation}
between the associated projective varieties, where~$\P(\tsD_{V_3}^{\ge 1}) \cong \P(V_3^\vee) \times \P(V_3^\vee)$ is   the Segre variety 
and~$\P(\sD_{V_3}^{\ge 1})$ is the symmetric determinantal cubic hypersurface in~$\P(\Sym^2\!V_3^\vee) = \P^5$.\ 
The ramification locus of~$\uvt_{\P\sD}$ is the diagonal~$\P(V_3^\vee) \subset \P(V_3^\vee) \times \P(V_3^\vee)$
and the branch locus is the Veronese surface~$\P(V_3^\vee) \subset \P(\Sym^2\!V_3^\vee)$.
\end{rema}

\subsection{Lagrangian data, EPW varieties and their double covers}
\label{subsec:ld}

In~\cite{DK1}, we associated with a GM variety~$X$ a subspace
\begin{equation*}
A \subset \bw3V_6
\end{equation*}
which is Lagrangian with respect to the skew-symmetric wedge product on~$\bw3V_6$
and showed that many properties of~$X$ are controlled by~$A$.\ 
In particular, by~\cite[Theorem~3.16 and Remark~3.17]{DK1}, 
a GM variety satisfies Property~\eqref{hh} if and only if~$A$ {\sf contains no decomposable vectors}, that is, if
\begin{equation}
\label{eq:no-decomposable}
\P(A) \cap \Gr(3,V_6) = \vide,
\end{equation}
where the intersection is taken in~$\P(\bw3V_6)$.\

The main classification result of~\cite{DK1} (which was promoted to a description of moduli stacks in~\cite{DK4}) 
shows that a GM variety~$X$ of a given type (ordinary or special) is determined by the pair~$(A,V_5)$
consisting of a Lagrangian subspace~$A \subset \bw3V_6$ and a hyperplane~$V_5 \subset V_6$.\ 
More precisely, we have the following.

\begin{theo}[{\cite[Section~3]{DK1}}]
\label{thm:gm-ld}
For each integer~\mbox{$\ell\in\{0,1,2,3\}$}, there are explicit bijections between the following three sets:
\begin{aenumerate}
\item
\label{it:lag}
the set of $\PGL(V_6)$-orbits of pairs~$(A,V_5)$ such that~$A$ contains no decomposable vectors and~$\dim(A \cap \bw3V_5) = \ell$;
\item 
\label{it:ord}
the set of isomorphism classes of smooth ordinary GM varieties~$X$ of dimension~$5 - \ell$ satisfying Property~\eqref{hh};
\item 
\label{it:spe}
the set of isomorphism classes of smooth special GM varieties~$X$ of dimension~$6 - \ell$.
\end{aenumerate}
The bijection from  set~\ref{it:spe} to  set~\ref{it:ord} is given by the operation~$X \mapsto X_0$ defined in~\eqref{eq:x-x0}.\ 
In particular, if~$X$ is special,~$X_0$ has the same Lagrangian data~$(A, V_5)$ as~$X$.
\end{theo}

Many geometric properties of a GM variety can be explained in terms of 
the {\sf Eisenbud--Popescu--Walter} ({\sf EPW}) schemes
\begin{equation}
\label{ya} 
\begin{array}{rlll}
\sY^{\ge \ell}_{A} 	&\coloneqq \{ U_1 \subset V_6 \mid \dim(A \cap (U_1 \wedge \bw2V_6)) \ge \ell \} 
&& \subset \PP(V_6),
\\
\sY^{\ge \ell}_{\Ap} 	&\coloneqq \{ U_5 \subset V_6 \mid \dim(A \cap \bw3U_5) \ge \ell \} 
&&\subset \PP(V_6^\vee) = \Gr(5,V_6),
\end{array}
\end{equation} 
associated with a Lagrangian subspace~$A \subset \bw3V_6$ 
and its orthogonal complement~\mbox{$\Ap \subset \bw3V_6^\vee$} (which is also Lagrangian).\
These schemes are known as {\sf EPW varieties} and {\sf dual EPW varieties}, respectively.\ 

When condition~\eqref{eq:no-decomposable} is satisfied, 
the geometry of EPW varieties was described by O'Grady (\cite[Theorem~B.2]{DK1}).

\begin{theo}[O'Grady]
\label{thm:og}
Let~$A \subset \bw3V_6$ be a Lagrangian subspace with no decomposable vectors, that is, assume~\eqref{eq:no-decomposable}.\ 
Then,
\begin{itemize}
\item 
$\sY_A^{\ge 1}$ and~$\sY_{A^\perp}^{\ge 1}$ are integral normal sextic hypersurfaces;
\item 
$\sY_A^{\ge 2} = \Sing(\sY_A^{\ge 1})$ and $\sY_{A^\perp}^{\ge 2} = \Sing(\sY_{A^\perp}^{\ge 1})$ 
are integral   normal surfaces of degree~$40$;
\item 
$\sY_A^{\ge 3} = \Sing(\sY_A^{\ge 2})$ and $\sY_{A^\perp}^{\ge 3} = \Sing(\sY_{A^\perp}^{\ge 2})$ are finite and smooth schemes, 
and are empty for general $A$;
\item $\sY_A^{\ge 4}$ and $\sY_{A^\perp}^{\ge 4}$ are empty.
\end{itemize}
\end{theo}

Similarly, there is a chain of subschemes (see~\cite[Proposition~2.6]{IKKR})
\begin{equation}
\label{eq:def-za}
\sZ^{\ge \ell}_A = \{ U_3 \subset V_6 \mid \dim(A \cap (\bw2U_3 \wedge {V_6})) \ge \ell \} \subset \Gr(3,V_6)
\end{equation} 
which have properties similar to those of the EPW schemes.

\begin{theo}[{\cite[Corollary~2.10]{IKKR} and~\cite[Theorem~2.1]{Repwcube}}]
\label{theorem:ikkr}
Let~$A \subset \bw3V_6$ be a Lagrangian subspace with no decomposable vectors.\ Then,
\begin{itemize}
\item
$\sZ_A^{\ge 1} $ is an integral normal quartic hypersurface in $\Gr(3,V_6)$;
\item
$\sZ_A^{\ge 2} = \Sing(\sZ_A^{\ge 1})$ is an integral normal Cohen--Macaulay sixfold of degree~$480$;
\item
$\sZ_A^{\ge 3} = \Sing(\sZ_A^{\ge 2})$ is an integral normal Cohen--Macaulay threefold of degree~$4944$;
\item
$\sZ_A^{\ge 4} = \Sing(\sZ_A^{\ge 3})$ is a finite and smooth scheme, and is empty for general~$A$;
\item
$\sZ_A^{\ge 5}$ is empty.
\end{itemize}
\end{theo}

We often consider the chains of EPW schemes as stratifications with  strata
\begin{equation*}
\sY_A^\ell \coloneqq \sY_A^{\ge \ell} \ssm \sY_A^{\ge \ell + 1},
\qquad
\sY_\Ap^\ell \coloneqq \sY_\Ap^{\ge \ell} \ssm \sY_\Ap^{\ge \ell + 1},
\qquad\text{and}\qquad 
\sZ_A^\ell \coloneqq \sZ_A^{\ge \ell} \ssm \sZ_A^{\ge \ell + 1}.
\end{equation*}
Comparing this with the definition~\eqref{eq:plucker-point} of the Pl\"ucker point and Theorem~\ref{thm:gm-ld}, we see that
\begin{equation}
\label{eq:px-dim}
\bp_X \in 
\begin{cases}
\sY_\Ap^{5 - \dim(X)} & \text{if~$X$ is ordinary,}\\
\sY_\Ap^{6 - \dim(X)} & \text{if~$X$ is special.}
\end{cases}
\end{equation}

Each of the schemes $\sY_A^{\ge \ell}$ and $\sY_{A^\perp}^{\ge \ell}$ comes with a natural double covering.

\begin{theo}[{\cite{og4}, \cite[Theorem~5.2 and Remark~5.3]{DK3}}]
\label{thm:tya}
Let $A\subset \bw3V_6$ be a Lagrangian subspace with no decomposable vectors.\ 
For each~\mbox{$\ell \in \{0,1,2,3\}$}, there are canonical double coverings
 \begin{equation*}
\uvt_A \colon \wtY_A^{\ge \ell} \lra \sY_A^{\ge \ell},
\qquad 
\uvt_\Ap \colon \wtY_{A^\perp}^{\ge \ell} \lra \sY_{A^\perp}^{\ge \ell},
\end{equation*}
respectively branched along $\sY_A^{\ge \ell+1}$ and $\sY_{A^\perp}^{\ge \ell + 1}$.\ 
Moreover, these double covers restrict to isomorphisms
\begin{equation*}
\Sing(\wtY^{\ge 0}_A) \isomto \sY^{\ge 2}_A,
\qquad 
\Sing(\wtY^{\ge 1}_A)  \isomto \sY^{\ge 3}_A = \sY^3_A,
\qquad
\Sing(\wtY^{\ge 2}_A)  \isomto \sY^{\ge 3}_A =  \sY_A^3,
\end{equation*}
and analogously for $\wtY^{\ge \ell}_\Ap$.\ 
Finally,~$\wtY^{\ge \ell}_A$ is integral and normal for all~$\ell\le 2$.
\end{theo}

These double coverings are known as {\sf double EPW varieties} and {\sf double dual EPW varieties}.\ 
There are similar double coverings for the varieties~$\sZ^{\ge \ell}_A$, but we will not need them.

We write~$\wtY^\ell_A   \subset \wtY^{\ge \ell}_A $ and~$\wtY_\Ap^\ell  \subset \wtY^{\ge \ell}_\Ap $ 
for the preimages of~$\sY_A^\ell$ and~$\sY_\Ap^\ell$;
thus the morphisms~\mbox{$\uvt_A \colon \wtY^\ell_A \to \sY_A^\ell$} and~\mbox{$\uvt_\Ap \colon \wtY_\Ap^\ell \to \sY_\Ap^\ell$} 
are \'etale double coverings and~$\wtY^\ell_A$ and~$\wtY^\ell_\Ap$ are smooth and connected varieties if~$\ell \le 2$.

The next proposition describes the singularity of~$\wtY^{\ge \ell}_A$ (or~$\wtY^{\ge \ell}_\Ap$)
at a point of $\sY^3_A$ (or~$\sY^3_\Ap$).\ 
We use the notation introduced in Lemma~\ref{lem:tsd} and before it.

\begin{prop}
\label{propsing}
Let $A\subset \bw3V_6$ be a Lagrangian subspace with no decomposable vectors and let $\bp \in \sY^3_A$.\ 

\begin{aenumerate}
\item
\label{it:sing-y2}
The singularity of~$\sY^{\ge 2}_A$ at~$\bp$ is a quotient singularity of type~$\tfrac14(1,1)$, and 
the singularity of~$\wtY^{\ge 2}_A$ at the point over~$\bp$ is an ordinary double point.
\item
\label{it:sing-y1}
Locally analytically around the point~$\bp$, 
the double cover $\uvt_A \colon \wtY^{\ge 1}_A\to \sY^{\ge 1}_A$ 
is isomorphic to the restriction of the double cover~\eqref{defdc} to a general hyperplane in $\Sym^2\!V_3$.\ 
In particular, the singularity of~$\wtY^{\ge 1}_A$ at the point over~$\bp$  
is the cone over the flag variety~$\Fl(1,2;3) \subset \P^2\times \P^2$.
\end{aenumerate}
\end{prop}

\begin{proof}
Item~\ref{it:sing-y2} is explained in~\cite[Theorem~5.2(2) and its proof]{DK3}; 
see also \cite[Proposition~2.9]{og3} and \cite[Example~1.1]{og4}.\

For~\ref{it:sing-y1}, write $\bp=[V_1]$ 
and set~$V_3\coloneqq \bigl(A\cap (V_1\wedge\bw2V_6)\bigr)^\vee$ 
(by~\eqref{ya}, this is a vector space of dimension~$3$).\  
By~\cite[Proposition~2.5]{og4}, there is an injective map
\begin{equation*}
\uptau\colon V_6/V_1\lhra \Sym^2\!V_3
\end{equation*}
whose image is a general hyperplane (that is, it corresponds to a point of~$\P(\Sym^2\!V_3^\vee) $ of maximal rank).\ 
By~\cite[(2.17)]{og3} and the proof of~\cite[Proposition~3.10]{og4} (see also~\cite[Remark~1.4]{og6}), 
locally analytically around~$\bp$, the double cover~$\uvt_A \colon \wtY^{\ge 1}_A\to \sY^{\ge 1}_A$
is given by the restriction to this hyperplane of the double cover~\eqref{defdc}.\ 
It remains to note that~$\tsD_{V_3}^{\ge 1} \subset V_3^\vee \otimes V_3^\vee$ 
is the cone over the Segre variety~$\P(V_3) \times \P(V_3)$ (cf.\ Remark~\ref{rem:uvt-psd}),
and its intersection with the preimage of a general hyperplane is the cone over the flag variety.
\end{proof}

To end this section, we introduce some notation that will be important in Section~\ref{sec:og}.
 
\begin{lemm}
\label{lem:closure}
 Let~$A\subset \bw3V_6$ be a Lagrangian subspace with no decomposable vectors.\ 
For each~$\ell \in \{0,1,2\}$ and~$\bp \in \sY_\Ap^{\ge \ell}$, 
we define a subscheme~$\widetilde\bp \subset \wtY_\Ap^{\ge \ell}$ over~$\bp$ as follows:
\begin{aenumerate}
\item 
\label{it:tp-red}
if~$\ell \in \{1,2\}$ and~$\bp \in \sY^3_\Ap$ 
or if~$\ell \in \{0,1,2\}$ and~$\bp \in \sY^\ell_\Ap$,
we set~$\widetilde\bp \coloneqq \uvt_\Ap^{-1}(\bp)_{\mathrm{red}}$;
\item 
\label{it:tp-full}
if~$\ell = 0$ and~$\bp \in \sY^{\ge 1}_\Ap$ or~$\ell = 1$ and~$\bp \in \sY^2_\Ap$,
we set~$\widetilde\bp \coloneqq \uvt_\Ap^{-1}(\bp)$.
\end{aenumerate}
Then the scheme~$\Bl_{\widetilde\bp}(\wtY^{\ge \ell}_\Ap)$ is integral and normal
and there is a commutative square
\begin{equation*}
\xymatrix@C=3em{
\Bl_{\widetilde\bp}(\wtY^{\ge \ell}_\Ap) \ar[r] \ar[d]_{\tuvt_\Ap} &
\wtY^{\ge \ell}_\Ap \ar[d]^{\uvt_\Ap} 
\\
\Bl_{\bp}(\sY^{\ge \ell}_\Ap) \ar[r] &
\sY^{\ge \ell}_\Ap, 
}
\end{equation*}
where the horizontal arrows are the blowup morphisms and~$\tuvt_\Ap$ is a double covering.\ 
In particular, the scheme~$\Bl_{\widetilde\bp}(\wtY^{\ge \ell}_\Ap)$ is the integral closure of~$\Bl_{\bp}(\sY^{\ge \ell}_\Ap)$
in the field of rational functions of~$\wtY^{\ge \ell}_\Ap$.
\end{lemm}

\begin{proof}
To prove the existence of a morphism~$\tuvt_\Ap$ making the diagram commutative,
it is enough to check that the scheme preimage of the point~$\bp$ 
under the composition
\begin{equation*}
\Bl_{\widetilde\bp}(\wtY^{\ge \ell}_\Ap) \xrightarrow{\quad} 
\wtY^{\ge \ell}_\Ap \xrightarrow{\ \uvt_\Ap\ } 
\sY^{\ge \ell}_\Ap
\end{equation*}
is a Cartier divisor.\ 
This is clear in case~\ref{it:tp-full},
and also in the case where~$\bp \in \sY^\ell_\Ap$, because~$\uvt_\Ap$ is  then \'etale over~$\bp$,
hence the scheme~$\uvt^{-1}_\Ap(\bp)$ is reduced.\ 
Finally, if~$\ell \in \{1,2\}$ and~$\bp \in \sY^3_\Ap$, using the description of Proposition~\ref{propsing},
it is easy to see that~$\uvt^{-1}_\Ap(\bp)$ is the first infinitesimal neighborhood of the point~$\bp'$ over~$\bp$,
hence the blowup of~$\bp' = \uvt^{-1}_\Ap(\bp)_{\mathrm{red}}$ is isomorphic to the blowup of~$\uvt^{-1}_\Ap(\bp)$.

It remains to check that the scheme~$\Bl_{\widetilde\bp}(\wtY^{\ge \ell}_\Ap)$ is integral and normal.

If~$\ell \in \{0,1,2\}$ and~$\bp \in \sY^\ell_\Ap$, this is obvious 
because~$\widetilde\bp$ is a pair of nonsingular points on the integral and normal scheme~$\tY^{\ge \ell}_\Ap$ (see Theorem~\ref{thm:tya}).

If~$\ell = 0$ and~$\bp \in \sY^{\ge 1}_\Ap$, then~$\Bl_{\widetilde\bp}(\wtY^{\ge \ell}_\Ap)$ is the strict transform 
of a normal sextic hypersurface in~$\P(1^6,3)$
with respect to the blowup of a line~$\P(1,3) \subset \P(1^6,3)$,
hence it is integral and Cohen--Macaulay, and it is easily seen to be nonsingular in codimension~$1$.
 
If~$\ell = 1$ and~$\bp \in \sY^2_\Ap$, then~$\wtY^{\ge \ell}_\Ap$ is nonsingular over~$\bp$ 
and~$\widetilde\bp$ is the first infinitesimal neighborhood of a point in the direction normal to a smooth surface,
so the integrality and normality of its blowup is easily verified.

Finally, if~$\ell \in \{1,2\}$ and~$\bp \in \sY^3_\Ap$, 
by Proposition~\ref{propsing}, 
the scheme~$\wtY_\Ap^{\ge \ell}$ is, locally analytically around the point~$\bp'$ over~$\bp$,  
isomorphic to a cone over a smooth conic or over the flag variety~$\Fl(1,2;3)$.\ 
Therefore, its blowup is nonsingular over~$\bp$, hence normal.
\end{proof}


\section{Hilbert schemes}
\label{sec:hilbert}

In this section, we discuss the Hilbert schemes of linear spaces 
on linear sections of~$\CGr(2,V_5)$ (Section~\ref{subsec:ls-qu}) and GM varieties (Section~\ref{ss:fkx}),
and we introduce Hilbert schemes of quadrics on GM varieties (Section~\ref{sec:gkx}).

\subsection{Linear spaces on linear sections of~$\CGr(2,V_5)$}
\label{subsec:ls-qu}

 Given a projective scheme~$Z$ and an ample divisor class~$H $ on~$Z$,
we denote by~$\F_k(Z)$ the Hilbert scheme of linear spaces of dimension~$k$ in~$Z$, 
that is, of subschemes with Hilbert polynomial
\begin{equation*}
h^\F_k(t) \coloneqq \frac{(t+1)\cdots(t+k)}{k!}
\end{equation*}
with respect to the ample class~$H$.\

A closed subscheme~$\Pi \subset \P(W)$ with Hilbert polynomial equal to~$h^\F_k(t)$
(with respect to the hyperplane class of~$\P(W)$)
is a linear subspace~$\P^k \subset \P(W)$.\ 
In other words,
\begin{equation*}
\F_k(\P(W)) \cong \Gr(k+1,W).
\end{equation*}
Similarly, if~$Z \subset \P(W)$, then~$\F_k(Z)$ is the subscheme of~$\Gr(k+1,W)$ 
parameterizing linear subspaces contained in~$Z$.\ 
In particular, $\F_0(Z) = Z$, the scheme~$\F_1(Z)$ is the Hilbert scheme of lines on~$Z$, and so on.\

In this section, we discuss the Hilbert schemes of linear spaces on~$\CGr(2,V_5)$ and on its linear sections.\ 
We first introduce some notation.\

\begin{defi}
\label{def:st-subspaces}
A linear subspace~$\Pi \subset \P(\C \oplus \bw2V_5)$ is a {\sf $\sigma$-space} 
if there is a $1$-dimensional subspace~$V_1 \subset V_5$ such that
\begin{align*}
\Pi \subset \P(\C \oplus (V_1 \wedge V_5)) &= \P^4 \subset \CGr(2,V_5),\\
\intertext{and it is a {\sf ${\rtx}$-space} if there is a $3$-dimensional subspace~$V_3 \subset V_5$ such that}
\Pi \subset \P(\C \oplus \bw2V_3)  &= \P^3\subset \CGr(2,V_5).
\end{align*}
Finally, $\Pi$ is a {\sf $\sigma\rtx$-space} if it is both a $\sigma$-space and a $\rtx$-space.
\end{defi}

Note that~$(V'_1 \wedge V_5) \cap (V''_1 \wedge V_5) = V'_1 \wedge V''_1$ for~$V'_1 \ne V''_1$
and~$\bw2V'_3 \cap \bw2V''_3 = \bw2(V'_3 \cap V''_3)$ for~$V'_3 \ne V''_3$, 
hence the line~$V_1$ and the subspace~$V_3$ in Definition~\ref{def:st-subspaces} are uniquely determined
(except for a point or line in the ruling of the cone~$\CGr(2,V_5)$).

Similarly, $(V_1\wedge V_5) \cap \bw2V_3 = V_1 \wedge V_3$ if~$V_1 \subset V_3$, and  $0$ otherwise, 
hence~$\Pi$ is a $\sigma\rtx$-space if and only if it is contained in the cone over a line in~$\Gr(2,V_5)$.

It is a classical fact (see for example \cite[Theorem~4.9]{LM}) that any linear subspace in the Grassmannian~$\Gr(2,V_5) \subset \P(\bw2V_5)$, 
and hence also in~$\CGr(2,V_5) \subset \P(\C \oplus \bw2V_5)$,  is a $\sigma$-space or a $\rtx$-space.\ 
Moreover, the conditions to be a $\sigma$-space, a $\rtx$-space, or a~$\sigma\rtx$-space in~$\CGr(2,V_5)$ are closed, so that 
\begin{equation*}
\F_k(\CGr(2,V_5)) = \Fs_k(\CGr(2,V_5)) \cup \Ft_k(\CGr(2,V_5)),
\end{equation*}
where~$\Fs_k(\CGr(2,V_5))$ and~$\Ft_k(\CGr(2,V_5))$ are the closed subschemes of $\sigma$-spaces and~$\tau$-spaces, respectively.\ 
Similarly, for a closed subscheme~$Z \subset \CGr(2,V_5)$, we write
\begin{equation*}
\F^\star_k(Z) = \F_k(Z) \cap \F^\star_k(\CGr(2,V_5)),
\end{equation*}
where~$\star \in \{\sigma,\tau,\sigma\tau\}$,
so that~$\F_k(Z) = \Fs_k(Z) \cup \Ft_k(Z)$ and~$\Fst_k(Z) = \Fs_k(Z) \cap \Ft_k(Z)$.

Consider a smooth transverse linear section~$M = \Gr(2,V_5) \cap \P(W_0)$ of~$\Gr(2,V_5)$ of dimension at least~$3$,
where~$W_0 \subset \bw2V_5$ is a vector subspace.\ 
We let~$W_0^\perp \subset \bw2V_5^\vee$ be its orthogonal complement;
by~\cite[Proposition~2.24]{DK1}, any nonzero skew form in~$W_0^\perp$ has rank~4,  
so the maps~$\upkappa$ and~$\tilde\upkappa$ from~\eqref{eq:kappa} and Corollary~\ref{cor:kappa} are defined.\ 
To describe~$\F_k(M)$, we introduce some notation:
\begin{itemize}
\item 
If~$\dim(M) = 5$, so that~$\dim(W_0^\perp) = 1$, 
we denote by~$ V^M_1  \coloneqq \tilde\upkappa(\Sym^2(W_0^\perp)) \subset V_5$ 
the line corresponding to the point~$\upkappa([W_0^\perp]) \in \P(V_5)$.\ 
Any skew form~$\omega$ generating~$W_0^\perp$ induces a symplectic form on~$V_5/V^M_1$ and 
we denote by
\begin{equation*}
\IGr(2,V_5/V^M_1) \cong \IGr(3,V_5) \subset \Gr(3,V_5)
\end{equation*}
the Grassmannian of $\omega$-isotropic $3$-dimensional subspaces in~$V_5$
(it is a smooth $3$-dimensional quadric).\ 
\item 
If~$\dim(M) = 4$, so that~$\dim(W_0^\perp) = 2$, we denote by~$V^M_3 \coloneqq \tilde\upkappa(\Sym^2(W_0^\perp))\subset V_5$ 
the linear span of the smooth conic~$\P^1 \cong \upkappa(\P(W_0^\perp)) \subset \P(V_5)$.
\end{itemize}

\begin{lemm}
\label{lem:fk-mx}
Let~$M = \Gr(2,V_5) \cap \P(W_0^\perp)$ be a smooth transverse linear section  of dimension~$n+1\ge3$.\ 
For~$k \ge 4$, all schemes~$\F^\star_k(M)$ are empty; 
for~$k \le 3$, they are smooth and connected, and are listed in the following table:
\begin{equation*}
\arraycolsep=0.8em
\setlength{\extrarowheight}{4pt}
\begin{array}{c|cccc}
n & 2 & 3 & 4 & 5
\\[2pt]
\hline
 \F_1(M)  & \upkappa(\P(W_0^\perp)) \cong \P^2 &  \Bl_{\upkappa(\P(W_0^\perp))}(\P(V_5)) & \Bl_{\IGr(3,V_5)}(\Gr(3,V_5)) &  \Fl(1,3; V_5) 
\\
\Fs_2(M) & \vide  & \upkappa(\P(W_0^\perp)) \cong \P^1 &  \Bl_{\P(V^M_1)}(\P(V_5))&   \Fl(1,4; V_5) 
\\
\Ft_2(M) &  \vide &\P^0   & \IGr(3,V_5) & \Gr(3,V_5) 
\\
\F_3(M) & \vide &\vide &  \P^0 & \P(V_5)
\end{array}
\end{equation*}
When~$n=4$, the point~$\F_3(M)$ corresponds to the subspace~$\P(V^M_1 \wedge V_5) \subset \Gr(2,V_5)$, and
when~\mbox{$n=3$}, the point~$\Ft_2(M)$ corresponds to the plane~$\Gr(2,V^M_3) \subset \Gr(2,V_5)$.
\end{lemm}

\begin{proof}
The case~$n=5$ is classical (see~\cite[Section~4.1]{DK2}).\ 
The other cases follow easily: see~\cite[Lemma~4.2]{K12} for the case~$n=2$ 
and~\cite[Section~3]{dim} for the case~$n=3$.
\end{proof}

The dimensions of the   Hilbert schemes in the table
are listed in Corollary~\ref{cor:dim-fkmx} as the entries of the columns~$X_n^\ord $.\ 

For~$k \ge 3$, we have~$\Ft_k(M) = \vide$, while~$\Ft_1(M) = \Fs_1(M) = \F_1(M)$;
this is why we specify~$\Fs_k(M)$ and~$\Ft_k(M)$ only for~$k = 2$.\ 
The schemes~$\Fs_2(M) $ and~$\Ft_2(M)$ are disjoint.

Assume now that~$\P(W) \subset \P(\C \oplus \bw2V_5)$ contains the vertex~$\bv$ of the cone~$\CGr(2,V_5)$.\ 
Then~$M = \CGr(2,V_5) \cap \P(W)$ is the cone with vertex~$\bv$ 
over a linear section~$M' \subset \Gr(2,V_5)$
and for any linear subspace~$\Pi' \subset M'$ of dimension~$k - 1$,
the linear span~$ \langle \Pi', \bv \rangle$ is a linear subspace in~$M$ of dimension~$k$.\ 
This defines a closed embedding~$\F_{k-1}(M') \subset \F_k(M)$.

\begin{defi}
\label{def:fok}
If~$M = \CGr(2,V_5) \cap \P(W)$, we define 
\begin{equation*}
\Fo_k(M) \coloneqq 
\begin{cases}
\F_k(M) \ssm \F_{k-1}(M')  & \text{if~$\bv \in M$,}\\
\F_k(M)  & \text{otherwise.}
\end{cases}
\end{equation*}
This is the open subscheme in~$\F_k(M)$ parameterizing~$k$-planes in~$M$ not passing through~$\bv$.
\end{defi}

The following lemma is obvious.

\begin{lemm}
\label{lem:fk-cm}
If~$M = \CGr(2,V_5) \cap \P(W)$ is a cone of dimension~$n+1$ with vertex~$\bv$ 
over a smooth transverse linear section~$M' = \Gr(2,V_5) \cap \P(W_0)$ of dimension~$n$,
the projection from the vertex~$\bv$ induces morphisms
\begin{equation*}
\Fo_k(M) \lra \F_k(M')
\qquad\text{and}\qquad 
\F_k(M) \ssm \Fo_k(M) \isomto \F_{k-1}(M'),
\end{equation*}
where the first morphism is a Zariski-locally trivial~$\A^{k+1}$-bundle and the second morphism is an isomorphism.\ 
In particular, the schemes~$\Fo_k^\star(M)$ are smooth and connected for all~$k$.

Moreover, for~$n \ge 3$, we have the following stratifications of~$\F_k(M)$:
\begin{itemize}
\item 
$\F_1(M) = \Fo_1(M) \sqcup M'$ and the scheme~$\F_1(M)$ is irreducible of dimension~$2n-2$.
\item 
$\F_2(M) = \Fo^\sigma_2(M) \sqcup \F_1(M') \sqcup  \Fo^\rtx_2(M)$;
moreover, we have~$\Fs_2(M) = \Fo^\sigma_2(M) \sqcup  \F_1(M')$ and~\mbox{$\Ft_2(M) = \F_1(M') \sqcup  \Fo^\rtx_2(M)$}, 
and the scheme~$\F_2(M)$ is connected.
\item 
$\F_3(M)  = (\Fo_3(M)  \sqcup  \Fs_2(M')) \sqcup \Ft_2(M')$,
where~$\Fs_3(M) =  \Fo_3(M) \sqcup \Fs_2(M')$ and~$\Ft_2(M')$ are the two connected components of~$\F_3(M)$.
\item 
$\F_4(M)  = \F_3(M')$.
\end{itemize}
\end{lemm}

\begin{rema}
\label{rem:f3m}
Let~$M$ be as in Lemma~\ref{lem:fk-cm}.

When~$n=3$, the schemes~$\F_2(M)=\Fs_2(M) =\Ft_2(M) =\Fst_2(M) \isom \F_1(M')$ are isomorphic to~$\P^2$,
while the schemes~$\F_3(M)$ and~$\F_4(M)$ are empty.
 
When~$n=4$, the scheme~$\F_2(M)$ has three irreducible components: 
$\F_1(M')$ and the closure of~$\Fo^\sigma_2(M)$, both smooth of dimension~$4$,
and~$\F^\tau_2(M)$, smooth of dimension~3,
while the scheme~\mbox{$\F_3(M) \cong \F_2(M')$} has two connected components, 
$\Fs_2(M') \cong \P^1$ and~$\Ft_2(M') \cong \P^0$,
and~$\F_4(M)$ is empty.

When~$n=5$, the scheme~$\F_2(M)$ has two irreducible components: 
$\Fs_2(M)$, smooth of dimension~$7$, and~$\Ft_2(M)$, smooth of dimension~$6$.\ 
The scheme~$\Fs_3(M)$ has two  irreducible components:  
the closure of~$\Fo_3(M)$, which is the projective space~$\P^4$ of hyperplanes in the unique $4$-space on~$M$,
and~$\Fs_2(M') \cong \Bl_{[V_1^M]}(\P(V_5))$;
these smooth $4$-dimensional components intersect along the exceptional divisor of the second component 
(which lies as a hyperplane in the first).\ 
Furthermore, $\Ft_3(M) \cong \Ft_2(M') \cong \IGr(3, V_5)$ and~$\F_4(M) \cong \F_3(M') \cong \P^0$.

When~$n=6$, the scheme~$\F_2(M)$ has two irreducible components: 
$\Fs_2(M)$, smooth of dimension~$10$, and~$\Ft_2(M)$, smooth of dimension~9.\  
The scheme~$\Fs_3(M)$ is smooth and irreducible of dimension~$8$:
it has the structure of a $\P^4$-bundle over~$\Fs_3(M') \cong \P(V_5)$,
and the subscheme~\mbox{$\Fs_2(M') \cong \Fl(1,4;V_5)$} sits in it as a relative hyperplane.\ 
Furthermore, there is an isomorphism~$\Ft_3(M) \cong \Ft_2(M') \cong \Gr(3,V_5)$
and, analogously,~$\F_4(M) \cong \F_3(M') \cong \P(V_5)$.
\end{rema}

\begin{coro}
\label{cor:dim-fkmx}
The dimensions of the Hilbert schemes of linear spaces on the Grassmannian hull~$M_X$ 
of a  GM variety~$X_n$ of dimension~$n$ satisfying Property~\eqref{hh} are the following
\begin{equation*}
\setlength{\extrarowheight}{4pt}
\begin{array}{c|cccccccc}
& X_2^\ord & X_3^\spe & X_3^\ord & X_4^\spe & X_4^\ord & X_5^\spe & X_5^\ord & X_6^\spe
\\[2pt] 
\hline 
\dim(\F_1(M_X) )& 2 & 4 & 4 & 6 & 6 & 8 & 8 & 10
\\
\dim(\Fs_2(M_X) )& \cellcolor{lightgray} & 2 & 1 & 4 & 4 & 7 & 7 & 10 
\\
\dim(\Ft_2(M_X)) & \cellcolor{lightgray} & 2 & 0 & 3 & 3 & 6 & 6 & 9 
\\
\dim(\F_3(M_X) )& \cellcolor{lightgray} & \cellcolor{lightgray} & \cellcolor{lightgray} & 1 & 0 & 4 & 4 & 8 
\\
\dim(\F_4(M_X) )& \cellcolor{lightgray} & \cellcolor{lightgray} & \cellcolor{lightgray} & \cellcolor{lightgray} & \cellcolor{lightgray} & \cellcolor{lightgray} & \cellcolor{lightgray} & 4 
\end{array}
\end{equation*}
where grey cells correspond to the cases where the Hilbert schemes are empty.\ 
  In particular, 
\begin{itemize}
\item  $\dim(\F_1(M_X)) = 2n - 2$ for~$n \ge 2$, 
\item $\dim(\Fs_2(M_X)) = 3n - 8$ for~$n \ge 4$, 
\item  $\dim(\Ft_2(M_X)) = 3n - 9$ for~$n \ge 4$.
\end{itemize}
 \end{coro}

\subsection{Linear spaces on GM varieties}
\label{ss:fkx}

We discuss in this section some results about the Hilbert schemes~$\F_k(X)$ for a GM variety~$X$
(for more details and other cases, see~\cite[Section~4]{DK2}).\ 

First of all, note that~$\F_k(X) \subset \Fo_k(M_{X})$, because~$\bv \notin X$; 
moreover, the schemes~$\Fs_k(X)$ and~$\Ft_k(X)$ are disjoint for all~$k\ge 2$ when~$X$ is ordinary, 
and for all~$k\ge 3$ when~$X$ is special (see Section~\ref{subsec:ls-qu} for the notation).

Next, we describe the schemes~$\F_k(X)$ when~$k\ge n/2$.\   
Our description uses the subschemes~$\sY^3_A \subset \P(V_6)$ and~$\sZ^4_A \subset \Gr(3,V_6)$ 
defined in~\eqref{ya} and~\eqref{eq:def-za}.\ 
When~$A$ contains no decomposable vectors, they are finite reduced schemes, 
empty for general~$A$ (see Theorems~\ref{thm:og} and~\ref{theorem:ikkr}).\  
For each hyperplane~$V_5 \subset V_6$, we write
\begin{equation}
\label{eq:def-yav-zav}
\sY^3_{A,V_5} \coloneqq \sY^3_A \cap \P(V_5)
\qquad\text{and}\qquad 
\sZ^4_{A,V_5} \coloneqq \sZ^4_A \cap \Gr(3,V_5).
\end{equation}
They are also finite reduced schemes.

\begin{prop}
\label{prop:fk-gm}
Let~$X$ be a GM variety of dimension~$n$ satisfying Property~\eqref{hh}.\ 

If~$2k > n$, then~$\F_k(X) = \vide$.\ 

If~$2k = n$, the scheme~$\F_k(X)$ is finite and reduced
and admits a finite surjective morphism onto one of the following finite reduced schemes:
\begin{equation*}
\begin{cases}
\sY^3_{A,V_5}  & \text{if~$k = 1$  \textup(hence~$n = 2$\textup),}\\
\sY^3_{A,V_5} \sqcup \sZ^4_{A,V_5} & \text{if~$k = 2$ \textup(hence~$n = 4$\textup),}\\
\sY^3_{A,V_5}  & \text{if~$k = 3$ \textup(hence~$n = 6$\textup).}
\end{cases}
\end{equation*}
Moreover, if~$k = 1$, the morphism~$\F_1(X) \to \sY^3_{A,V_5}$ is an isomorphism.
\end{prop}

\begin{proof}
The emptiness of~$\F_k(X)$ in the case where~$2k > n$ follows from~\cite[Corollary~3.5]{DK2};
the case where~$2k = n$ is discussed in~\cite[Section~4]{DK2}.\ 
For the reader's convenience, we provide some details.\ 

Let~$\cU_X$ denote (the pullback to~$X$ of) the tautological subbundle of rank~2 on~$\Gr(2,V_5)$.\ 
By~\cite[Proposition~4.1]{DK2}, there is an isomorphism
\begin{equation*}
\Fs_k(X) \cong \F_k(\P_X(\cU_X)/\P(V_5))
\end{equation*}
with the relative Hilbert scheme of $\P^k$ in the fibers 
of the so-called ``first quadric fibration''~$\rho_1 \colon \P_X(\cU_X) \to \P(V_5)$ (see~\cite[Section~4.2]{DK1}).

If~$k = 1$, then~$n = 2$, so that~$X$ is an ordinary GM surface with a smooth Grassmannian hull~$M_X$.\ If~$\cU_{M_X}$ is the restriction to~$X$ of the tautological subbundle of rank~2 on~$\Gr(2,V_5)$,
it is classically known that the morphism~$\P_{M_X}(\cU_{M_X}) \to \P(V_5)$ 
is the blowup of the Veronese surface~\mbox{$\Sigma_1(X) \coloneqq \upkappa(\P(W_0^\perp)) \subset \P(V_5)$}, hence
\begin{equation*}
\F_1(\P_X(\cU_X)/\P(V_5)) \subset
\F_1(\P_{M_X}(\cU_{M_X})/\P(V_5)) = \Sigma_1(X).
\end{equation*}
Moreover, over~$\Sigma_1(X)$, the first quadric fibration~$\rho_1$ 
sits in a $\P^1$-bundle   as a divisor of relative degree~2 (see~\cite[Proposition~4.5]{DK1}),
hence~$\F_1(\P_X(\cU_X)/\P(V_5))$ is the second degeneracy locus of~$\rho_1$, 
which is equal to~$\sY^3_{A,V_5}$ as a scheme.

If~$k = 3$, then~$n = 6$, so that~$X$ is a special GM sixfold.\ 
The first quadric fibration~$\rho_1 \colon \P_X(\cU_X) \to \P(V_5)$ is flat of relative dimension~3 (see~\cite[Proposition~4.5]{DK1}),
hence~$\F_3(\P_X(\cU_X)/\P(V_5))$ is an \'etale double cover of the third degeneracy locus, 
which is equal to~$\sY^3_{A,V_5}$, again as a scheme.

Finally, let~$k = 2$, hence~$n = 4$, so that~$X$ is a GM fourfold (ordinary or special).\ 
Consider the scheme~$\Fs_2(X)$.\
By~\cite[Proposition~4.5]{DK1}, over the complement of~$\Sigma_1(X)$, the map~$\rho_1  \colon \P_X(\cU_X) \to \P(V_5)$ is a conic bundle, 
hence over this complement~$\Fs_2(X) = \sY^3_{A,V_5}$ as   schemes.\ 
Furthermore, for~\mbox{$[V_1] \in \Sigma_1(X) = \upkappa(\P(W_0^\perp))$}, 
the fiber~${\rho_1^{-1}}([V_1])$ is a quadric surface in the space
\begin{equation*}
\Theta_{V_1} \coloneqq \P((\C \oplus (V_1 \wedge V_5)) \cap W) \cong \P(V_1 \wedge V_5) \cong \P^3. 
\end{equation*}
It contains a plane~$\Pi$ if and only if its corank is at least~2, 
which by~\cite[Proposition~4.5]{DK1} is equivalent to~$[V_1] \in \Sigma_1(X) \cap \sY^3_{A,V_5}$.\ 
It remains to show that this plane is a reduced point of the Hilbert scheme~$\Fs_2(X)$, that is, $H^0(\Pi,N_{\Pi/X}) = 0$.

To prove this, consider the standard exact sequences
\begin{equation*}
0 \to N_{\Pi/X} \to N_{\Pi/M_X} \to \cO_\Pi(2) \to 0
\qquad\text{and}\qquad 
0 \to \cO_\Pi(1) \to N_{\Pi/M_X} \to N_{\Theta_{V_1}/M_X}\vert_\Pi \to 0.
\end{equation*}
It is easy to see that the composition~$\cO_\Pi(1) \to N_{\Pi/M_X} \to \cO_\Pi(2)$ 
is given by multiplication by the equation of the plane~$\Pi'$, the second  component of the quadric in~$\Theta_{V_1}$.\ 
Hence, its cokernel is~$\cO_L(2)$, where~$L = \Pi \cap \Pi'$ is a line.\ 
Therefore, we have an exact sequence
\begin{equation}
\label{eq:cn-pi-x-mx}
0 \to N_{\Pi/X} \to N_{\Theta_{V_1}/M_X}\vert_\Pi \to \cO_L(2) \to 0.
\end{equation}
Furthermore, the bundle~$N_{\Theta_{V_1}/M_X}$ is the null-correlation bundle on~$\P(V_1 \wedge V_5)$ 
(corresponding to a skew form~$\omega \in W_0^\perp \subset \bw2V_5^\vee$ whose kernel space is~$\upkappa(\omega) = [V_1]$).\ 
Hence,~$H^0(\Pi, N_{\Theta_{V_1}/M_X}\vert_\Pi)$ is the 1-dimensional space 
generated by the orthogonal complement of~$\Pi$ in~$V_5/V_1$ with respect to~$\omega$.\ 
Moreover, the argument of~\cite[Proposition~4.5]{DK1} shows that~$L = \P(K_3/K_1)$, 
where~$K_3 = A \cap (V_1 \wedge \bw2V_6)$ and~$K_1 = A \cap (V_1 \wedge \bw2V_5)$, and there is a commutative diagram
\begin{equation*}
\xymatrix{
0 \ar[r] & 
H^0(\Pi, N_{\Pi/M_X}) \ar[r] \ar[d] & 
V_5/V_1 \ar[r] \ar[d] & 
K_1^\vee \otimes K_3^\vee \ar[r] \ar@{=}[d] & 
0
\\
0 \ar[r] &
\Sym^2(K_3/K_1)^\vee \ar[r] &
\Sym^2K_3^\vee \ar[r] &
K_1^\vee \otimes K_3^\vee \ar[r] & 
0,
}
\end{equation*}
where the left vertical arrow is induced by the second map of~\eqref{eq:cn-pi-x-mx} 
and the middle vertical arrow is the restriction of the map~$\uptau \colon V_6/V_1 \to \Sym^2K_3^\vee$ from the proof of Proposition~\ref{propsing}.\ 
Since~$\uptau$ is injective (by~\cite[Proposition~2.5]{og4}), we obtain~$H^0(\Pi, N_{\Pi/X}) = 0$.

It remains to consider the scheme~$\Ft_2(X)$.\ 
In this case,~\cite[Proposition~4.1]{DK2} gives us an isomorphism
\begin{equation*}
\Ft_k(X) \cong \F_2(\P_X(V_5/\cU_X)/\Gr(3,V_5)).
\end{equation*}
So we can apply the same argument as in the case of $\sigma$-planes, 
using the ``second quadric fibration''~\mbox{$\rho_2 \colon \P_X(V_5/\cU_X) \to \Gr(3,V_5)$} (see~\cite[Section~4.4]{DK1}) instead of the first, 
\cite[Proposition~4.10]{DK1} instead of~\cite[Proposition~4.5]{DK1},
and the injectivity of the map 
\begin{equation*}
\uptau \colon \Hom(V_3, V_6/V_3) \lra \Sym^2\! K_4^\vee 
\end{equation*}
for any~$[V_3] \in \sZ^4_A$ with~$K_4 = A \cap (\bw2V_3 \wedge V_6)$,
which is proved in~\cite[Proposition~2.3]{Repwcube}.
\end{proof}

The following lemma bounds the dimension of~$\F_k(X)$ when~$k  <  n/2$.\

\begin{lemm}
\label{lem:fkx-fkmx}
Let~$X$ be a smooth GM variety of dimension~$n\ge3$.
\begin{aenumerate}
\item
\label{it:f1}
One has~$\dim(\F_1(X)) = 2n - 5$.
\item
\label{it:f2}
When~$n \ge 5$, one has~$\dim(\Fs_2(X)) \le 3n - 14$ and~$\dim(\Ft_2(X)) \le 3n - 12$.
\end{aenumerate}

In particular, $\Fs_k(M_X) \setminus \Fs_k(X)$ and~$\Ft_k(M_X) \setminus \Ft_k(X)$ 
are dense in~$\Fs_k(M_X)$ and~$\Ft_k(M_X)$, respectively, for every~$k \in \{1,2,3\}$.
\end{lemm}

\begin{proof}
For~\ref{it:f1}, we apply~\cite[Theorem~4.7]{DK2} 
(when~$n = 5$ or~$n = 6$, we need to extend the argument of~\cite{DK2} 
by an analysis of all the fibers of the map~$\sigma$ 
analogous to the one performed in Proposition~\ref{prop:corank-n-k} below).\ 

We now prove~\ref{it:f2}.\ 
For~$\Fs_2(X)$, we apply~\cite[Theorem~4.3]{DK2}.\ 
For~$\Ft_2(X)$, we apply~\cite[Theorem~4.5]{DK2}, 
which gives the inequalities
\begin{equation*}
\dim(\Ft_2(X))\le \dim (\sZ^{\ge 8-n}_{A,V_5} )\le \dim (\sZ^{\ge 8-n}_{A}),   
\end{equation*}
and note that by~Theorem~\ref{theorem:ikkr}, 
we have~$\dim(\sZ^{\ge 2}_{A}) = 6$ and~$\dim(\sZ^{\ge 3}_{A}) = 3$. 

For the last statement of the lemma: 
if~$X$ is ordinary, it follows from the irreducibility of~$\Fs_k(M_X)$ and~$\Ft_k(M_X)$ (see Lemma~\ref{lem:fk-mx}),
Corollary~\ref{cor:dim-fkmx}, and the above dimension bounds; if~$X$ is special, 
then~$\F_k^\star(X) \subsetneq \Fo_k^\star(M_X)$, and since~$\Fo_k^\star(M_X)$ is smooth and connected (see Lemma~\ref{lem:fk-cm}),
the complement~$\Fo_k^\star(M_X) \ssm \F_k^\star(X)$ is dense in~$\Fo_k^\star(M_X)$,
hence~$\F_k^\star(M_X) \ssm \F_k^\star(X)$ is dense in~$\F_k^\star(M_X)$.
\end{proof}

Finally, we compute the dimensions of~$\F_2^\star(X)$ when~$X$ is general.\ 
Schemes of negative dimensions are empty.

\begin{lemm}
\label{lem:f2-general}
When~$X$ is general of dimension~$n \ge 2$, one has~$\dim(\Fs_2(X)) = 3n - 14$ and~$\dim(\Ft_2(X)) = 3n - 15$.
\end{lemm}

\begin{proof}
We use again~\cite[Theorem~4.3]{DK2} and~\cite[Theorem~4.5]{DK2} 
and note that:
\begin{itemize}[wide]
\item 
$\Fs_2(X) = \Ft_2(X) = \vide$ for all GM varieties~$X$ of dimension~$n \le 3$ by Proposition~\ref{prop:fk-gm}.
\item 
$\sY^3_A = \sZ^4_A = \vide$ when~$A$ is a general Lagrangian subspace by Theorems~\ref{thm:og} and~\ref{theorem:ikkr},
hence~$\Fs_2(X) = \Ft_2(X) = \vide$ when~$X$ is a general GM fourfold.
\item 
$\dim(\sY^{\ge 2}_{A,V_5}) = 1$ and~$\dim(\sY^{\ge 1}_{A,V_5}) = 3$ by~\cite[Lemma~2.6]{DK2} for all~$A$ and~$V_5$, 
hence one has~$\dim(\Fs_2(X)) = 3n - 14$ for any smooth GM variety~$X$ of dimension~$n \in \{5,6\}$.
\item 
$\dim(\sZ^{\ge 3}_{A,V_5}) = 0$ and~$\dim(\sZ^{\ge 2}_{A,V_5}) = 3$ for any~$A$ and general~$V_5$, 
because they are fibers of a morphism from a $\P^2$-bundle 
over the irreducible threefold~$\sZ^{\ge 3}_A$ or sixfold~$\sZ^{\ge 2}_A$ to~$\Gr(5,V_6)$,
hence~$\dim(\Ft_2(X)) = 3n - 15$ for a general~$X$ of dimension~$n \in \{5,6\}$.\qedhere
\end{itemize}
\end{proof}

\begin{rema}
\label{rem:ft2}
Rizzo also proves in~\cite{Rthesis} that when~$X$ is general of dimension~$n\ge 4$, 
the scheme~$\Ft_2(X)$ has expected dimension~$3n - 15$ and that this also holds for {\em all} smooth GM sixfolds.
\end{rema}


\subsection{Hilbert schemes of quadrics}
\label{sec:gkx}

Given a projective scheme~$Z$ and an ample divisor class~$H$ on $Z$,
we denote by~$\G_k(Z)$ the Hilbert scheme of quadrics of dimension~$k$ in~$Z$, 
that is, of subschemes with Hilbert polynomial
\begin{equation*}
h^\G_k(t) = \frac{(t+1)\cdots(t+k-1)}{k!}(2t+k)
\end{equation*}
with respect to the ample class~$H$.\

If the Hilbert polynomial of~$\Sigma \subset \P(W)$ is equal to~$h^\G_k(t)$,
then~$\Sigma$ is a hypersurface of degree~$2$ in a linear subspace~$\Pi \subset \P(W)$ of dimension~$k + 1$,
called the {\sf linear span} of~$\Sigma$ and denoted by~$\langle \Sigma \rangle$.\ 
Therefore, 
\begin{equation}
\label{eq:gk-pw}
\G_k(\P(W)) \cong \P_{\Gr(k+2,W)}(\Sym^2\!\cR_{k+2}^\vee),
\end{equation}
where~$\cR_{k+2}$ is the tautological vector subbundle of rank~\mbox{$k + 2$} on~$ \Gr(k+2,W)$.\ 
If~$Z \subset \P(W)$, then~$\G_k(Z)$ is the subscheme of~$\P_{\Gr(k+2,W)}(\Sym^2\!\cR_{k+2}^\vee)$ 
parameterizing quadrics contained in~$Z$.\ 
In particular, $\G_0(Z)$ is the Hilbert square of~$Z$, the scheme~$\G_1(Z)$ is the Hilbert scheme of conics on~$Z$, and so on.

Following Definition~\ref{def:st-subspaces}, we introduce some more   terminology and notation.

\begin{defi}
\label{def:st-quadrics}
 A quadric~$\Sigma \subset \CGr(2,V_5)$ 
is a {\sf $\sigma$-quadric} or a {\sf $\rtx$-quadric} or a {\sf $\sigma\rtx$-quadric} 
if its linear span~$\langle \Sigma \rangle \subset \P(\C \oplus \bw2V_5)$ is contained in~$\CGr(2,V_5)$ and
is a {\sf $\sigma$-space} or a {\sf $\rtx$-space} or a {\sf $\sigma\rtx$-space}, respectively.\ 
If~$X$ is a GM variety and~$k\ge0$, we denote by
\begin{equation*}
\Gs_k(X) \subset \G_k(X),
\qquad 
\Gt_k(X) \subset \G_k(X),
\qquad 
\Gst_k(X) = \Gs_k(X) \cap \Gt_k(X) \subset \G_k(X)
\end{equation*}
the corresponding closed subschemes of the Hilbert scheme~$\G_k(X)$.\ 
 Finally, we write
\begin{equation}
\label{eq:g0k}
\G^0_k(X) \coloneqq \G_k(X) \ssm (\Gs_k(X) \cup \Gt_k(X)).
\end{equation} 
The closure~$\overline{\G^0_k(X)} \subset \G_k(X)$ is called {\sf the main component} of~$\G_k(X)$.
 \end{defi}

The scheme~$\G^0_k(X)$ parameterizes quadrics in~$X$ 
whose linear span is {\em not} contained in~$\CGr(2,V_5)$ (equivalently, is not contained in~$M_X$).\ 
The above definition implies  
\begin{equation*}
\G_k(X) = \overline{\G^0_k(X)} \cup \Gs_k(X) \cup \Gt_k(X),
\end{equation*}
a union of closed subschemes,
where~$\Gs_k(X)$ and~$\Gt_k(X)$ are disjoint for all~$k\ge 1$ when~$X$ is ordinary, 
and for all~$k\ge 2$ when~$X$ is special.\ 
As we will see, in many cases, this is a decomposition of~$\G_k(X)$ into irreducible or even connected components.\ 

Note that the quadrics parameterized by the scheme~$\G^0_k(X)$ 
were called ``$\rtx$-quadrics'' in~\cite[Sections~3 and~7.3]{dim2},
while those in~$\Gt_k(X)$ were called ``$\rho$-quadrics''.

The schemes~$\G_k(X)$ are the main subjects of study in this paper.\ 
As we will see, for~$k \ge 2$, their structure simplifies (see Corollary~\ref{cor:fkm-connected}).\ 
The following lemma is important to understand the structure of~$\G_1(X)$, the Hilbert scheme of conics.

Recall that any conic~$\Sigma$ is a local complete intersection scheme in any ambient smooth variety~$Z$,
hence its deformation theory is controlled by its normal bundle~$\cN_{\Sigma/Z}$.\ 
In particular, if~$H^1(\Sigma, \cN_{\Sigma/Z}) = 0$, the scheme~$\G_1(Z)$ is smooth at~$[\Sigma]$,
of dimension
\begin{equation*}
h^0(\Sigma, \cN_{\Sigma/Z}) = \upchi(\Sigma, \cN_{\Sigma/Z}) =  \dim(Z) - K_Z \cdot \Sigma - 3,
\end{equation*}
where the Euler characteristic is computed by Riemann--Roch.
 
In the next lemma, we apply this to the punctured cone over~$\Gr(2,V_5)$.

\begin{lemm}
\label{lem:g1-cgr}
The Hilbert scheme~$\G_1(\CGr(2,V_5) \ssm \{\bv\})$ of conics 
on the punctured cone $\CGr(2,V_5) \ssm \{\bv\}$  over the Grassmannian is smooth and irreducible of dimension~$16$.
\end{lemm}

\begin{proof}
First, we observe that the Hilbert scheme~$\G_1(\Gr(2,V_5))$ of conics on the Grassmannian is smooth of dimension~$13$ 
(see also~\cite[Section~3]{IM}).\
Indeed, it is enough to show that for any conic~$\Sigma \subset \Gr(2,V_5)$,
 the normal bundle~$\cN_{\Sigma/\!\Gr(2,V_5)}$ has zero higher cohomology.\ 
When~$\Sigma$ is smooth or is a union of two lines, 
this follows from the global generation of~$\cN_{\Sigma/\Gr(2,V_5)}$ away from the singular point
(which itself follows from the global generation of the tangent bundle of the Grassmannian).\ 
When~$\Sigma$ is nonreduced, this follows from the exact sequence of~\cite[Lemma~A.2.4]{KPS}
because the normal bundle to any line on~$\Gr(2,V_5)$ is globally generated.

Furthermore, it is easy to see that the scheme~$\G_1(\Gr(2,V_5))$  
is birational to a $\Gr(3,6)$-bundle over~$\Gr(4,V_5)$ hence it is irreducible.

Let~$\mathring{\G}_1(\CGr(2,V_5)) \subset \G_1(\CGr(2,V_5) \ssm \{\bv\})$ be the open subscheme
parameterizing conics whose linear span does not contain the vertex~$\bv$.\ 
Clearly, linear projection from~$\bv$ induces a Zariski-locally trivial fibration
\begin{equation*}
\mathring{\G}_1(\CGr(2,V_5)) \lra \G_1(\Gr(2,V_5))
\end{equation*}
with fiber~$\A^3$  hence,~$\mathring{\G}_1(\CGr(2,V_5))$ is smooth and irreducible of dimension~$16$.

The complement of~$\mathring{\G}_1(\CGr(2,V_5))$ parameterizes conics contained in the cones over lines $L$ on~$\Gr(2,V_5)$.\ 
For any such conic~$\Sigma$, there is an exact sequence
\begin{equation*}
0 \to \cO_{\CGr(2,V_5)}(2)\vert_\Sigma \to \cN_{\Sigma/\!\CGr(2,V_5)} \to \gamma^*\cN_{L/\!\Gr(2,V_5)} \to 0,
\end{equation*}
where~$\gamma \colon \CGr(2,V_5) \ssm \{\bv\} \to \Gr(2,V_5)$ is the projection from the vertex $\bv$ and~$L = \gamma(\Sigma)$.\ 
Since~$\mathbf{R}\gamma_*\gamma^*\cN_{L/\!\Gr(2,V_5)} \cong \cN_{L/\!\Gr(2,V_5)} \oplus \cN_{L/\!\Gr(2,V_5)}(-1)$ 
and~$\cN_{L/\!\Gr(2,V_5)}$ is globally generated,
it follows  that~$H^1(\Sigma, \gamma^*\cN_{L/\!\Gr(2,V_5)}) = 0$.\ 
It is also clear that~$H^1(\Sigma, \cO_{\CGr(2,V_5)}(2)\vert_\Sigma) = 0$.\ 
Therefore, $H^1(\Sigma, \cN_{\Sigma/\!\CGr(2,V_5)}) = 0$, hence~$\G_1(\CGr(2,V_5) \ssm \{\bv\})$ is smooth of dimension~$16$ at~$[\Sigma]$,
and hence it is everywhere smooth of dimension~$16$.

Finally, it is easy to see that the complement of~$\mathring{\G}_1(\CGr(2,V_5))$ in~$\G_1(\CGr(2,V_5) \ssm \{\bv\}))$ has dimension~$13$, 
hence~$\mathring{\G}_1(\CGr(2,V_5))$ is dense in~$\G_1(\CGr(2,V_5) \ssm \{\bv\}))$,
and therefore the Hilbert scheme~$\G_1(\CGr(2,V_5) \ssm \{\bv\}))$ is irreducible.
\end{proof}


\section{Yet another quadric fibration}
\label{sec:qf}

In~\cite{DK1,DK2}, we defined and extensively used two quadric fibrations associated with a GM variety.\ 
In Section~\ref{ss:qf}, we define yet another quadric fibration~$\cQ$ with base the variety~$B$ defined in~\eqref{base},
and in Section~\ref{sec51}, we describe the corank stratification of~$B$ induced by it.

\subsection{The quadric fibration}
\label{ss:qf}

Consider the~$4$-dimensional projective space
\begin{equation}
\label{eq:b4}
B_4 \coloneqq \Gr(4,V_5) \cong \P(V_5^\vee).
\end{equation} 
The vector bundle on~$B$ in whose projectivization~$\cQ$ lives 
is the pullback by the projection~\mbox{$B\to B_4$} of a vector bundle~$\cW$ on~$B_4$ which we construct in the next lemma.

For each $4$-dimensional subspace~$U_4 \subset V_5$, 
we define the vector space
\begin{equation*}
\cW_{[U_4]} \coloneqq (\C \oplus \bw2U_4) \cap W \subset \C \oplus \bw2V_5.
\end{equation*}
Following~\cite[Lemma~3.1]{lo} and~\cite[Section~3.4]{dim}, 
we show that these spaces form a vector bundle of rank~$n+1$ over~$B_4$.\ 
Denote by~$\cU_4 \subset V_5 \otimes \cO_{B_4}$ the tautological subbundle on~$B_4$.

\begin{lemm}
\label{lemma:cw-locally-free}
Let~$X$ be a  GM variety of dimension~$n$ satisfying Property~\eqref{hh}.\ 
The morphism   
\begin{equation*}
\cO_{B_4} \oplus \bw2\cU_4 \to 
(\C \oplus \bw2V_5) \otimes \cO_{B_4} \to 
((\C \oplus \bw2V_5)/W) \otimes \cO_{B_4}
\end{equation*}
of vector bundles on~$B_4$ is surjective at every point and its kernel~$\cW_X$ 
is locally free of rank~$n+1$.\ 
If~$X$ is special and~$X_0$ is the corresponding ordinary GM variety, we have
\begin{equation}
\label{defwk1:special}
\cW_X = \cO_{B_4} \oplus \cW_{X_0}.
\end{equation} 
\end{lemm}

\begin{proof}
Assume that~$X$ is ordinary.\ 
Then~$\cW_X$ is isomorphic to the kernel of the composition
\begin{equation*}
\bw2\cU_4 \to 
\bw2V_5 \otimes \cO_{B_4} \to 
(\bw2V_5 / W_0) \otimes \cO_{B_4},
\end{equation*}
where recall that~$W_0$ is the image of~$W$ in~$\bw2V_5$.\ 
Assume that, at a point~$[U_4]$, 
the morphism is not surjective, or equivalently, that the transposed map is not injective.\ 
There exists then a nonzero skew form~$\omega \in W_0^\perp$   
such that~$\bw2U_4$ is contained in the hyperplane of~$\bw2V_5$ orthogonal to~$\omega$,
that is, $U_4$ is isotropic for~$\omega$.\ 
But~$\omega$ is then decomposable, which, since~\mbox{$n \ge 2$},  
is impossible by Lemma~\ref{lem:wperp}.\
 This proves that~$\cW_X$ is locally free of rank~$n + 1$ for ordinary~$X$.

If~$X$ is special, we have~$W = \C \oplus W_0$ (see~\eqref{eq:w-w0}),
hence the morphism defining~$\cW_X$ is the direct sum of the zero morphism~$\cO_{B_4} \to 0$ 
and a similar morphism for~$X_0$, hence~\eqref{defwk1:special} holds, and the lemma follows in that case.
\end{proof}

By definition, the vector bundle~$\cW_X$ only depends on the linear space~$W$, that is, 
only on the Grassmannian hull~$M_X$ 
(in the proof above, we only used the fact that~$M_X$ is smooth in the ordinary case, 
and a cone over the smooth variety~$M'_X = M_{X_0}$ in the special case).\ 
So, we will denote it simply by~$\cW$.

Let~$H_4$ denote the hyperplane class of~$B_4$, so that~$(V_5 \otimes \cO_{B_4})/\cU_4 \cong \cO_{B_4}(H_4)$.\ 
Using the Koszul sequence~$0 \to \bw2\cU_4 \to \bw2V_5 \otimes \cO_{B_4} \to \cU_4(H_4) \to 0$, 
one can also define~$\cW$ by means of the exact sequence
\begin{equation}
\label{defwk2}
0 \to \cW \to W\otimes\cO_{B_4} \to \cU_4(H_4) \to 0.
\end{equation} 
This presentation is sometimes more convenient. 

\begin{rema}
If~$W \subset \C \oplus \bw2V_5$ is a vector subspace such that~$\P(W)$ is transverse to~$\CGr(2,V_5)$
and~$M \coloneqq \CGr(2,V_5) \cap \P(W)$ is the corresponding linear section, there is an isomorphism
\begin{equation*}
\P_{B_4}(\cW) \cong \Bl_{M}(\P(W))
\end{equation*}
that follows easily from the standard isomorphism~$\P_{B_4}(\bw2\cU_4) \cong \Bl_{\Gr(2,V_5)}(\P(\bw2V_5))$.\ 
This gives a geometric meaning to the vector bundle~$\cW$.
\end{rema}

Recall that $X\subset  \P(W)$ is the intersection of the  space~$V_6$ of quadrics
generated by the space~$V_5$ of Pl\"ucker quadrics and the  quadric~$Q$.\ 
We denote by 
\begin{equation*}
\bq \colon V_6 \lhra \Sym^2\!W^\vee
\end{equation*}
the corresponding embedding.\
The following observation is crucial for the rest of the paper.

\begin{lemm}
\label{lemma:u4-vanishes}
Let~$X$ be a  GM variety  satisfying Property~\eqref{hh}.\  
The composition
\begin{equation}
\label{eq:cu4-s2wv}
\cU_4 \lhra V_5\otimes\cO_{B_4}
\xrightarrow{\ \bq\vert_{V_5}\ }
\Sym^2\!W^\vee \otimes \cO_{B_4} \lra
\Sym^2\!\cW^\vee
\end{equation}
of morphisms of sheaves on~$B_4$ vanishes;
in other words, the Pl\"ucker quadrics coming from any hyperplane~$U_4\subset V_5$ all vanish on~$\P(\cW_{[U_4]})  \subset \P(W)$.\ 

More precisely,  if
\begin{equation*}
\pr_1 \colon B_4 \times (\P(W) \ssm M_X) \to B_4
\qquad\text{and}\qquad 
\pr_2 \colon B_4 \times (\P(W) \ssm M_X) \to \P(W) \ssm M_X
\end{equation*}
are the projections,
the zero locus of the morphism~$\pr_1^*\cU_4 \to \pr_2^*(\cO_{\P(W) \ssm M_X}(2))$
of sheaves on~$B_4 \times (\P(W) \ssm M_X)  = (B_4 \times \P(W)) \ssm (B_4 \times M_X)$ 
induced by the composition of the first two arrows in~\eqref{eq:cu4-s2wv}
is the subscheme~$\P_{B_4}(\cW) \ssm (B_4 \times M_X)$.
\end{lemm}

\begin{proof}
We consider the morphisms of sheaves 
\begin{equation}
\label{eq:two-maps}
\pr_1^*(\cU_4) \otimes \pr_2^*(\cO_{\P(W)}(-2)) \to \cO
\quad\text{and}\quad 
\pr_1^*(\cU_4^\vee(-H_4)) \otimes \pr_2^*(\cO_{\P(W)}(-1)) \to \cO,
\end{equation}
where the first morphism is given by the first two arrows in~\eqref{eq:cu4-s2wv} 
and the second is given by the dual of the second arrow in~\eqref{defwk2}.\ 
Since the vanishing of the composition~\eqref{eq:cu4-s2wv} is a consequence of the second statement of the lemma 
and since by~\eqref{defwk2} the zero locus of the second morphism in~\eqref{eq:two-maps} is $ \P_{B_4}(\cW)$, 
all we need to show is that the images of these two morphisms in~\eqref{eq:two-maps} coincide over~$\P(W) \ssm M_X$.\ 
We will prove this coincidence on appropriate open subsets of~$B_4 \times \P(W)$ 
and then show that these subsets cover~$B_4 \times (\P(W) \ssm M_X)$.

First, we choose a basis~$(e_1,e_2,e_3,e_4,e_5)$ for~$V_5$
and consider the open subset~\mbox{$B_4^{(e_1)} \subset B_4$} defined by the condition~$e_1 \notin U_4$.\ 
 Clearly, $B_4^{(e_1)} \cong \A^4$, and we can write
\begin{equation*}
U_4 = \langle e_2 + y_2e_1, e_3 + y_3e_1, e_4 + y_4e_1, e_5 + y_5e_1 \rangle,
\end{equation*}
where~$(y_2,y_3,y_4,y_5)$ are coordinates on~$B_4^{(e_1)}$.\ 
The  first map in~\eqref{eq:two-maps} is then given 
by the Pl\"ucker quadrics
\begin{equation*}
(\bq_{e_2} + y_2\bq_{e_1}, \bq_{e_3} + y_3\bq_{e_1}, \bq_{e_4} + y_4\bq_{e_1}, \bq_{e_5} + y_5\bq_{e_1}).
\end{equation*}
Since $V_5 = \C e_1 \oplus U_4$, there is a direct sum decomposition
\begin{equation*}
\bw2V_5 = (e_1 \wedge U_4) \oplus \bw2U_4 =
\langle e_{12}, e_{13}, e_{14}, e_{15} \rangle \oplus
\langle e_{ij} + y_i e_{1j} - y_j e_{1i} \rangle_{2 \le i < j \le 5},
\end{equation*}
where~$e_{ij}\coloneqq e_i\wedge e_j$.\ 
A direct computation shows that the projection to its first summand is given by
\begin{equation*}
\sum_{1 \le i < j \le 5} x_{ij}e_{ij} \longmapsto
\sum_{i=2}^5 \Big( x_{1i} + \textstyle\sum_{j=2}^5 x_{ij}y_j \Big)e_{1i},
\end{equation*}
hence the second map in~\eqref{eq:two-maps} is given by
\begin{equation*}
\Big(
x_{12} + \textstyle\sum x_{2j}y_j,\ 
x_{13} + \textstyle\sum x_{3j}y_j,\  
x_{14} + \textstyle\sum x_{4j}y_j,\ 
x_{15} + \textstyle\sum x_{5j}y_j 
\Big).
\end{equation*}
It remains to note that
\begin{equation*}
\begin{pmatrix}
0 & \hphantom{-}x_{23} & {-}x_{24} & \hphantom{-}x_{34} \\
{-}x_{23} & 0 & \hphantom{-}x_{25} & {-}x_{35} \\
\hphantom{-}x_{24} & {-}x_{25} & 0 & \hphantom{-}x_{45} \\
{-}x_{34} & \hphantom{-}x_{35} & {-}x_{45} & 0
\end{pmatrix}
\cdot
\begin{pmatrix}
x_{15} + \sum x_{5j}y_j \\ x_{14} + \sum x_{4j}y_j \\ x_{13} + \sum x_{3j}y_j \\ x_{12} + \sum x_{2j}y_j 
\end{pmatrix}
=
\begin{pmatrix}
\bq_{e_5} + y_5\bq_{e_1} \\ \bq_{e_4} + y_4\bq_{e_1} \\ \bq_{e_3} + y_3\bq_{e_1} \\ \bq_{e_2} + y_2\bq_{e_1}  
\end{pmatrix},
\end{equation*}
and that the Pfaffian of the first factor in the left side is equal to~$\bq_{e_1}$, hence this factor is invertible over~$\P(W) \ssm Q_{e_1}$,
where~$Q_{e_1}$ is the Pl\"ucker quadric.\ 
Therefore, on the open subscheme~$B_4^{(e_1)} \times (\P(W) \ssm Q_{e_1})$, the images of the two maps in~\eqref{eq:two-maps} agree.

It follows that the required equality of schemes is established on the open subset
\begin{equation*}
\bigcup_{ v \in V_5\ssm\{0\}} B_4^{(v)} \times (\P(W) \ssm Q_v)
\end{equation*}
of~$B_4 \times \PP(W)$.
It remains to check that this union is equal to~$B_4 \times (\P(W) \ssm M_X)$.\ 
For that, it is enough to identify the fibers over an arbitrary point~$[U_4] \in B_4$,
that is, to show that
\begin{equation*}
\bigcup_{v \in V_5 \ssm U_4} (\P(W) \ssm Q_v) = \P(W) \ssm \bigcap_{v \in V_5 \ssm U_4} Q_v
\end{equation*}
is equal to~$\P(W) \ssm M_X$.\ 
But this is evident, because the Pl\"ucker quadrics~$Q_v$ with~$v \in V_5 \ssm U_4$ span the space of all Pl\"ucker quadrics,
and the intersection of those is~$M_X$.
\end{proof}

By Lemma~\ref{lemma:u4-vanishes}, the morphism 
\begin{equation}
\label{eq:v6-s2wd}
V_6\otimes\cO_{B_4} \xrightarrow{\ \bq\ } 
\Sym^2\!W^\vee \otimes \cO_{B_4} \lra
\Sym^2\!\cW^\vee
\end{equation}
factors through the quotient bundle~$(V_6\otimes\cO_{B_4}) / \cU_4$
and so defines a subsheaf of~$\Sym^2\!\cW^\vee$ of rank at most~$2$.\  
Note also that a choice of splitting~$V_6 = V_5 \oplus \C$ induces an isomorphism
\begin{equation}
\label{eq:v6-mod-cu4}
(V_6\otimes\cO_{B_4}) / \cU_4 \cong 
((V_5\otimes\cO_{B_4})/\cU_4) \oplus \cO_{B_4} \cong 
\cO_{B_4}(H_4) \oplus \cO_{B_4}.
\end{equation}
We use it in the next lemma, which shows that under a mild assumption on~$X$, 
the morphism~\eqref{eq:v6-s2wd} factors through a fiberwise monomorphism, 
and its image is locally free of rank~$2$.

\begin{lemm}
\label{lemma-subbundle-oe}
Let~$X$ be a  GM variety of dimension~$n$ satisfying Property~\eqref{hh}.\ 
If~$X$ contains no quadrics of dimension~$n - 1$, the morphism
\begin{equation*}
\eta \colon \cO_{B_4}(H_4) \oplus \cO_{B_4} \cong (V_6\otimes\cO_{B_4}) / \cU_4 \lra \Sym^2\!\cW^\vee
\end{equation*}
induced by~$\bq$ is an embedding of vector bundles.\ 
This holds in particular when~$n \ge 3$.
\end{lemm}

\begin{proof}
The isomorphism was discussed above,
so we only need to check fiberwise injectivity.
Let~$[U_4] \in B_4$ and assume~$U_4 \subset U_5 \subset V_6$ is such that~$\eta(U_5/U_4) = 0$.\
The intersection of~$X$ with~$\P((\C \oplus \bw2U_4) \cap W) = \P(\cW_{[U_4]}) \cong \P^{n}$ 
is then given by at most one quadratic equation, and therefore this intersection (and hence~$X$ as well) 
contains a quadric of dimension~$n - 1$.\ 

It remains to note that when~$n \ge 3$, the variety~$X$ contains no quadrics of dimension~\mbox{$n - 1$}
by the Lefschetz theorem (see~\cite[Corollary~3.5]{DK2}).
\end{proof}

In addition to the 4-dimensional projective space~$B_4$ defined by~\eqref{eq:b4}, 
we consider the~$5$-dimensional projective space
\begin{equation*}
B_5 \coloneqq \Gr(5,V_6) \cong \P(V_6^\vee).
\end{equation*} 
We denote by~$\cU_5 \subset V_6 \otimes \cO_{B_5}$ the tautological subbundle and by~$H_5$ the hyperplane class of~$B_5$,
so that~$(V_6 \otimes \cO_{B_5})/\cU_5 \cong \cO_{B_5}(H_5)$.\
 
Recall that for a GM variety~$X$, the embedding~$V_5 \subset V_6$ of the space of Pl\"ucker quadrics 
into the space of quadrics containing~$X$ corresponds to a point~$\bp_X \in B_5$, called the Pl\"ucker point of~$X$.
We defined the subvariety~$B \subset B_4 \times B_5$ in~\eqref{base}.\ 
Abusing notation, we will denote the pullbacks to~$B$ 
of the tautological bundles and hyperplane classes of~$B_4$ and~$B_5$ by~$\cU_4$, $\cU_5$, $H_4$, and~$H_5$, respectively.\ 
The main properties of~$B$ are summarized in the following lemma.

\begin{lemm}
\label{ccb}
The projection~$B \to B_5$ is the blowup of the Pl\"ucker point~${\bp_X} \in B_5$ 
and the class of the exceptional divisor~$E \subset B$ of the blowup 
is given by~$E \lin H_5 - H_4$.

The projection~$B \to B_4$ induces isomorphisms
\begin{equation*}
B \cong \P_{B_4}(\cO_{B_4}(H_4) \oplus \cO_{B_4})
\qquad\text{and}\qquad 
E \cong B_4
\end{equation*}
such that the relative tautological subbundle for this projectivization is isomorphic to~$\cO_B(-E)$.
Moreover, $E \subset B$ is the section of the projection~$B  = \P_{B_4}(\cO_{B_4}(H_4) \oplus \cO_{B_4}) \to B_4$ corresponding to the first summand.
\end{lemm}

\begin{proof}
The first part of the lemma is standard.\ 
For the second part, 
we note that~$\cU_5/\cU_4$ is a line subbundle in~$(V_6 \otimes \cO_B)/\cU_4$
so, using~\eqref{eq:v6-mod-cu4}, we obtain a morphism~\mbox{$B \to \P_{B_4}(\cO_B(H_4) \oplus \cO_B)$} 
which is easily seen to be an isomorphism, under which~$\cU_5/\cU_4$ is identified with the tautological subbundle.\ 
Moreover, since~$(V_5 \otimes \cO_B)/\cU_4$ corresponds to the summand~$\cO_{B_4}(H_4)$,
it follows that~$E$ is the corresponding section of the projection~$B \to B_4$.\ 
Finally, the map
\begin{equation*}
\cU_5/\cU_4 \to (V_6 \otimes \cO_B)/\cU_4 \to (V_6 \otimes \cO_B) / (V_5 \otimes \cO_B) \cong \cO_B
\end{equation*}
vanishes exactly on~$E$, hence the tautological subbundle coincides with~$\cO_B(-E)$.
\end{proof}

Let~$X$ be a  GM variety of dimension~$n$ satisfying Property~\eqref{hh}.\ 
We denote by~$\cW_B$ the pullback by the projection~$B\to B_4$ of the vector bundle~$\cW$ 
of rank~$n+1$ on~$B_4$ defined in Lemma~\ref{lemma:cw-locally-free}.\ 
Pulling back to~$B$ the embedding~$\eta$ from Lemma~\ref{lemma-subbundle-oe}
and restricting it to the tautological subbundle~$\cU_5/\cU_4 \cong \cO_B(-E)$ (see Lemma~\ref{ccb}),
we obtain an embedding
\begin{equation}
\label{eq:iota}
\eta \colon \cO_B(-E) \lhra \Sym^2\!\cW_B^\vee.
\end{equation}
We denote by 
\begin{equation}
\label{eq:defQ}
 \cQ \subset \PP_B(\cW_B)\lra B
\end{equation}
the family of quadrics defined by~$\eta$.\ 
Under the assumption of Lemma~\ref{lemma-subbundle-oe}, the image of~$\eta$ is a line subbundle;
therefore, the map~$\cQ \to B$ is a flat fibration in quadrics of dimension~$n-1$.

\subsection{The corank stratification}
\label{sec51}

Consider the cokernel sheaf of the quadric bundle~$\cQ$, 
that is, the sheaf~${\cC_X}$ on~$B$ defined by the exact sequence 
\begin{equation}
\label{equation-sheaf-cf}
0 \to \cW_B(-E) \xrightarrow{\quad} \cW_B^\vee \xrightarrow{\quad} \cC_X \to 0, 
\end{equation}
where the morphism~$\cW_B(-E) \to \cW_B^\vee$ comes from the equation~$\eta$ of~$\cQ$ (see~\eqref{eq:iota}).\
When the GM variety~$X$ is clear from the context, we may abbreviate~$\cC_X$ to~$\cC$.\ 
Denote by
\begin{equation}
\label{eq:bgec}
B^{\ge c} \subset B
\end{equation}
the corank~$c$ degeneracy locus of this morphism 
(alternatively, this is the locus of points where the rank of~$\cC_X$ is at least~$c$)
with its natural scheme structure, defined by the corresponding Fitting ideal.\
We will also use the notation~$B^c \coloneqq B^{\ge c} \ssm B^{\ge c+1}$.

On the exceptional divisor~$E = \P(V_5^\vee)\subset B$, 
the corank stratification for the quadric fibration~$\cQ \to B$ restricts to the corank stratification 
for the restricted quadric fibration
\begin{equation}
\label{defqe}
\cQ_E := \cQ \times_B E\lra E.
\end{equation}
It is given by the subschemes
\begin{equation}
 \label{newe}
E^{\ge c} := E \cap B^{\ge c}
\qquad\text{and}\qquad
E^c := E \cap B^c = E^{\ge c} \ssm E^{\ge c+1}.
\end{equation}
The following lemma reduces the computation of the stratifications to the ordinary case.

\begin{lemm}
\label{lem:b-strata-special}
If~$X$ is a special GM variety and~$X_0$ is its associated ordinary variety, we have 
\begin{equation*}
\cC_X \cong \cC_{X_0} \oplus \cO_E.
\end{equation*}
In particular, the corank stratifications of~$B$ corresponding to~$X$ and~$X_0$ agree on~$B \ssm E$
and differ by~$1$ on~$E$.
\end{lemm}

\begin{proof}
We have~$\cW_X = \cO_{B} \oplus \cW_{X_0}$ by~\eqref{defwk1:special}.\ 
Moreover, it follows from~\cite[Proposition~2.30]{DK1} that this direct sum decomposition 
is orthogonal for the family of quadratic forms~$\eta$,
the restriction of~$\eta$ to the first summand is the natural morphism~$\cO_B(-E) \to \cO_B$,
and the restriction to the second summand is given by the family of quadratic forms associated with~$X_0$,
hence the lemma.
\end{proof}

We analyze next the corank stratification on the complement~$B \ssm E$ of the exceptional divisor.\
Let~$A \subset \bw3V_6$ be the Lagrangian subspace associated with~$X$ (see Section~\ref{subsec:ld}).\ 
On~$\PP(V_6^\vee) = B_5$, we   consider the two Lagrangian subbundles 
\begin{equation}
\label{eq:ai}
\cA_1 \coloneqq A \otimes \cO_{B_5}
\qquad\text{and}\qquad 
\cA_2 \coloneqq \bw3\cU_5 
 \end{equation}
 of the symplectic bundle~$\bw3V_6 \otimes \cO_{B_5}$.\ 
Consider the Lagrangian cointersection sheaf
\begin{equation*}
 \cC(\cA_1,\cA_2) \coloneqq \Coker\bigl(\bw3\cU_5 \lra (\bw3V_6 / A) \otimes \cO_{B_5}\bigr)
\end{equation*}
as defined in~\cite[Section~4]{DK3}.\
In the next lemma, we use the double covers
constructed in~\cite[Sections~4.1 and~4.3]{DK3} from these Lagrangian subbundles,
the double covers associated with a quadratic fibration in~\cite[Section~3.1]{DK3}, 
and the EPW stratification~$\sY_\Ap^{\ge c}$ of~$B_5 = \P(V_6^\vee)$ defined in~\eqref{ya}.\ 
Recall that~$B \ssm E = B_5 \ssm \{\bp_X\}$ by Lemma~\ref{ccb}.\

\begin{lemm}
\label{lemma:ker-qb}
Let~$X$ be a  GM variety   satisfying Property~\eqref{hh}.\ 
 Then,
\begin{equation*}
\cC_X\vert_{B \ssm E} \cong \cC(\cA_1,\cA_2)\vert_{B_5 \ssm \{\bp_X\}}
\qquad\text{and}\qquad
B^{\ge c} \ssm E = \sY^{\ge c}_\Ap \ssm \{\bp_X\}
\end{equation*}
for all~$c \ge 0$, where the sheaves~$\cC_X$ and~$\cC(\cA_1,\cA_2)$ were defined in~\eqref{equation-sheaf-cf} and~\eqref{eq:ai}.\ 
Moreover, for each~$c \ge 0$, the double cover of~$B^{\ge c} \ssm E$ associated with the quadratic fibration~$\cQ \to B$
is isomorphic to the restriction of the double cover~$\uvt_\Ap \colon \wtY_\Ap^{\ge c} \to \sY_\Ap^{\ge c}$ from Theorem~\textup{\ref{thm:tya}}.
\end{lemm}

\begin{proof}
By Lemma~\ref{lem:b-strata-special}, we have~$\cC_X\vert_{B \ssm E} \cong \cC_{X_0}\vert_{B \ssm E}$
and, by Theorem~\ref{thm:gm-ld}, the Lagrangian subspaces associated with~$X$ and~$X_0$ coincide,
so we may and will assume that~$X$ is ordinary.

Consider   the other Lagrangian subbundle~$\cA_3 \coloneqq \bw3V_5 \otimes \cO_{B_5} \subset \bw3V_6 \otimes \cO_{B_5}$
and the restriction of~$\cA_1,\cA_2,\cA_3$ to~$B_5 \ssm \{\bp_X\} = B \ssm E$.\
We have
\begin{equation*}
\cA_1 \cap \cA_3 = (A \cap \bw3V_5) \otimes \cO = W_0^\perp \otimes \cO,
\qquad 
\cA_2 \cap \cA_3 = \bw3\cU_5 \cap (\bw3V_5 \otimes \cO) = \bw3\cU_4,
\end{equation*}
where the identification with~$W_0^\perp$ in the first equality follows from~\cite[Proposition~3.13]{DK1},
while~$\cU_4 = \cU_5 \cap (V_5 \otimes \cO)$ is the restriction of the tautological bundle from~$B$.\
Moreover, the subbundles~$\cA_1 \cap \cA_3$ and~$\cA_2 \cap \cA_3$ intersect trivially in~$\bw3V_6 \otimes \cO$,
because~$\cA_2 \cap \cA_3$ consists of decomposable vectors while~$\cA_1 \cap \cA_3$ 
contains no such vectors by Theorem~\ref{thm:gm-ld}.\
We then apply isotropic reduction (see~\cite[Section~4.2]{DK3}) with respect to the isotropic subbundle
\begin{equation*}
\cI = \cI_1 \oplus \cI_2 = (\cA_1 \cap \cA_3) \oplus (\cA_2 \cap \cA_3)
\end{equation*}
and obtain a new symplectic bundle~\mbox{$\cI^\perp/\cI$} over~$B \ssm E$
with three Lagrangian subbundles~$\bcA_1, \bcA_2, \bcA_3$
defined as~$\bcA_i \coloneqq (\cA_i \cap \cI^\perp) / (\cA_i \cap \cI)$.\

Note that~$\cI \subset \cA_3$, hence~$\cA_3 \subset \cI^\perp$, hence
\begin{equation*}
\bcA_3  = \cA_3/\cI \cong 
(\bw3V_5 \otimes \cO) / ((W_0^\perp \otimes \cO) \oplus \bw3\cU_4) \cong
(W^\vee \otimes \cO) / \bw3\cU_4.
\end{equation*}
Taking into account the isomorphism~$\bw3\cU_4 \cong \cU_4^\vee(-H_4)$ and using the dual of~\eqref{defwk2},
we obtain
\begin{equation*}
\bcA_3 \cong \cW^\vee\vert_{B \ssm E}.
\end{equation*}

Note that the subbundle~$\bcA_3$ intersects 
trivially both~$\bcA_1$ and~$\bcA_2$ by definition of~$\cI$.\ 
Therefore, the construction of~\cite[Section~4.3]{DK3} endows the bundle~$\cW\vert_{B \ssm E} \cong (\bcA_3)^\vee$
with a quadratic form  and, by~\cite[Proposition~4.7]{DK3}, 
its cokernel sheaf is isomorphic to the cointersection sheaf~$\cC(\bcA_1,\bcA_2)$,
which is itself  isomorphic to~$\cC(\cA_1,\cA_2)$ 
by the argument of~\cite[Proposition~4.5]{DK3}.

It remains to show that the associated family of quadrics in~$\cW\vert_{B \ssm E}$ 
coincides with~$\cQ$ over~$B \ssm E = B_5 \ssm \{\bp_X\}$.\ 
Take~$[U_5] \in B_5 \ssm \{\bp_X\}$, let~$U_4 = U_5 \cap V_5$, and let~$v_0 \in U_5 \ssm U_4$.\ 
By~\cite[Proposition~3.13]{DK1}, the quadratic form~$\bq(v_0)$ on~$W = W_0$
can be obtained by the above construction applied to the Lagrangian subspaces
\begin{equation*}
A \subset \bw3V_6,
\qquad 
\bw3V_5 \subset \bw3V_6, 
\qquad 
v_0 \wedge \bw2V_5 \subset \bw3V_6
\end{equation*}
after isotropic reduction with respect to the isotropic subspace~$W_0^\perp = A \cap \bw3V_5$, 
the fiber of~$\cI_1$ at~$[U_5]$.\ 
Therefore, performing another reduction with respect to the fiber~$\bw3U_4$ of~$\cI_2$,
we see that this quadric restricts to~$\cQ_b$ on the fiber~$\cW_b \subset W$ at~$b = (U_4,U_5)$.\ 
Furthermore, the triple of Lagrangian subspaces obtained in this way 
coincides with the triple~$(\bcA_1,\bcA_3,\bcA_2)$, which implies the required coincidence.

Finally, to identify the double covers, we note that by~\cite[Theorems~3.1 and~4.2]{DK3}, 
the isomorphism of sheaves~$\cC_X\vert_{B \ssm E} \cong \cC(\cA_1,\cA_2)\vert_{B_5 \ssm \{\bp_X\}}$
implies an isomorphism of the corresponding double covers
and, by~\cite[Theorem~5.2]{DK3}, the double cover corresponding to the sheaf~$\cC(\cA_1,\cA_2)$ 
is the covering~$\uvt_\Ap$ from Theorem~\ref{thm:tya}.
\end{proof}

For the rest of this section, we assume that~$X$ is ordinary, so that~$W = W_0  \subset \bw2V_5$.\ 
In the next two lemmas, we study the corank stratification~$E^{\ge c}$ for~$\cQ_E$  in that case.
 
\begin{lemm}
\label{lemma:q-e}
If~$X$ is an ordinary GM variety, we have an isomorphism
\begin{equation*}
\cQ_E \cong \Fl(2,4;V_5) \times_{\P(\sbw2V_5)} \P(W_0),
\end{equation*}
where the fiber product is taken with respect to the map~$\Fl(2,4;V_5) \to \Gr(2,V_5) \hookrightarrow \P(\bw2V_5)$.
\end{lemm}

\begin{proof}
For~$b = (U_4,U_5) \in E$, we have~$U_5 = V_5$, 
hence~$\cQ_b$ is cut out in~$\P(\cW_b) \subset \P(\bw2V_5)$ by the Pl\"ucker quadrics.\ 
Therefore, $\cQ_b = \Gr(2,V_5) \cap \P(\cW_b)$.\ 
Since~$\cW_b = \bw2U_4 \cap W_0$ by definition, 
it follows that
\begin{equation*}
\cQ_b = \Gr(2,V_5) \cap \P(\bw2U_4) \cap \P(W_0) = \Gr(2,U_4) \cap \P(W_0). 
\end{equation*}
Finally, since~$\Gr(2,U_4)$ is the fiber of the projection~$\Fl(2,4;V_5) \to \Gr(4,V_5)$ over~$[U_4]$,
this equality is precisely what is claimed in the lemma.
\end{proof}

In the next lemma, we use the linear map~$\tilde\upkappa \colon \Sym^2(W_0^\perp) \to V_5$ defined in Corollary~\ref{cor:kappa}
and the (projectivized) corank stratification~$ \P(\sD^{\ge \bullet}_{W_0^\perp})\subset \P(\Sym^2(W_0^\perp)^\vee)$ of the space of quadrics.
 
\begin{lemm}
\label{lemma:e-stratification}
If~$X$ is an ordinary GM variety of dimension~$n$ satisfying Property~\eqref{hh},
the stratification~$E^{\ge \bullet} \subset E = \P(V_5^\vee)$ defined in~\eqref{newe} is obtained 
by pulling back the corank stratification of the space of quadratic forms~$\Sym^2(W_0^\perp)^\vee$ by the linear map
\begin{equation*}
\tilde\upkappa^\vee \colon V_5^\vee \lra \Sym^2(W_0^\perp)^\vee.
\end{equation*}
More precisely,
\begin{itemize}
\item 
if~$n = 5$, we have~$E^{\ge 1} = \emptyset$;
\item 
if $n = 4$, the scheme~$E^{\ge 1} \subset \P(V_5^\vee)$ 
is a hyperplane and~$E^{\ge 2} = \emptyset$;
\item 
if $n = 3$, the scheme~$E^{\ge 1} \subset \P(V_5^\vee)$ 
is the cone with vertex~$\P(\Ker(\tilde\upkappa^\vee)) \cong \P^1$ over a smooth conic,
$E^{\ge 2} = \P(\Ker(\tilde\upkappa^\vee)) = \Sing(E^{\ge 1})$,
and~$E^{\ge 3} = \emptyset$;
\item 
if $n = 2$, the scheme~$E^{\ge 1}$ is a general hyperplane section~$\P(V_5^\vee) \cap \P(\sD^{\ge 1}_{W_0^\perp})$
of the symmetric determinantal cubic fourfold~$\P(\sD^{\ge 1}_{W_0^\perp}) \subset \P(\Sym^2(W_0^\perp)^\vee)$,
$E^{\ge 2} = \Sing(E^{\ge 1})$ is the smooth rational quartic curve~$\P(V_5^\vee) \cap \P(\sD^{\ge 2}_{W_0^\perp})$, and~$E^{\ge 3} = \vide$.
\end{itemize}
In all cases, $E^c \subset E$ is empty or smooth of codimension $c(c+1)/2$ and~$E^{\ge c}$ is Cohen--Macaulay and normal.
\end{lemm}

\begin{proof}
Consider the diagram 
\begin{equation*}
\xymatrix{
0 \ar[r] &
\cW(H_4) \ar[r] &
(\bw2\cU_4)(H_4) \ar[r] \ar[d]^-{\rotatebox{90}{\raisebox{-0.7ex}[0ex][0ex]{$\sim$}}} &
(W_0^\perp)^\vee \otimes \cO_E(H_4) \ar[r] & 
0
\\
0 \ar[r] &
W_0^\perp \otimes \cO_E \ar[r] &
\bw2\cU_4^\vee \ar[r] &
\cW^\vee \ar[r] & 
0
}
\end{equation*}
of vector bundles on~$E \cong B_4$, where both rows come from the definition of~$ \cW$ in Lemma~\ref{lemma:cw-locally-free} 
and the vertical arrow is the canonical isomorphism.\ 
The composition~\mbox{$\cW(H_4) \to \cW^\vee$} constructed from the diagram 
coincides by Lemma~\ref{lemma:q-e} with the restriction to the divisor~$E$ 
of the map~$\eta \colon \cW_B(-E) \to \cW_B^\vee $ from~\eqref{equation-sheaf-cf} 
(note that~\mbox{$H_5\vert_E \lin 0$}, hence~$E\vert_E \lin -H_4$).\ 
It follows that its cokernel is isomorphic to the cokernel of the map
\begin{equation*}
W_0^\perp \otimes \cO_E \lra (W_0^\perp)^\vee \otimes \cO_E(H_4)
\end{equation*}
constructed using the inverse of the vertical arrow in the diagram.\ 
Clearly, this morphism is induced by the linear map~$\tilde\upkappa^\vee \colon V_5^\vee \to \Sym^2(W_0^\perp)^\vee$,
hence the stratification of~$E = \P(V_5^\vee)$ is induced by the pullback of the corank stratification of~$\Sym^2(W_0^\perp)^\vee$.\ 

The explicit descriptions of the strata can be deduced from Corollary~\ref{cor:kappa} (see also Lemma~\ref{lem:tsd}).\ 
The smoothness and codimensions of the locally closed subsets~$E^c$ are clear by inspection of the above cases;
the Cohen--Macaulay property and normality of~$E^{\ge c}$ are also clear in all  cases.
\end{proof}

Since the stratification of~$E$ is induced by a family of quadrics, there are induced double coverings of~$E^{\ge c}$.\ 
We describe one of them below;
we use the notation of Lemma~\ref{lemma:e-stratification}.

\begin{lemm}
\label{lem:tek-covers}
If~$n = 2$ and~$c = 1$, the double covering of~$E^{\ge 1}$ induced by the family of quadrics~$\cQ_E \to E$ is
 \begin{equation*}
\Fl(1,2;3) \lra \P(V_5^\vee) \cap \P(\sD^{\ge 1}_{W_0^\perp}),
\end{equation*}
that is, the covering from Proposition~\textup{\ref{propsing}\ref{it:sing-y1}}.
\end{lemm}

\begin{proof}
This follows from~\eqref{defdc} by the argument of Proposition~\ref{propsing}\ref{it:sing-y1}.
\end{proof}

We can now combine Lemma~\ref{lemma:ker-qb} with Lemma~\ref{lemma:e-stratification}
to describe the stratification~$B^{\ge c} \subset B$ defined in~\eqref{eq:bgec}.\ 
Recall  the equality~$B = \Bl_{\bp_X}(\P(V_6^\vee) ) $  {from} Lemma~\ref{ccb}.

\begin{prop}
\label{proposition:b-stratification}
Let~$X$ be an ordinary GM variety satisfying Property~\eqref{hh} 
and let~\mbox{$\bp_X$} be its Pl\"ucker point, defined in~\eqref{eq:plucker-point}.\ 
For any~$c \ge 0$, we have
\begin{equation*}
B^{\ge c} = \Bl_{\bp_X}(\sY_\Ap^{\ge c}).
\end{equation*}
For~$c \in \{0,1,2\}$, the subscheme~$B^c \subset B$ is smooth of codimension $c(c+1)/2$,
while the scheme~$B^{\ge 3} = \sY_\Ap^{\ge 3} \ssm \{\bp_X\}$ is finite, reduced, and disjoint from~$E$, and~$B^{\ge 4} = \vide$.\

Finally, for~$c  \in \{0,1,2\}$, the scheme~$B^{\ge c}$ is Cohen--Macaulay, normal, and integral,
and its subscheme~$E^{\ge c} \subset B^{\ge c}$ is a Cartier divisor.
\end{prop}

\begin{proof}
When~$c = 0$, there is nothing to prove, 
and for~$c \ge 3$, all claims follow immediately from Lemmas~\ref{lemma:ker-qb} and~\ref{lemma:e-stratification},
so we assume~$c \in \{1,2\}$.

A combination of Lemmas~\ref{lemma:ker-qb} and~\ref{lemma:e-stratification} shows that~$\codim(B^{\ge c}) \ge c(c+1)/2$.\ 
As~$B^{\ge c}$ is a symmetric determinantal locus, it follows from~\cite[Theorem~1]{Kutz} that~$B^{\ge c}$ is Cohen--Macaulay of codimension~$c(c+1)/2$.

Furthermore, $E \subset B$ is a Cartier divisor and~$E^{\ge c} = B^{\ge c} \cap E$ has codimension~1 in~$B^{\ge c}$, 
hence it is a Cartier divisor as well.\ 
Therefore, since~$E^c$ is smooth by Lemma~\ref{lemma:e-stratification}, we conclude that~$B^c$ is smooth along~$E^c$,
and by Theorem~\ref{thm:og} it is also smooth away from~$E^c$.\ 
Thus, it is smooth of codimension~$c(c+1)/2$ in~$B$.\ 
The same argument shows that~$\Sing(B^{\ge c}) \subset B^{\ge c + 1}$;
since~$c \in \{1,2\}$ this has codimension~2 or~3 in~$B^{\ge c}$, 
hence~$B^{\ge c}$ is normal by Serre's criterion.

Since~$E^{\ge c} \subset B^{\ge c}$ is a Cartier divisor, its complement is dense in~$B^{\ge c}$ 
and, since we have~$B^{\ge c} \ssm E^{\ge c} = \sY_\Ap^{\ge c} \ssm \{\bp_X\}$ (Lemma~\ref{lemma:ker-qb})
and~$\sY_\Ap^{\ge c} \ssm \{\bp_X\}$ is irreducible (Theorem~\ref{thm:og}), 
we deduced that~$B^{\ge c}$ is irreducible.\  
Using the Cohen--Macaulay property, we conclude that it is integral.

Finally, since~$B^{\ge c}$ is closed, it contains the closure of the preimage of~$\sY^{\ge c}_\Ap \ssm \{\bp_X\}$,
which is equal to~$\Bl_{\bp_X}(\sY_\Ap^{\ge c})$, and since~$B^{\ge c}$ is integral, it is equal to~$\Bl_{\bp_X}(\sY_\Ap^{\ge c})$.
\end{proof}

\begin{rema}
\label{rem:strata-special}
If~$X$ is special, one can check that the equality
\begin{equation*}
B^{\ge c} = \Bl_{\bp_X}(\sY_\Ap^{\ge c}) \cup (\Bl_{\bp_X}(\sY_\Ap^{\ge c - 1}) \cap E) 
\end{equation*}
holds; it follows from Lemma~\ref{lem:b-strata-special} and some general properties of Fitting ideals.
\end{rema}


\section{Orthogonal Grassmannian}
\label{sec:og}

Consider the vector bundle~$\cW_B$ of rank~$n+1$ on $B$ and the family of quadrics $\cQ \to B$ defined in~\eqref{eq:defQ}.\ 
In this section, we study the $B$-scheme~$\OGr_B(p, \cQ)$ that  parameterizes vector subspaces of dimension~$p$ in the fibers of the vector bundle~$\cW_B$ 
that are isotropic with respect to the quadratic form~\eqref{eq:iota} of the family of quadrics~$\cQ/B$.\  
In other words, 
\begin{equation*}
\OGr_B(p,\cQ) = \F_{p-1}(\cQ/B)
\end{equation*}
is the relative Hilbert scheme of linear spaces of dimension~\mbox{$p-1$} 
contained in the fibers of the family of quadrics~$\cQ/ B$.\ 
In the subsequent sections, when applying the present results  to the study of the schemes~$\G_k(X)$, 
we will take~$p = k + 2$. 

The following description of~$\OGr_B(p,\cQ)$ as a subscheme of~$B  \times \Gr(p,W)$ is quite useful.

\begin{lemm}
\label{lem:ogrb-triples}
 For any positive integer~$p$, we have  {an equality of schemes}
\begin{equation*}
\OGr_B(p,\cQ) = 
\left\{  (U_4,U_5,R_{p}) \in  B  \times \Gr(p,W) 
\,\left|\,
\begin{array}{l}
U_5 \subset \Ker(V_6 \to \Sym^2\!R_{p}^\vee)\\[.5ex]
R_{p} \subset \cW_{[U_4]}
 \end{array}
\right.\right\},
\end{equation*}
where the scheme structure of the right side 
is that of the zero locus of the morphisms
\begin{equation*}
 {\cU_5 \hookrightarrow V_6 \otimes \cO \xrightarrow{\ \bq\ } \Sym^2\!W^\vee \otimes \cO \twoheadrightarrow \Sym^2\!\cR_{p}^\vee
\qquad\text{and}\qquad
\cR_p \hookrightarrow W \otimes \cO \twoheadrightarrow (W \otimes \cO)/\cW_B}
\end{equation*}
of vector bundles.
\end{lemm}

\begin{proof}
The proof is straightforward.\ 
 {First, the zero locus of~$\cR_p \to (W \otimes \cO)/\cW_B$ is equal to the relative Grassmannian~$\Gr_B(p, \cW_B) \subset B \times \Gr(p,W)$.\ 
Moreover, by Lemma~\ref{lemma:u4-vanishes}, on this relative Grassmannian, the morphism~$V_6 \otimes \cO \to \Sym^2\!\cR_{p}^\vee$ 
factors through the quotient bundle~$(V_6 \otimes \cO) / \cU_4$ 
and, by definition of the quadric bundle~$\cQ \subset \P_B(\cW_B)$, the zero locus of the induced morphism~$\cU_5/\cU_4 \to \Sym^2\!\cR_{p}^\vee$ 
coincides with~$\OGr_B(p,\cQ)$.}
 \end{proof}

In this section, we study the natural morphism
\begin{equation*}
 f \colon \OGr_B(p,\cQ) \lra B.
\end{equation*}
The main  {results are Propositions~\ref{prop:corank-n-k} and~\ref{prop:corank-special}; they describe the Stein factorization of~$f$ and the geometric properties of~$\OGr_B(p, \cQ)$.}\ 
The complexity parameter
\begin{equation}
\label{eq:ell-p}
\ell \coloneqq 2p - n - 1
\end{equation} 
will play a crucial role.\ 
 {For~$p = k + 2$, we have~$\ell = 2k + 3 - n$, as in the introduction.}

\subsection{The map~$f$}
\label{subsec:f}

We first study the fibers of~$f$.\ 
 {If~$\cQ_b$ is the fiber of~$\cQ$ over a point~$b \in B$}, we have
\begin{equation*}
\OGr_B(p,\cQ)_b = \OGr(p, \cQ_b),
\end{equation*}
where the right side is the isotropic Grassmannian for the (possibly degenerate) quadric~$\cQ_b$.\ 
To explain its structure,
we denote by~$\OGr(p,n+1)$ the orthogonal Grassmannian parameterizing isotropic vector subspaces of dimension~$p$ 
in an $(n+1)$-dimensional vector space endowed with a \emph{nondegenerate} quadratic form.\ 
We will use the following easy observation.

\begin{rema}
\label{rem:ogr}
The variety~$\OGr(p,n+1)$ is nonempty if and only if~$n \ge 2p - 1$; in this case, it is a smooth homogeneous variety of dimension
\begin{equation}
\label{eq:nnp}
N(n,p) \coloneqq p(n+1-p) - \tfrac12p(p+1).
\end{equation} 
  {Moreover, if~$n \ge 2p$, it is connected, and if~$n = 2p - 1$, it  has two connected components;}
when~$p \in \{1,2,3\}$, these components are a point~$\P^0$, a line~$\P^1$, and a space~$\P^3$, respectively.
\end{rema}

The next lemma explains what happens if we allow the quadratic form to degenerate.

\begin{lemm}
\label{lem:ogr-dim}
Let~$Q \subset \P^n$ be a quadric of corank~$c$ and let~$p \ge  {1}$.\ 
Let~$\ell = 2p - n - 1$ as in~\eqref{eq:ell-p}.\ 
The isotropic Grassmannian~$\OGr(p,Q)$ is nonempty if and only if~$c \ge \ell$, in which case
\begin{equation*}
\dim(\OGr(p,Q)) = N(n,p) +  {\updelta(c,\ell)},
\end{equation*}
where~{$N(n,p)$ is defined in~\eqref{eq:nnp} and} the excess  {dimension}~$\updelta(c,\ell)$ is equal to
\begin{equation*}
  {\updelta(c,\ell) = \max \Big\{0,  \tfrac12 \left\lfloor \tfrac{c+\ell}2 \right\rfloor \left( \left\lfloor \tfrac{c+\ell}2 \right\rfloor + 1 \right) \Big\}.}
\end{equation*} 
\end{lemm}

\begin{rema}
 {A few values of~$\updelta(c,\ell)$ are listed in the following table:
  \begin{equation*}
\label{table}
\renewcommand\arraystretch{1.2}
\begin{array}{c|rrrrrrrr}
\ell & \ -3 & -2 & -1 & \hphantom{-}0 & \hphantom{-}1 & \hphantom{-}2 & \hphantom{-}3 & \hphantom{-}4 \\\hline
c = 0 & 0 & 0 & 0 & 0 & \cellcolor{lightgray} & \cellcolor{lightgray} & \cellcolor{lightgray} & \cellcolor{lightgray} \\
c = 1 & 0 & 0 & 0 & 0 & 1 & \cellcolor{lightgray} & \cellcolor{lightgray} & \cellcolor{lightgray} \\
c = 2 & 0 & 0 & 0 & 1 & 1 & 3 & \cellcolor{lightgray} & \cellcolor{lightgray} \\
c = 3 & 0 & 0 & 1 & 1 & 3 & 3 & 6 & \cellcolor{lightgray} \\
\end{array}
\end{equation*}
Gray cells in the table correspond to the cases where~$c < \ell$ (hence~$\OGr(p,Q)$ is empty).}
\end{rema}

\begin{proof}
Let~$K \subset \C^{n+1}$ be the kernel of~$Q$, so that~$\dim(K) = c$
and the induced quadratic form on~$\C^{n+1}/K$ is nondegenerate.\ 
For each~$0 \le i \le c$, consider the locally closed subset 
\begin{equation*}
\OGr(p,Q)^i \coloneqq \{ [R_p] \in \OGr(p,Q) \mid \dim(R_p \cap K) = i \}.
\end{equation*}
On the one hand, these subsets form a    stratification
\begin{equation*}
 \OGr(p,Q) = \bigsqcup_{0 \le i \le c} \OGr(p,Q)^i.
\end{equation*}
On the other hand, the linear projection~$\C^{n+1} \to \C^{n+1}/K$ induces a locally trivial fibration
\begin{equation*}
 \OGr(p,Q)^i \to \OGr(p-i,{n+1-c})
\end{equation*}
whose fiber is an open Schubert cell in~$\Gr(p, p+c-i)$.\ 
Using Remark~\ref{rem:ogr}, we see that the stratum~$\OGr(p,Q)^i$ is nonempty if and only if~$n - c \ge 2(p-i) - 1$, 
 that is, if and only if
\begin{equation}
\label{eq:ell-inequality}
2i - c \ge \ell.
\end{equation} 
Since~$\max_{0 \le i \le c}(2i - c) = c$, we conclude that~$\OGr(p,Q)$ is nonempty if and only if~$c \ge \ell$.\ 

Moreover, if~\eqref{eq:ell-inequality} holds, $\OGr(p,Q)^i$ is smooth  of dimension
\begin{equation*}
N(n-c,p-i) + p(c-i) = N(n,p) + i(\ell + c + \tfrac12 - \tfrac32i).
\end{equation*}
Therefore, if~$c \ge \ell$, it follows that~$\dim(\OGr(p,Q)) \ge N(n,p)$
and it is easy to compute the excess 
dimension~$\updelta( {c,\ell}) = \max_{0 \le i \le c,\ i \ge (c+\ell)/2}(i(\ell + c + \tfrac12 - \tfrac32i))$.
\end{proof}

We now apply the results of~\cite[Section~3.3]{DK3}, 
where isotropic Grassmannians for families of quadrics were studied.\ 
We start with a general result.

\begin{lemm}
\label{lem:quadrics}
Let~$\cQ \to S$ be a flat family of $(n-1)$-dimensional quadrics over a smooth scheme~$S$,
 {let~$S^{\ge c}$ be its corank~$c$ degeneracy loci, and set as usual~$S^c = S^{\ge c} \ssm S^{\ge c+1}$.
Let~$1 \le p \le n$, let~$f \colon \OGr_S(p,\cQ) \to S$ be the natural morphism, and assume}
 \begin{equation*}
\ell \coloneqq 2p - n - 1 \ge 0.
\end{equation*}
Assume also that~$S^\ell \ne \vide$ and~$\codim_S(S^\ell) = \updelta(\ell,\ell) = \ell(\ell + 1)/2$.
\begin{enumerate}[label={\textup{(\alph*)}}]
\item 
\label{it:ogrs-cm}
If~$\codim_S(S^c) \ge \updelta(c,\ell)$ for all~$c \ge \ell + 1$, then~$\OGr_S(p,\cQ)$ is Cohen--Macaulay and
\begin{equation*}
\dim(\OGr_S(p,\cQ)) = N(n,p) + \dim(S).
\end{equation*}

\item 
\label{it:ogrs-dense}
If~$\codim_S(S^c) \ge \updelta(c,\ell) + 1$ for all~$c \ge \ell + 1$, then~$f^{-1}(S^\ell)$ is dense in~$\OGr_S(p,\cQ)$.

\item 
\label{it:ogrs-normal}
Assume  {the following conditions~\ref{it:codim0}, \ref{it:codim2}, and either~\ref{it:codim1-1} or~\ref{it:codim1-2} hold, where}
\begin{enumerate}[label={\textup{(\roman*)}}]
\item
\label{it:codim0}
$S^\ell$ is smooth with~$\codim_S(S^\ell) = \ell(\ell + 1)/2$, 

\item 
\label{it:codim2}
$\codim(S^c) \ge \updelta(c,\ell) + 2$ for all~$c \ge \ell + 2$,
\item
\label{it:codim1-1}
$\codim(S^{\ell + 1}) \ge \ell(\ell + 1)/2 + 2$,

\item
\label{it:codim1-2}
 {$S^{\ell + 1}$ is smooth with~$\codim(S^{\ell + 1}) \ge  {(\ell + 1)(\ell + 2)/2}$.}
 \end{enumerate}\smallskip
Then~$\OGr_S(p,\cQ)$ is normal and the Stein factorization of~$f$ is
\begin{equation}
\label{eq:stein-s}
\OGr_S(p,\cQ) \xrightarrow{\  {\tf}\ } \tS^{\ge \ell} \xrightarrow{\  {\uvt_\cQ}\ } S^{\ge \ell} \lhra S,
\end{equation}
where~$\uvt_\cQ$ is a double covering branched over~$S^{\ge \ell + 1}$,
the scheme~$\tS^{\ge \ell}$ is normal, 
and~$\tf$ is an \'etale-locally trivial fibration over~$\tS^{\ge \ell} \ssm \uvt_\cQ^{-1}(S^{\ge \ell + 2})$
whose fibers for~$n - 2 \le p \le n$ are isomorphic to~$\P^{\frac12(n-p)(n-p+1)}$.
If~$\tS^{\ell}$ is connected, $\OGr_S(p,\cQ)$ is integral.

\item
\label{it:ogrs-smooth}
If~$S^c$ is smooth with~\mbox{$\codim_S(S^c) = \tfrac12c(c+1)$} for all~$c \ge \ell$,
then~$\OGr_S(p,\cQ)$ is nonsingular.
 \end{enumerate}
\end{lemm}

\begin{proof}
\ref{it:ogrs-cm}
Applying Lemma~\ref{lem:ogr-dim} to the fibers of the morphism~$f^{-1}(S^c) \to S^c$, we find
\begin{equation*}
\dim(f^{-1}(S^c)) = \dim(S^c) + N(n,p) + \updelta(c,\ell).
\end{equation*}
If~$\codim_S(S^c) \ge \updelta(c,\ell)$, this does not exceed the expected dimension~$\dim(S) + N(n,p)$ of~$\OGr_S(p,\cQ)$,
hence~$\OGr_S(p,\cQ)$ is Cohen--Macaulay of dimension~$\dim(S) + N(n,p)$.

\ref{it:ogrs-dense}
Under these assumptions, we have~$\dim(f^{-1}(S^c)) < \dim(S) + N(n,p)$ for all~$c \ge \ell + 1$,
so the Cohen--Macaulay property implies that~$f^{-1}(S^\ell)$ is dense in~$\OGr_S(p,\cQ)$.

\ref{it:ogrs-normal}
Since~$\OGr_S(p,\cQ)$ is Cohen--Macaulay by~\ref{it:ogrs-cm}, to prove normality, it is enough to check it is smooth in codimension~$1$.\ 
By assumption~\ref{it:codim2}, the codimension of~$f^{-1}(S^c)$ is at least~2 for~$c \ge \ell + 2$, 
so it is enough to show that~$f^{-1}(S^{\ge \ell} \ssm S^{\ge \ell + 2})$ is nonsingular  {in codimension~1.\
If~\ref{it:codim0} and~\ref{it:codim1-1} hold,~$f^{-1}(S^\ell)$ is nonsingular 
by~\cite[Lemmas~3.6 and~3.9]{DK3} and~$\codim(f^{-1}(S^{\ell + 1})) \ge 2$,
and if~\ref{it:codim0} and~\ref{it:codim1-2} hold,~$f^{-1}(S^{\ge \ell} \ssm S^{\ge \ell + 2})$ is nonsingular 
by~\cite[Lemmas~3.6 and~3.9]{DK3}.\ 
In either case, we see that~$\OGr_S(p,\cQ)$ is normal.}

Consider now the Stein factorization~\eqref{eq:stein-s} of~$f$.\ 
Since~$\OGr_S(p,\cQ)$ is normal, the scheme~$\tS^{\ge \ell}$ is also normal.\ 
Moreover, the proof of Lemma~\ref{lem:ogr-dim} shows that for~$s \in  {S^{\ge \ell} \ssm S^{\ge \ell + 2}}$,
the only nontrivial stratum in the fiber~$\OGr(p,\cQ_s)$ of the morphism~$f$ over~$s$
is the one with~\mbox{$i = c$}, and this stratum is isomorphic to the homogeneous variety
\begin{equation*}
\OGr(p - c, n + 1 - c) = 
\begin{cases}
\OGr(n + 1 - p, 2(n + 1 - p))  & \text{if~$c = \ell$,}\\
\OGr(n - p, 2(n - p) + 1) & \text{if~$c = \ell + 1$.}
\end{cases}
\end{equation*}
It is connected in the second case and has two connected components in the first case.\ 
Therefore, over~$s \in  {S^{\ge \ell} \ssm S^{\ge \ell + 2}}$, the finite morphism~$\uvt_\cQ$ in~\eqref{eq:stein-s}
is a double covering branched over~$S^{\ell + 1}$.\ 
Note also that if~$p = n$, the right side is two or one points;
if~$p = n - 1$, it is two or one copies of~$\P^1$;
and if~$p = n - 2$, it is two or one copies of~$\P^3$.\ 
Thus, in each of these cases, the fibers of~$\tf$ are isomorphic to~$\P^{\frac12(n-p)(n-p+1)}$.

If~$\tS^\ell$ is connected,~$f^{-1}(S^\ell)$ is connected and smooth, hence irreducible.\ 
Since  it is dense in~$\OGr_S(p,\cQ)$ by~\ref{it:ogrs-dense}, it follows that~$\OGr_S(p,\cQ)$ is irreducible
and, since~$\OGr_S(p,\cQ)$ is Cohen--Macaulay by~\ref{it:ogrs-cm}, it is integral.

Let us show that~$\tf$ is smooth over the locally closed subset~$S^{[\ell,\ell+1]} \coloneqq S^{\ge \ell} \ssm S^{\ge \ell + 2}$.\ 
Smoothness over~$S^\ell$ is obvious, so consider a point~$s_0 \in S^{\ell+1}$.\ 
The question is local around~$s_0$, so we may assume that the family of quadrics is given by a symmetric matrix~$q$ of size~$n+1$ of functions.\ 
The corank of~$q(s_0)$ is~$\ell + 1 = 2p - n$, 
hence its rank is~$n+1-(\ell+1) = 2n - 2p + 1$.\ 
Let~$K_{2p-n} \subset \C^{n+1}$ be the kernel of~$q(s_0)$.\ 
For any point~$( {R_{p}},s_0) \in \OGr_S(p,\cQ)$ over~$s_0$, we have~$K_{2p-n} \subset  {R_{p}}$.\ 
It follows that there is a direct sum decomposition
\begin{equation*}
\C^{n+1} = F'_{2n-2p} \oplus F''_{2p -n + 1}
\end{equation*}
which is orthogonal with respect to~$q(s_0)$ and such that the restriction of~$q(s_0)$ to the first summand is nondegenerate, 
$K_{2p-n} \subset F''_{2p-n+1}$,
and~$ {R_{p}} = R'_{n-p} \oplus K_{2p-n}$ with~\mbox{$R'_{n-p} \subset F'_{2n-2p}$}.\ 
Clearly, the quadratic form~$q(s)$ remains nondegenerate on~$F'_{2n-2p}$ in a neighborhood~$S_0$ of~$s_0$ in~$S$.\ 
 Let~\mbox{$\cF''_{2p - n + 1} \subset \cO_{S_0}^{\oplus (n+1)}$} be the orthogonal complement 
 of~${\cF'_{ {2n-2p}} \coloneqq {}} F'_{2n-2p} \otimes \cO_{S_0}$ with respect to~$q(s)$; 
then,
\begin{equation*}
\cO_{S_0}^{\oplus (n+1)} = 
 {\cF'_{2n-2p}} \oplus \cF''_{2p - n + 1} 
 \end{equation*}
is an orthogonal direct sum.\ 
Let~$\cQ' \subset \P_{S_0}( {\cF'_{2n-2p}})$ and~$\cQ'' \subset \P_{S_0}(\cF''_{2p-n+1})$ be the corresponding quadric bundles over~$S_0$.\ 
Choosing \'etale-locally around~$s_0$ a section~$\cR'_{n-p}$ of~$\OGr_{S_0}(n-p,\cQ')$  {passing through the point~$R'_{n-p}$
(since~$\cQ' \to S_0$ is a smooth morphism,   such a section exists)},
one can identify~$\cQ''$ with a hyperbolic reduction of~$\cQ$, 
hence the double covers over~$S_0^{[\ell,\ell+1]}$ associated with these quadric bundles are \'etale-locally isomorphic 
 {by~\cite[Proposition~1.1 and Corollary~1.5]{K21}}.\ 
Moreover, over~$S_0^{[\ell,\ell+1]}$, the rank of~$\cQ''$ 
is equal to~$2$ over~$S_0^\ell$ and~$1$ over~$S_0^{\ell+1}$, 
hence the corresponding double cover coincides with~$\OGr_{S_0^{[\ell,\ell+1]}}(2p - n,\cQ'')$.\ 
Thus,
\begin{equation*}
\OGr_{S_0^{[\ell,\ell+1]}}(2p - n,\cQ'') \cong \tS_0^{[\ell,\ell+1]} \coloneqq S_0^{[\ell,\ell+1]} \times_{S^{\ge \ell}} \tS^{\ge \ell} 
\end{equation*}
and we conclude that the embedding
\begin{equation*}
\tS_0^{[\ell,\ell+1]} \cong \OGr_{S_0^{[\ell,\ell+1]}}(2p - n,\cQ'') \xrightarrow{\ \cR'_{n-p} \oplus -\ } \OGr_{S_0^{[\ell,\ell+1]}}(p,\cQ)
\end{equation*}
provides a section of the morphism~$\OGr_{S_0^{[\ell,\ell+1]}}(p,\cQ) \to \tS_0^{[\ell,\ell+1]}$.\ 
By construction, this section passes through the point~$( {R_{p}},s_0)$.\ 
 {Since also the fiber of~$\tf$ over~$s_0$ is smooth, we conclude that}~$\tf$ is smooth at that point.\ 
Since the point~$(R_p,s_0)$ was arbitrary, $\tf$ is smooth over~$S^{[\ell, \ell + 1]}$.

Since, over~$S^{[\ell, \ell + 1]}$, the morphism~$\tf$ is smooth and all its fibers are homogeneous varieties~$\OGr(n - p, 2(n - p) + 1)$, 
it is \'etale-locally trivial over~$S^{[\ell, \ell + 1]}$ (see~\cite[Proposition~4]{Demazure}).

\ref{it:ogrs-smooth}
By~\cite[Lemma~3.6]{DK3}, the quadric fibration~$\cQ/S$ is regular,
and, by~\cite[Lemma~3.9]{DK3},~$\OGr_S(p,\cQ)$ is nonsingular.
 \end{proof}

In the next proposition, we summarize basic properties of the schemes~$\OGr_B(p,\cQ)$ and~$\OGr_E(p,\cQ_E)$
and of the canonical morphism~$f \colon \OGr_B(p,\cQ) \to B$.\ 
We denote by
\begin{equation}
\label{eq:ogre}
\OGr_{B \ssm E}(p,\cQ) \coloneqq \OGr_B(p,\cQ) \times_B (B \ssm E),
\qquad 
 \OGr_E(p,\cQ_E) \coloneqq \OGr_B(p,\cQ) \times_B E
\end{equation}
 {the relative isotropic Grassmannians 
of the quadric fibrations~$\cQ \times_B (B \ssm E) \to B \ssm E$ and~$\cQ_E \to E$, respectively.}\ 
The double covering~\mbox{$\uvt_\Ap \colon \wtY_\Ap^{\ge \ell} \to \sY_\Ap^{\ge \ell}$} was defined in Theorem~\ref{thm:tya}.\ 
If~\mbox{$\bp_X \in \sY^{\ge \ell}_\Ap$}, the subscheme~\mbox{$\widetilde\bp_X \subset \wtY_\Ap^{\ge \ell}$} over~$\bp_X$ 
and the double cover~$\tuvt_\Ap$ were defined in Lemma~\ref{lem:closure}; 
if~\mbox{$\bp_X \notin \sY^{\ge \ell}_\Ap$}, set~$\widetilde\bp_X = \vide$ and~$\tuvt_\Ap = \uvt_\Ap$.

\begin{prop}
\label{prop:corank-n-k}
Let~$X$ be a smooth GM variety of dimension~$n$  {satisfying Property~\eqref{hh}}, 
let~\mbox{$ {2} \le p \le n$}, and  {assume~$\ell \coloneqq 2p - n - 1 \ge 0$}.\ 
 The morphism~$f \colon \OGr_B(p,\cQ) \to B$ then factors through~$B^{\ge \ell} \subset B$.\ 
Furthermore,
 \begin{aenumerate}
\item 
\label{it:l4}
If~$\ell \ge 4$  {and~$X$ is arbitrary}, then~$\OGr_B(p,\cQ) = \vide$.

\item 
\label{it:l3}
If~$\ell = 3$  {and~$X$ is arbitrary}, we have a disjoint union  {of open and closed subschemes}
\begin{equation*}
\OGr_B(p,\cQ) = \OGr_{B \ssm E}(p,\cQ) \sqcup \OGr_E(p,\cQ_{E}) .
\end{equation*}
The target of the morphism $f \colon \OGr_{B \ssm E}(p,\cQ) \to  {(B \ssm E)^3 = {}} \sY^3_\Ap \ssm \{\bp_X\}$ 
is a finite reduced scheme and its fibers are all isomorphic: 
to a point  if~$p = 3$ and~$n = 2$, 
to two reduced points if~$p = 4$ and~$n = 4$,
and to two copies of~$\P^1$ if~$p = 5$ and~$n = 6$.

\item 
\label{it:ogr-bme}
If~$\ell \le 2$ and~$X$ is arbitrary, $\OGr_{B \ssm E}(p,\cQ)$ 
is a normal integral Cohen--Macaulay variety of dimension~$N(n,p) + 5$, 
nonsingular over the complement of~\mbox{$\sY^3_\Ap \ssm \{\bp_X\}$}.\ 
Moreover,  the morphism~$f \colon \OGr_{B \ssm E}(p,\cQ) \to B \ssm E$ has a factorization
\begin{equation}
\label{eq:stein-bme}
\OGr_{B \ssm E}(p,\cQ) \xrightarrow{\ \tf\ } 
\wtY^{\ge \ell}_\Ap \ssm  {\uvt^{-1}_\Ap(\bp_X)} \xrightarrow{\ \uvt_\Ap\ } 
\sY^{\ge \ell}_\Ap \ssm \{\bp_X\} \lhra 
B_5 \ssm \{\bp_X\} \isom 
B \ssm E,
\end{equation}
and~$\tf$ is an \'etale-locally trivial fibration with fibers~$\P^{\frac12(n-p)(n-p+1)}$ 
over the complement of~$\sY^{\ge \ell + 2}_\Ap$.

\item 
\label{it:l2-ord}
If~$\ell \le 2$ and~$X$ is ordinary, $\OGr_B(p,\cQ)$
is a normal integral Cohen--Macaulay variety of dimension~$N(n,p) + 5$, 
nonsingular over the complement of~\mbox{$\sY^3_\Ap \ssm\{ \bp_X\}$}.\ 
Moreover, the morphism~$f \colon \OGr_B(p,\cQ) \to B$ has a factorization
\begin{equation*}
 \OGr_B(p,\cQ) \xrightarrow{\ \tf\ } 
\Bl_{\widetilde\bp_X}(\wtY^{\ge \ell}_\Ap) \xrightarrow{\ \tuvt_\Ap\ } 
 \Bl_{\bp_X}(\sY^{\ge \ell}_\Ap) \lhra 
B,
\end{equation*}
and~$\tf$ is an \'etale-locally trivial fibration with fibers~$\P^{\frac12(n-p)(n-p+1)}$ 
over the complement of~$\Bl_{{\bp_X}}( {\sY}^{\ge \ell + 2}_\Ap)$.\ 
Finally, if~$\bp_X \in \sY^{\ge \ell}_\Ap$ \textup(this condition is equivalent to~$p \le 3$\textup),
then~$\OGr_E(p,\cQ_{E}) \subset \OGr_B(p,\cQ)$ is a smooth Cartier divisor; otherwise it is empty.
\end{aenumerate}
\end{prop}

The case where~$0 \le \ell \le 2$ and~$X$ is special   will be discussed in Section~\ref{se52}.

\begin{proof}
By Lemma~\ref{lem:ogr-dim}, the fiber of~$f$ over a point~$b \in B^c$ is empty when~$c < \ell$, 
hence~$f$ factors through the subscheme~$B^{\ge \ell}$.\

\ref{it:l4}
Assume~$\ell \ge 4$.\ 
Since~$B^{\ge 4} = \vide$ (by Proposition~\ref{proposition:b-stratification} and  {Lemma~\ref{lem:b-strata-special}}),
 the image of~$f$ is empty, hence~$\OGr_B(p,\cQ) = \vide$.\

\ref{it:l3}
Assume~$\ell = 3$.\ 
Applying Proposition~\ref{proposition:b-stratification}, we obtain
\begin{equation*}
\OGr_{B \ssm E}(p,\cQ) = f^{-1}(B^{\ge 3} \ssm E) = f^{-1}(\sY^{\ge 3}_\Ap \ssm \{\bp_X\}).
\end{equation*}
Since~{$\sY^{\ge 3}_\Ap$ is finite}, $\sY^{\ge 3}_\Ap \ssm \{\bp_X\} = \sY^3_\Ap \ssm \{\bp_X\}$ is    open and closed in~$\sY^{\ge 3}_\Ap$, 
 {hence}~$\OGr_{B \ssm E}(p,\cQ)$ is open and closed in~$\OGr_B(p,\cQ)$, 
hence we obtain the required disjoint union decomposition.\ 
The fibers of the morphism~$f \colon \OGr_{B \ssm E}(p,\cQ) \to \sY^3_\Ap \ssm \{\bp_X\}$ are all isomorphic 
because the quadric fibration~$\cQ/B$ has constant rank over its finite reduced target.

\ref{it:ogr-bme}
Applying Proposition~\ref{proposition:b-stratification}, Lemma~\ref{lem:b-strata-special}, and Theorem~\ref{thm:og},
we see that the hypotheses of Lemma~\ref{lem:quadrics}\ref{it:ogrs-normal} are satisfied by the family of quadrics~$\cQ$ over~$B \ssm E$, 
hence~$\OGr_{B\ssm E}(p,\cQ)$ is a normal Cohen--Macaulay scheme of dimension~$N(n,p) + 5$.\ 
We also obtain the Stein factorization through a normal double cover of~$\sY^{\ge \ell}_\Ap \ssm \{\bp_X\}$.
Moreover, by ~\cite[Proposition~3.10]{DK3} and Lemma~\ref{lemma:ker-qb}, 
this double cover is isomorphic over the open subset~$\sY^{\ell}_\Ap \ssm \{\bp_X\}$ to the restriction of~$\uvt_\Ap$;
by normality, this isomorphism extends over~$\sY^{\ge \ell}_\Ap \ssm \{\bp_X\}$, 
hence the Stein factorization has the form~\eqref{eq:stein-bme}.\ 
Since~$\wtY^{\ge \ell}_\Ap$ is connected by Theorem~\ref{thm:tya}, 
we see that~$\OGr_{B\ssm E}(p,\cQ)$ is integral.\ 
We also obtain \'etale-local triviality and the required description of the fibers of~$\tf$ over the complement of~$\sY^{\ge \ell + 2}_\Ap$.

Finally, the hypotheses of Lemma~\ref{lem:quadrics}\ref{it:ogrs-smooth} are satisfied over~$B \ssm (B^3 \cup E)$,
hence the scheme~$\OGr_{B\ssm E}(p,\cQ)$ is nonsingular over the complement of~$\sY^3_\Ap \cup \{\bp_X\}$.

\ref{it:l2-ord}
Assume now that~$X$ is ordinary and~$\ell \le 2$.\ 
By Proposition~\ref{proposition:b-stratification}, the hypotheses of Lemma~\ref{lem:quadrics}{\ref{it:ogrs-normal}}
are then satisfied by the family of quadrics~$\cQ$ over~{$B$};
 therefore, the arguments of part~\ref{it:ogr-bme} work over the entire~$B$
(we use Lemma~\ref{lem:closure}, which describes the normal closure of~$\Bl_{\bp_X}(\sY^{\ge \ell}_\Ap)$
in the field of rational functions of~$\wtY^{\ge \ell}_\Ap$).

Finally, if~$\bp_X \notin \sY^{\ge \ell}_\Ap$, we have~$E^{\ge \ell} = \vide$ by Proposition~\ref{proposition:b-stratification},
hence~$\OGr_E(p,\cQ_E)$ is empty as well.\ 
Otherwise, by Lemma~\ref{lemma:e-stratification}, the hypotheses of Lemma~\ref{lem:quadrics}{\ref{it:ogrs-smooth}} are verified for~$\cQ_E$,
hence~$\OGr_E(p,\cQ_E)$ is nonsingular  of expected dimension~$N(n,p) + 4$.\ 
In particular, we see that~$\OGr_E(p,\cQ_E)$ is a divisor in~$\OGr_B(p,\cQ)$.\ 
Since~$E \cap B^3 = \vide$ (again by Lemma~\ref{lemma:e-stratification}), 
this divisor is contained in the smooth locus of~$\OGr_B(p,\cQ)$, hence it is a Cartier divisor.
\end{proof}

\begin{rema}
\label{rema:l-neg}
In Proposition~\ref{prop:corank-n-k}, we only considered the case~$\ell \ge 0$, because it is geometrically more interesting.\ 
If~$\ell < 0$ (this case appears for the Hilbert squares of GM varieties of dimension~$n \ge 4$ 
and for the Hilbert schemes of conics on GM sixfolds), 
the same arguments prove the first statements in parts~\ref{it:ogr-bme} and~\ref{it:l2-ord} of Proposition~\ref{prop:corank-n-k};
however, in this case,~$f$ is surjective with connected general fiber, so the Stein factorization becomes trivial.
\end{rema}

 {A} useful consequence of Proposition~\ref{prop:corank-n-k} is the following.

\begin{coro}
\label{cor:ogre-component}
 Let~$X$ be a  GM variety satisfying Property~\eqref{hh}.\ 
For any integer $p \ge 4$, the scheme~$\OGr_B(p,\cQ)$ is a disjoint union
\begin{equation}
\label{eq:ogrb-union}
\OGr_B(p,\cQ) = \OGr_{B \ssm E}(p,\cQ) \sqcup \OGr_E(p,\cQ_E)
\end{equation}
of closed subschemes.
 \end{coro}

\begin{proof}
 If~$p \ge 4$, we have~$\ell  = 2p - 1 - n  \ge 7 - n$,  {so~\eqref{eq:px-dim} implies}~$\bp_X \notin \sY^{\ge \ell}_\Ap$.\ 
By Lemma~\ref{lemma:ker-qb},  
\begin{equation*}
B^{\ge \ell} = (B \ssm E)^{\ge \ell} \sqcup E^{\ge \ell} = \sY^{\ge \ell}_\Ap \sqcup E^{\ge \ell}
\end{equation*}
is a disjoint union of closed subschemes.\ 
Taking the preimage under~$f$, we obtain~\eqref{eq:ogrb-union}.
\end{proof}

\subsection{The map~$f$ for special GM varieties}\label{se52}

 In Proposition~\ref{prop:corank-n-k}, we described the morphism~$f \colon \OGr_B(p,\cQ) \to B$ for ordinary GM varieties;
here, we discuss the case of special GM varieties.\ 
The situation becomes more complicated   because, as we noted in Lemma~\ref{lem:b-strata-special},
the corank stratification of~$\cQ \to B$ is shifted on~$E$ 
and, as a result, the conditions of Lemma~\ref{lem:quadrics} are not satisfied on~$E$
and the subscheme~$\OGr_E(p,\cQ_E) \subset \OGr_B(p,\cQ)$  {may become} an irreducible component of higher dimension.\ 
 {In this section,} we describe the remaining ``main'' component
\begin{equation*}
\label{eq:def-ogrm}
\OGrm_B(p,\cQ) \coloneqq 
\overline{\OGr_{B \ssm E}(p,\cQ)} \subset 
\OGr_B(p,\cQ) \subset 
\Gr_B(p,\cW_B)
\end{equation*}
 of~$\OGr_B(p,\cQ)$  and its intersection
 \begin{equation*}
\label{eq:def-ogrme}
\OGrm_E(p,\cQ_E)  \coloneqq 
\overline{\OGr_{B \ssm E}(p,\cQ)} \cap \Gr_E(p,\cW_B\vert_E)  \subset 
\OGr_E(p,\cQ_E) \subset 
\Gr_E(p,\cW_B\vert_E)
\end{equation*}
with the preimage of the Cartier divisor~$E \subset B$.

It follows from Corollary~\ref{cor:ogre-component} that for~$p \ge 4$, we have~$\OGrm_B(p,\cQ) = \OGr_{B \ssm E}(p,\cQ)$, 
which was already described in Proposition~\ref{prop:corank-n-k}\ref{it:ogr-bme},
while~$\OGrm_E(p,\cQ_E) = \vide$.\ 
 {Furthermore, as we explained in the introduction, $\OGr_B(p,\cQ)$ is used to describe the Hilbert scheme~$\G_k(X)$ with~$k = p - 2$, hence}
the case~$p \le 1$ has no geometric meaning, while the case~$p = 2$ corresponds to  Hilbert squares of surfaces, which we only consider for ordinary GM surfaces.\ 
 Therefore, throughout this  section, we assume~$p = 3$.\ 
In this case,~{$\ell = 2p - 1 - n = 5 - n$, so}~\eqref{eq:px-dim}  implies~$\bp_X \in \sY^{  \ell + 1}_\Ap$.

 {The main result of this section is} the following analog of Proposition~\ref{prop:corank-n-k}\ref{it:l2-ord}.

\begin{prop}
\label{prop:corank-special}
Let~$X$ be a smooth special GM variety of dimension~${n \in \{3,4,5\}}$,
let~{\mbox{$p = 3$}},  {so that~$\ell = 5 - n \in \{0,1,2\}$}.\ 
Then~$\OGrm_{B \ssm E^{\ge \ell + 2}}({3},\cQ)$ is a normal integral Cohen--Macaulay variety of dimension~$N(n, {3}) + 5 = 3n - 7$, 
nonsingular over the complement of the finite subscheme~\mbox{$\sY^3_\Ap \ssm\{ \bp_X\}$} of~$B\ssm E$.\ 
Moreover, the morphism~$f \colon \OGrm_{B \ssm E^{\ge \ell + 2}}( {3},\cQ) \to B$  factors as
 \begin{equation*}
  \OGrm_{B \ssm E^{\ge \ell + 2}}( {3},\cQ) \xrightarrow{\ \tf'\ } 
\Bl_{\widetilde\bp_X}(\wtY^{\ge \ell}_\Ap) \ssm \tuvt^{-1}_\Ap(E^{\ge \ell + 2}) \xrightarrow{\ \tuvt_\Ap\ } 
\Bl_{\bp_X}(\sY^{\ge \ell}_\Ap) \ssm E^{\ge \ell + 2} \lhra 
B 
\end{equation*}
and the map~$\tf'$ is an \'etale-locally trivial fibration with fibers~$\P^{ {\frac12(n-3)(n-2)}}$ 
over the complement of~\mbox{${\sY^{\ge \ell + 2}_\Ap} \cup E^{\ge \ell + 2}$};
 {in particular, $f$ is ramified along~$\OGrm_{E \ssm E^{\ge \ell + 2}}( {3},\cQ_{E})$,
which is a smooth Cartier divisor in~$\OGrm_{B \ssm E^{\ge \ell + 2}}( {3},\cQ)$}.
\end{prop}

The proof of this proposition is quite technical and takes up the most part of this  section.\ 
The main idea of the proof is to relate~$\OGrm_B( {3},\cQ)$ and~$\OGrm_E({3},\cQ_E)$
to the relative isotropic Grassmannian for a family of quadrics obtained from~$\cQ \to B$
by a Hecke transformation along the divisor~$E$, analogous to the one used in~\cite[Section~2.5]{DK4}.\ 
However, this cannot be applied literally, because the line bundle~$\cO_{ {B}}(E)$ is not a square.\ 
One way to solve this would be to base change to the root stack of~$B$.\ 
In order to keep the exposition more elementary, 
we instead base change to a double covering~$S \to B$ ramified over~$E$, which exists locally.\ 
Then we apply a Hecke transformation to the base change~$\cQ_S \to S$ of~$\cQ \to B$,
obtain in this way a family~$\cQ'_S \to S$ that satisfies the hypotheses of Lemma~\ref{lem:quadrics},
and   relate the scheme~$\OGr_S({3},\cQ'_S)$ to~$\OGrm_B({3},\cQ)$.

So, let~$X$ be a smooth special GM variety of dimension~{$n \in \{3,4,5\}$ and let~$\ell = 5 - n$}.\ 
Let~$X_0$ be the ordinary GM variety associated with~$X$ by the bijection of Theorem~\ref{thm:gm-ld}.\ 
We denote  {by~$\cQ_0 \to B$ the quadric fibration~\eqref{eq:defQ} associated with~$X_0$,
and} by~$B_0^{\ge c}$ and~$E_0^{\ge c}$ the closed strata
 of the respective corank stratifications of~$B$ and~$E$, 
 so that, by~\eqref{defwk1:special} and Lemma~\ref{lem:b-strata-special}, we have
\begin{equation*}
\cW_X = \cO_{B} \oplus \cW_{X_0},
\qquad
(B \ssm E)^{\ge c} = (B \ssm E)_0^{\ge c},
\qquad\text{and}\qquad 
E^{\ge c} = E_0^{\ge c - 1}.
\end{equation*}
 {In particular, $E^{\ge \ell + 2} = E_0^{\ge \ell + 1}$.}

Let~$e \in E \ssm E^{\ge \ell + 2} = E \ssm E_0^{\ge \ell + 1}$.\ 
Over a neighborhood of~$e$ in~$B \setminus  {(B^{\ge 3} \cup E_0^{\ge \ell + 1})}$,
we can find a double covering
\begin{equation*}
\vp \colon S \to B \setminus  {(B^{\ge 3} \cup E_0^{\ge \ell + 1})} \hra B
\end{equation*}
which is branched only over~$E$ 
{(as we will see in the proof of Lemma~\ref{lem:strata-cqp}, excluding~$B^{\ge 3}$ will allow us to simplify the argument)}.\
We denote by~$\iota \colon S \to S$ the involution of the double covering and by~$E_S \subset S$ the ramification divisor of~$\varpi$.\ 
Note that~$E_S$ is $\iota$-invariant, while its equation is $\iota$-antiivariant.

Consider the pullback~$\cQ_S \to S$ of the family~$\cQ \to B$, so that we have
\begin{equation*}
\cQ_S  \coloneqq \cQ \times_B S  \subset \P_S(\cO_S \oplus \vp^*\cW_{X_0}).
 \end{equation*}
 {The equation of~$\cQ_S$ is the pullback~$\varpi^*  {(\eta)}$ of the equation~\eqref{eq:iota} of~$\cQ$.}

By construction, $\varpi^*\cO_B(E) \cong \cO_S(2E_S)$.\ 
We now apply a Hecke transformation.

\begin{lemm}
\label{lem:cqps}
Consider the exact sequence 
\begin{equation}
\label{eq:cwp}
0 \to \cO_S \oplus \vp^*\cW_{X_0} \xrightarrow{\ (E_S,\id)\ } \cO_S(E_S) \oplus \vp^*\cW_{X_0} \xrightarrow{\quad} \cO_{E_S}(E_S) \to 0
\end{equation}
of vector bundles on~$S$.\ 
The equation~$\varpi^* {(\eta)} \colon \Sym^2(\cO_S \oplus \vp^*\cW_{X_0}) \to \varpi^*\cO_B(E)$ 
of the family of quadrics~$\cQ_S$ factors in a unique way as the composition
\begin{equation}
\label{eq:cwp2}
\Sym^2(\cO_S \oplus \vp^*\cW_{X_0}) \xrightarrow{\ \ }  \Sym^2(\cO_S(E_S) \oplus \vp^*\cW_{X_0}) \xrightarrow{\ \eta'\ } \cO_S(2E_S) \cong \varpi^*\cO_B(E).
\end{equation}
Moreover, $\eta'$ is the orthogonal direct sum of an isomorphism~$\Sym^2(\cO_S(E_S)) \isomto \cO_S(2E_S)$
 and the pullback~$\varpi^*(\eta_0) \colon \Sym^2(\vp^*\cW_{X_0}) \to \cO_S(2E_S)$ 
of the  {equation~$\eta_0 \colon \Sym^2 \cW_{X_0}  \to \cO_B(E)$ of~$\cQ_0$}.
 \end{lemm}

\begin{proof}
We apply~\cite[Lemma~2.11]{DK4} with~$\cE = \cO_S \oplus \vp^*\cW_{X_0}$, $D = 2E_S$, and~$\cK = \cO_D$
(the argument of Lemma~\ref{lem:b-strata-special} shows that~$\cK$ is contained in the kernel of~$\varpi^*(\eta)\vert_D$).\ 
\end{proof}

  {Consider now} the family of quadrics
\begin{equation*}
\cQ'_S \subset \P_S(\cO_S(E_S) \oplus \vp^*\cW_{X_0}) 
\end{equation*}
defined by the equation~$\eta'$ of~\eqref{eq:cwp2}.\ 
The uniqueness of~$\eta'$ implies that the involution~$\iota$ of the double covering~$\varpi \colon S \to B$
lifts to an involution of~$\cQ'_S$, which we denote by~$\iota_\cQ$.

\begin{lemm}
\label{lem:iota}
The involution~$\iota_\cQ$ on~$\cQ'_{E_S} \coloneqq \cQ'_S \times_S E_S$
is induced by the involution of the projective bundle~$\P_{E_S}(\cO_{E_S}(E_S) \oplus \vp^*\cW_{X_0}\vert_{E_S})$
that acts by multiplication by~$-1$ on the first summand and as the identity on the second.
 \end{lemm}

\begin{proof}
This follows from the exact sequence~\eqref{eq:cwp}, 
because its first term is $\iota$-invariant and the equation of~$E_S$ is $\iota$-antiinvariant.
\end{proof}

We denote by~$\cC'$ the cokernel sheaf of~$\cQ'_S$ and by~$S^{\ge \ell} \subset S$ its degeneracy loci, 
so that the~$S^\ell =  {S^{\ge \ell} \ssm S^{\ge \ell + 1}}$ are the strata of the corank stratification.

\begin{lemm}
\label{lem:strata-cqp}
We have~$\cC' \cong \varpi^*(\cC_{X_0})$ and~$S^{\ge \ell} = \varpi^{-1}(B_0^{\ge \ell})$;
 {in particular, $S^{\ge 3} = \vide$.}
\end{lemm}

\begin{proof}
By Lemma~\ref{lem:cqps}, the morphism~$\eta' \colon \cO_S(-E_S) \oplus \vp^*\cW_{X_0}(-2E_S) \to \cO_S(-E_S) \oplus \vp^*\cW^\vee_{X_0}$ 
is the direct sum of an isomorphism and~$\vp^*{(\eta_0)}$, hence its cokernel is the pullback of~$\cC_{X_0}$.
\end{proof}

We consider the isotropic Grassmannian~$f'_S \colon \OGr_S( {3},\cQ'_S) \to S$ associated with the family of quadrics~$\cQ'_S/S$
and denote by~$\OGr_{E_S}( {3}, \cQ'_{E_S})$ its base change to the divisor~$E_S \subset S$.\

\begin{coro}
\label{cor:ogrps}
The scheme~$\OGr_S( {3},\cQ'_S)$ is  {smooth and connected of dimension~$3n - 7$}.

Moreover, there is a commutative diagram 
 \begin{equation*}
\vcenter{\xymatrix@C=3em{
\OGr_S( {3},\cQ'_S) \ar[r]^-{\ \tf'_S\ } &
\tS^{\ge \ell} \ar[r]^-{\ \uvt_{\cQ'_S}\ } \ar[d] &
S^{\ge \ell} \ar@{^{(}->}[r] \ar[d] &
S \ar[d]^{\varpi} 
\\
& 
\Bl_{\widetilde\bp_X}(\wtY^{\ge \ell}_\Ap) \ar[r]^-{\tuvt_\Ap} &
\Bl_{\bp_X}(\sY^{\ge \ell}_\Ap) \ar@{^{(}->}[r] &
\,B,
}}
\end{equation*} 
where the upper row is the Stein factorization of~$f'_S$, 
the morphism~$\tf'_S$ is an \'etale-locally trivial fibration 
with fibers~$\P^{\frac12(n- {3})(n- {2})}$ over the complement of~$S^{\ge \ell + 2}$,
the right square is Cartesian, and the left square is Cartesian up to normalization.

Finally, $\OGr_{E_S}( {3}, \cQ'_{E_S}) \subset \OGr_S( {3},\cQ'_S)$ is a smooth Cartier divisor.
\end{coro}

\begin{proof}
By definition of the morphism~$\varpi$ and by Lemma~\ref{lem:strata-cqp}, 
the locus~$S^{\ge c}$ is, for~$c \in \{0,1,2\}$, a double covering 
(of an open subset) of~$B_0^{\ge c}$ ramified over~$E_0^{\ge c}$,
 {while~$S^{\ge 3} = \vide$}.\ 
Since~$B_0^c$ is smooth of codimension~$\tfrac12c(c+1)$ in~$B$ by Proposition~\ref{proposition:b-stratification},
and since~$E_0^c$ is a smooth Cartier divisor in~$B_0^c$ by Lemma~\ref{lemma:e-stratification},
we conclude that~$S^c$ is smooth of codimension~$\tfrac12c(c+1)$ in~$S$.\ 
Therefore, the hypotheses of Lemma~\ref{lem:quadrics}{\ref{it:ogrs-normal} and~\ref{it:ogrs-smooth}} 
are satisfied  {over~$B$} 
 and the proof of Proposition~\ref{prop:corank-n-k}\ref{it:ogr-bme} works.\ 
This proves the first statement.

Similarly, we have~$E_S^{\ge c} = \vp^{-1}(E_0^{\ge c}) \cap E_S \cong E_0^{\ge c}$,
hence  {the strata~$E_S^c$ are} smooth of codimension~$\tfrac12c(c + 1)$ in~$E_S$.\ 
The arguments of Proposition~\ref{prop:corank-n-k}\ref{it:l2-ord} 
 imply that~$\OGr_{E_S}( {3}, \cQ'_{E_S})$ is a smooth Cartier divisor in~$\OGr_S( {3},\cQ'_S)$.

It only remains to construct the Cartesian squares.\ 
The right square is obvious from Lemma~\ref{lem:strata-cqp} 
because~$B_0^{\ge \ell} = \Bl_{\bp_X}(\sY^{\ge \ell}_\Ap)$ by Proposition~\ref{proposition:b-stratification}.\ 
For the left square, recall that the morphism~$\varpi$ is \'etale over the complement of~$E$ 
so, since Stein factorization is stable under \'etale base change, 
we conclude that over the complement of~$E$, the left square exists and is Cartesian.\ 
Since the complement of~$E_S$ is dense in~$S^{\ge \ell}$, 
it follows that~$\tS^{\ge \ell}$ is the normalization of the fiber product of~$S^{\ge \ell}$ and~$\Bl_{\widetilde\bp_X}(\wtY^{\ge \ell}_\Ap)$.
 \end{proof}

We will also need the following lemma.

\begin{lemm}
\label{lem:rp3}
If~$e' \in E_S$ and if~$R'_3 \subset \C \oplus \cW_{X_0,\varpi(e')}$ is such that~$(R'_3,e') \in \OGr_{E_S}(3,\cQ'_{E_S})$, 
then~\mbox{$R'_3 \not\subset \cW_{X_0,\varpi(e')}$} and~$R'_3 \cap \cW_{X_0,\varpi(e')}$ is a maximal isotropic subspace in~$\cW_{X_0,\varpi(e')}$.
 \end{lemm}

\begin{proof}
If~$R'_{3} \subset \cW_{X_0,\varpi(e')}$, the quadric~$\cQ_{{0},\varpi(e')}$ of dimension~$n-2$ 
(the fiber of the quadratic fibration associated with~$X_0$) 
admits an isotropic subspace of dimension~$ {3}$.\ 
This means that the rank of this quadric is less than or equal to~$2(n -  {3})$, 
and the corank is greater than or equal to~$n - 2(n -  {3}) =  {6} - n = \ell + 1$, that is,~$\vp(e') \in E_0^{\ge \ell + 1}$.\ 
However, we have~\mbox{$\vp(S) \subset B \setminus E_0^{\ell + 1}$} by definition;
this contradiction proves  {the first part of the lemma}.

It also follows that~$R'_3 \cap \cW_{X_0,\varpi(e')}$ is a maximal isotropic subspace.
\end{proof}

It remains to relate~$\OGr_S( {3},\cQ'_S)$ to~$\OGrm_B( {3},\cQ) \subset \OGr_B( {3},\cQ)$.

\begin{lemm}
\label{lem:varphi}
There is a commutative diagram
\begin{equation}
\label{eq:ogr-diagram}
\vcenter{\xymatrix@C=4.em{
\OGr_S( {3},\cQ'_S) \ar[r]^-{\varphi} \ar[d]_{f'_S} & 
\OGr_S( {3},\cQ_S) \ar[r]^-{\varpi_{\OGr}} \ar[d]_{f_S} & 
\OGr_B( {3},\cQ) \ar[d]^f
\\
S \ar@{=}[r] &
S \ar[r]^-\varpi &
B,
}}
\end{equation}
where the right square is Cartesian 
 {and the left square is equivariant with respect to the involutions 
\begin{equation*}
\iota'_{\OGr} \colon \OGr_S( {3},\cQ'_S) \lra \OGr_S( {3},\cQ'_S)
\quad\text{{and}}\quad 
\iota_{\OGr} \colon \OGr_S( {3},\cQ_S) \lra \OGr_S( {3},\cQ_S)
\end{equation*}
induced by the involution~$\iota \colon S \to S$  over~$B$.\ 
Moreover, the involution~$\iota'_{\OGr}$ acts freely on~$\OGr_S( {3},\cQ'_S)$,}
the morphism~$\varphi$ is finite, birational, and unramified;
it is an isomorphism over~$S \ssm E_S$ and two-to-one over~$E_S$.\ 
Similarly, the morphism~$\varpi_{\OGr} \circ \varphi$ is unramified.
 \end{lemm}

The following picture illustrates the situation:
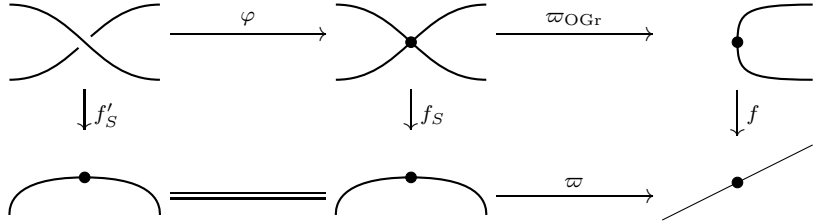
\begin{equation*}
\xymatrix@C=5em{
\begin{tikzpicture}[baseline=0]
\draw [thick] (0,-0.5) .. controls (0.9,-0.5) and (1.1,0.5) .. (2,0.5);
\filldraw[white] (1,0) circle (3pt);
\draw [thick] (0,0.5) .. controls (0.9,.5) and (1.1,-0.5) .. (2,-0.5);
\end{tikzpicture}
\ar[r]^{\varphi} \ar[d]^{f'_S} &
\begin{tikzpicture}[baseline=0]
\draw [thick] (0,-0.5) .. controls (0.9,-0.5) and (1.1,0.5) .. (2,0.5);
\filldraw[black] (1,0) circle (2pt);
\draw [thick] (0,0.5) .. controls (0.9,.5) and (1.1,-0.5) .. (2,-0.5);
\end{tikzpicture}
\ar[r]^{\vp_{\OGr}} \ar[d]^{f_S} &
\begin{tikzpicture}[baseline=0]
\draw [thick, white] (0,-0.5) .. controls (0.9,-0.5) and (1,-0.4) .. (1,0) .. controls (1,0.4) and (0.9,0.5) .. (0,0.5);
\draw [thick] (2,-0.5) .. controls (1.1,-0.5) and (1,-0.4) .. (1,0) .. controls (1,0.4) and (1.1,0.5) .. (2,0.5);
\filldraw[black] (1,0) circle (2pt);
\end{tikzpicture}
\ar[d]^{f}
\\
\begin{tikzpicture}[baseline=-10]
\draw [thick, white] (0,-0.5) .. controls (0.9,-0.5) and (1,-0.4) .. (1,0) .. controls (1,0.4) and (0.9,0.5) .. (0,0.5);
\draw [thick] (0,-0.5) .. controls (0,-0.4) and (0,0) .. (1,0) .. controls (2,0) and (2,-0.4) .. (2,-0.5);
\filldraw[black] (1,0) circle (2pt);
\end{tikzpicture}
\ar@{=}[r] &
\begin{tikzpicture}[baseline=-10]
\draw [thick, white] (0,-0.5) .. controls (0.9,-0.5) and (1,-0.4) .. (1,0) .. controls (1,0.4) and (0.9,0.5) .. (0,0.5);
\draw [thick] (0,-0.5) .. controls (0,-0.4) and (0,0) .. (1,0) .. controls (2,0) and (2,-0.4) .. (2,-0.5);
\filldraw[black] (1,0) circle (2pt);
\end{tikzpicture}
\ar[r]^{\vp} &
\begin{tikzpicture}[baseline = -8]
\draw (0,-0.5) -- (2,0.5);
\filldraw[black] (1,0) circle (2pt);
\end{tikzpicture}
}
\end{equation*}
The involutions act by symmetry about vertical axes.

\begin{proof}
The right Cartesian square of the diagram follows from the definition of~$\OGr_S( {3},\cQ_S)$.

To construct the left square, let~$\cR'_{ {3}} \subset (f'_S)^*(\cO_S(E_S) \oplus \vp^*\cW_{X_0})$ 
be the tautological subbundle of~$\OGr_S( {3},\cQ'_S)$ and consider the composition
\begin{equation}
\label{eq:crpp}
\cR'_{ {3}} \lhra (f'_S)^*(\cO_S(E_S) \oplus \vp^*\cW_{X_0}) \xrightarrow{\quad}  {(f'_S)^*}\cO_{E_S}(E_S)
\end{equation}
of the tautological embedding with  {the pullback of} the second arrow of~\eqref{eq:cwp}, 
that is, with the projection to the first summand restricted to~$E_S$.
If it vanishes at a point~$e'$ of~$E_S$ and if~$R'_3$ is the fiber of~$\cR'_3$ at~$e'$, 
we get~$R'_3 \subset \cW_{X_0,\varpi(e')}$, in contradiction with Lemma~\ref{lem:rp3}.\ 
Thus,  {the composition}~\eqref{eq:crpp} is surjective.

Let~$\cR_{ {3}} \subset (f'_S)^*(\cO_S(E_S) \oplus \vp^*\cW_{X_0})$ be the kernel of~\eqref{eq:crpp}, 
so that the sequence
\begin{equation}
\label{eq:crp-crpp}
0 \to \cR_{ {3}} \to \cR'_{ {3}} \to  {(f'_S)^*}\cO_{E_S}(E_S) \to 0
\end{equation} 
is exact.\ 
Using~\eqref{eq:cwp}, we see that~$\cR_{ {3}}$ 
is a vector subbundle of rank~$ {3}$ of~$ {(f'_S)^*}(\cO_S \oplus \varpi^*\cW_{X_0})$.\ 
Moreover, the definition of~$\eta'$ in Lemma~\ref{lem:cqps} 
implies that~$\cR_{ {3}}$ is isotropic with respect to~$\varpi^*(\eta)$, 
hence defines a morphism~$\OGr_S( {3},\cQ'_S) \to \OGr_B( {3},\cQ)$.\ 
Combining it with the projection to~$S$, we obtain a morphism
\begin{equation*}
\OGr_S( {3},\cQ'_S) \lra \OGr_B( {3},\cQ) \times_B S = \OGr_S( {3},\cQ_S) 
\end{equation*}
which we denote by~$\varphi$.\ 
This gives us the left square in the diagram.

The commutativity and equivariance of the left square are clear from the construction of~$\varphi$.\ 
 Moreover, the involution~$\iota$ acts freely on~$S \ssm E_S$,
hence~{$\iota'_{\OGr}$} acts freely on~$\OGr_S( {3}, \cQ'_S)$ over~$S \ssm E_S$.\ 
Over~$E_S$, the involution~$\iota$ is described in Lemma~\ref{lem:iota}.\ 
This description implies that if~$R'_{ {3}} \subset {(\cO_S(E_S) \oplus \vp^*\cW_{X_0})_{e'} = \C \oplus \cW_{X_0,\vp(e')}}$ 
is the subspace corresponding to a $\iota'_{\OGr}$-invariant point of~$\OGr_S( {3}, \cQ'_S)$ over~$e' \in E_S$,
then~$R'_{ {3}}$ must be compatible with the direct {sum decomposition}~$\C \oplus \cW_{X_0,\vp(e')}$.\ 
But the quadratic form~$\eta'$ is nondegenerate on the first summand by Lemma~\ref{lem:cqps},
hence~$R'_{ {3}}$ must be contained in the second summand, which is impossible  {by Lemma~\ref{lem:rp3}.}\ Thus, $\iota'_{\OGr}$ acts freely  {over~$E_S$,   hence everywhere}.

Since the morphism~$\cQ'_S \to \cQ_S$ is an isomorphism over~$S \ssm E_S$, the same is true for the morphism~$\varphi$.\ 
To describe~$\varphi$ over~$E_S$, 
note that the restriction to~$ {(f'_S)^{-1}(E_S)}$ of the exact sequence~{\eqref{eq:crp-crpp}}
 is the   
 exact sequence
\begin{equation*}
0 \to 
 {(f'_S)^*}\cO_{E_S} \to 
\cR_{{3}}\vert_{{(f'_S)^{-1}(E_S)}} \to 
\cR'_{ {3}}\vert_{{(f'_S)^{-1}(E_S)}} \to 
{(f'_S)^*}\cO_{E_S}(E_S) \to 
0.
\end{equation*}
The image of the middle arrow is an isotropic subbundle of rank~$ {2}$ in~${(f'_S)^*}\varpi^*\cW_{X_0}\vert_{E_S}$.\ 
This means that the restriction of~$\varphi$ to~$\OGr_{E_S}( {3}, \cQ'_{E_S})$ factors as
\begin{equation}
\label{eq:varphi-factorization}
\OGr_{E_S}( {3},\cQ'_{E_S}) \to \OGr_E( {2},\cQ_{0,E}) \to \OGr_E( {3},\cQ_E),
\end{equation}
where~$\cQ_{0,E}$ is the restriction to~$E$ of the family of quadrics~{$\cQ_0$} associated with~$X_0$.\ 

The first arrow in~\eqref{eq:varphi-factorization} takes 
an $\eta'$-isotropic $ {3}$-dimensional subspace~$R'_{ {3}} \subset \C \oplus \cW_{X_0,\vp(e')}$ 
to the intersection~$\bar{R}_{ {2}} \coloneqq R'_{ {3}} \cap \cW_{X_0,\vp(e')}$,
which is a maximal isotropic subspace by Lemma~\ref{lem:rp3}.\ 
 {In particular, $\bar{R}_2$ contains the kernel of~$\eta_{0,\vp(e')}$; 
therefore, the dimension of its orthogonal~$\bar{R}_{ {2}}^\perp$ in~$\cW_{X_0,\vp(e')}$ is~3.}\ 
 Thus, the space~$\C \oplus (\bar{R}_{ {2}}^\perp/\bar{R}_2 )$ is 2-dimensional 
and the quadratic form  induced on it by~$\eta'$ is nondegenerate;
this means that there are exactly two different $\eta'$-isotropic subspaces~$R'_{ {3}}$ that contain~$\bar{R}_{ {2}}$,
that is, the first arrow in~\eqref{eq:varphi-factorization} is \'etale of degree~2.

The second arrow in~\eqref{eq:varphi-factorization} takes an $\eta_0$-isotropic subspace~$\bar{R}_{ {2}} \subset \cW_{X_0,\vp(e')}$ of dimension~$ {2}$
(any such subspace is maximal isotropic by the argument  {of Lemma~\ref{lem:rp3}})
to the sum~\mbox{$R_{ {3}} \coloneqq \C \oplus \bar{R}_{ {2}}$}.\ 
Therefore, this arrow is an isomorphism.

It follows that~$\varphi$ is a quasi-finite morphism.\ 
Since~$\OGr_{E_S}( {3}, \cQ'_{E_S})$ is a smooth Cartier divisor in~$\OGr_S( {3}, \cQ'_S)$ 
equal to the preimage of the Cartier divisor~$E_S \subset S$
and the morphism~$\varphi$ is a morphism over~$S$,
it also follows that~$\varphi$ is unramified over~$E_S$.\ 
Finally, $\varphi$ is proper  because~$f_S$ and~$f'_S$ are,  (see~\eqref{eq:ogr-diagram}), hence~$\varphi$ is finite.

We now prove that~$\varpi_{\OGr} \circ \varphi$ is unramified.\ 
Over~$S \ssm E_S$, this is obvious.\ 
Moreover, the morphism~$\varpi$ induces an open embedding~$E_S \to E$,
hence~$\OGr_{E_S}( {3},\cQ_{E_S}) \to \OGr_E( {3},\cQ_E)$ is \'etale 
so, combining this with the \'etale morphism~$\OGr_{E_S}( {3},\cQ'_{E_S}) \to \OGr_{E_S}( {3},\cQ_{E_S})$ from~\eqref{eq:varphi-factorization},
we conclude that the restriction of~$\varpi_{\OGr} \circ \varphi$ to~$\OGr_{E_S}( {3},\cQ'_{E_S})$ is  {unramified}.\ 
Since~$\OGr_{E_S}( {3},\cQ'_{E_S})$ is a smooth Cartier divisor in~$\OGr_S( {3},\cQ'_S)$, 
it remains to show that for any point~$e' \in E_S$ and any~$[R'_{ {3}}] \in \OGr( {3},\cQ'_{e'})$,
 {there is a tangent vector}~$\uptau \in \rT_{\OGr_S( {3},\cQ'_S),   (e',R'_{ {3}})}$ 
which is not tangent to~$\OGr_{E_S}( {3},\cQ'_{E_S})$ 
  {and whose image} under the differential of~$\varpi_{\OGr} \circ \varphi$ is not tangent to~$\OGr_E( {3},\cQ_E)$.

 {To construct such a~$\uptau$}, we take a tangent vector~$\uptau_S \in \rT_{S,e'}$ in the kernel of the differential of~$\varpi$.\ 
We think of~$\uptau_S$ as a morphism~$\Spec(\C[\eps]/\eps^2) \to S$.\ 
Then~$\varpi \circ \uptau_S$ is constant (that is, factors through~$\Spec(\C)$), 
hence~$\uptau_S^*\cQ_S$ is a constant family of quadrics~$\eta = 0 \oplus \eta_0$.\ 
Moreover, the pullback of the equation of~$E_S$ is equal to~$\eps$.\ 
Now, applying the construction of Lemma~\ref{lem:cqps}, 
we see that~$\eta'$ is also a constant family of quadrics~$\eta' = 1 \oplus \eta_0$.\ 
Therefore, the constant family~$R'_{ {3}}[\eps]$ of isotropic subspaces 
gives a morphism~$\uptau \colon \Spec(\C[\eps]/\eps^2) \to \OGr_S( {3},\cQ'_S)$ lifting~$\uptau_S$.\ 
Clearly, $\uptau$ is a tangent vector  and it is enough to compute its image.

To do this, we choose a basis~$(r_1, {r_2,r_3})$ of~$R'_{ {3}}$ such that~$ {r_2,r_3} \in R'_{ {3}} \cap \cW_{X_0,\varpi(e')}$.\ 
Upon  rescaling, we can write~$r_1$ as~$(1,w)$, where again~$w \in \cW_{X_0,\varpi(e')}$.\ 
The condition that~$R'_{ {3}}$ be isotropic with respect to~$\eta'$ means that for~$ {i,j \in \{2,3\}}$, we have
\begin{equation*}
\eta_0(r_i,r_j) = 0,
\qquad 
\eta_0(w,r_i) = 0,
\qquad\text{and}\qquad
\eta_0(w,w) = -1.
\end{equation*}
The definition of the morphism~$\varphi$ then implies that the corresponding isotropic subspace~$R_{ {3}}$
is spanned by the vectors~$1 + \eps w$,  {$r_2$, and~$r_3$}.\ 
The image of~$\uptau$ in
\begin{equation*}
\rT_{\Gr( {3},\cW_{\varpi(e')}),[R_{ {3}}]} = \Hom(R_{ {3}},\cW_{\varpi(e')} / R_{ {3}})
\end{equation*}
then takes~$r_1$ to~$w$ and vanishes on~{$r_2$ and~$r_3$}.\ 
Since~$\eta_0(w,w) \ne 0$, this vector is not tangent to~$\OGr( {3}, \cW_{\varpi(e')})$, as required.
\end{proof}

We can now prove the proposition.

\begin{proof}[Proof of Proposition~\textup{\ref{prop:corank-special}}]
Since~$\OGr_S( {3},\cQ'_S)$ is smooth and its involution~$\iota'_{\OGr}$ is  fixed point free, 
the quotient~$\OGr_S( {3},\cQ'_S) / \iota'_{\OGr}$ is also smooth.\ 
Moreover, Lemma~\ref{lem:varphi} shows that the morphism~$\varpi_{\OGr} \circ \varphi$ 
induces a finite, birational, unramified, and bijective  {(onto a neighborhood of the point~$e$)} morphism
\begin{equation*}
\OGr_S( {3},\cQ'_S) / \iota'_{\OGr} \lra \OGrm_{B \ssm E_0^{\ge \ell + 1}}( {3}, \cQ_{B \ssm E_0^{\ge \ell + 1}}).
\end{equation*}
The first two properties imply that the morphism is the normalization of the target 
and the last two properties imply that it is an isomorphism.\ 
 Since~$e  $ was an arbitrary point of~$ E \ssm E^{\ell +2}$,
this proves all the properties of~$\OGrm_{B \ssm E^{\ge \ell + 2}}({3},\cQ)$ over a neighborhood of~$E \ssm E^{\ell +2}$ 
and, in combination with Proposition~\ref{prop:corank-n-k}\ref{it:ogr-bme}, proves the first part of the proposition.

Similarly, the  {above} morphism induces an isomorphism
\begin{equation*}
\OGr_{E_S}( {3},\cQ'_{E_S}) / \iota'_{\OGr} \isomto \OGrm_{E \ssm E_0^{\ge \ell + 1}}( {3}, \cQ_{E \ssm E_0^{\ge \ell + 1}}),
\end{equation*}
where the source is a smooth Cartier divisor in~$\OGr_S( {3},\cQ'_S) / \iota'_{\OGr}$,
hence the target is a smooth Cartier divisor in~$\OGrm_{E \ssm E_0^{\ge \ell + 1}}( {3}, \cQ_{E \ssm E_0^{\ge \ell + 1}})$.

 {The Stein factorization and its properties follow from Corollary~\ref{cor:ogrps} by taking the quotient with respect to~$\iota'_{\OGr}$;
the ramification along~$\OGrm_{E \ssm E_0^{\ge \ell + 1}}( {3}, \cQ_{E \ssm E_0^{\ge \ell + 1}})$ is obvious from the Stein factorization.}
\end{proof}

\begin{rema}
\label{rem:p3n3-special}
Consider the case where~$p =  n = 3$, so that~$X$ is a smooth special GM threefold.\ 
Then,~$\ell = 2$ and~$\bp_X \in \sY^3_\Ap$.\ 
Note that~$E_0^{\ge \ell + 1}  $ is empty by Lemma~\ref{lemma:e-stratification}, 
hence the results of Proposition~\ref{prop:corank-special} hold over the whole of~$B$.\ 
Moreover, the morphism~$\tf'$ in the Stein factorization is an isomorphism
{hence,} we have~$\OGrm_B( {3},\cQ) \cong \Bl_{\widetilde\bp_X}(\wtY_\Ap^{\ge 2})$,
where~$\widetilde\bp_X$,  {according to Lemma~\ref{lem:closure}}, is the point of~$\wtY^{\ge 2}_\Ap$ over~$\bp_X$.
\end{rema}


\section{Orthogonal Grassmannians vs.\ Hilbert schemes}
\label{sec:general-another}

In this section, we relate the orthogonal Grassmannian~$\OGr_B(k + 2,\cQ)$ studied in Section~\ref{sec:og} 
to the Hilbert scheme~$\G_k(X)$ of $k$-dimensional quadrics on a smooth GM variety~$X$.

 \subsection{{The degeneracy locus}}

Given a smooth GM variety~$X \subset \P(W)$  {of dimension~$n \ge 2$}
 and a nonnegative integer~$k$, 
we consider the Grassmannian~$\Gr(k + 2, W)  = \F_{k+1}(\P(W)) $ and its closed subscheme
\begin{equation*}
\chG_k(X) \coloneqq \{ R_{k+2} \subset W \mid 
 \rank(V_6 \xrightarrow{\ \bq\ } \Sym^2\!W^\vee \xrightarrow{\ \ } \Sym^2\!R_{k+2}^\vee ) \le 1 \} \subset \Gr(k + 2, W),
\end{equation*}
with the scheme structure defined by the corresponding Fitting ideal.

Consider also the Hilbert schemes~$\F_{k+1}(X)$ and~$\F_{k+1}(M_X)$ of~$(k+1)$-dimensional linear subspaces 
in~$X$ and its Grassmannian hull~$M_X$, respectively (see Sections~\ref{subsec:ls-qu}  {and~\ref{ss:fkx}}).\ 
The open subscheme~$\Fo_{k+1}(M_X)   \subset \F_{k+1}(M_X)$ was defined in Definition~\ref{def:fok}.\ 
All these schemes are subschemes of~$\F_{k+1}(\P(W)) = \Gr(k + 2, W)$.

\begin{lemm}
\label{lem:fkmx-dkx}
Let~$X$ be a smooth GM variety  {of dimension~$n \ge 2$} and let~$k$ be a nonnegative integer.\ 
 There are inclusions
\begin{equation*}
\F_{k+1}(X) \subset \F_{k+1}(M_X) \subset \chG_k(X)
\end{equation*}
of subschemes of~$\Gr(k+2,W)$.\ 
Moreover, if~$X$ is special, $\F_{k+1}(X) \subset \Fo_{k+1}(M_X)$.
\end{lemm}

\begin{proof}
The inclusion~$\F_{k+1}(X) \subset \F_{k+1}(M_X)$  {follows from the inclusion~$X \subset M_X$}.\ 
To prove the inclusion~$\F_{k+1}(M_X) \subset \chG_k(X)$, 
note that~$\P(R_{k+2}) \subset M_X$ implies~\mbox{$V_5 \subset \Ker(V_6 \to \Sym^2\!R_{k+2}^\vee)$}, 
hence we have~\mbox{$\rank(V_6 \to \Sym^2\!R_{k+2}^\vee ) \le 1$}.\ 
Finally, if~$X$ is special, so that~$M_X$ is a cone,~$X$ does not contain the vertex of~$M_X$, 
hence~$\F_{k+1}(X) \subset \Fo_{k+1}(M_X)$.
\end{proof}

We set 
\begin{equation*}
\chG^0_k(X) \coloneqq \chG_k(X) \ssm \F_{k+1}(M_X).
\end{equation*}
 We will use the notions of $\sigma$- and~$\tau$-quadrics introduced in Definition~\ref{def:st-quadrics}
and the corresponding closed subschemes~$\Gs_k(X) \subset \G_k(X)$ and~$\Gt_k(X) \subset \G_k(X)$, 
as well as the definition~\eqref{eq:g0k} of the open subscheme $\G^0_k(X) \subset \G_k(X) $.

{Using}~\eqref{eq:gk-pw}, we consider the composition
\begin{equation}
\label{eq:pik-composition}
\G_k(X) \hra
\G_k(\P(W)) \isomto
\P_{\Gr(k+2,W)}(\Sym^2\!\cR_{k+2}^\vee) \to
\Gr(k+2,W)
\end{equation}
that takes a quadric~$\Sigma   \subset X $ to its linear span~$\langle \Sigma \rangle$ in~$\P(W)$.\

\begin{prop}
\label{prop:pik}
Let~$X$ be a smooth GM variety and let~$k$ be a nonnegative integer.\ 
The composition~\eqref{eq:pik-composition} factors through a surjective morphism
 \begin{equation*}
\lambda_k \colon \G_k(X) \lra \chG_k(X),
\qquad 
\Sigma \longmapsto \langle \Sigma \rangle
\end{equation*}
such that
\begin{align}
\Gs_k(X) &= \lambda_k^{-1}(\Fs_{k+1}(M_X)),\nonumber\\
\Gt_k(X) &= \lambda_k^{-1}(\Ft_{k+1}(M_X)),\label{ggg}\\
\G^0_k(X) &= \lambda_k^{-1}(\chG^0_k(X)).\nonumber
\end{align}
Moreover, we have~$\lambda_k^{-1}(\F_{k+1}(X)) \cong \P_{\F_{k+1}(X)}(\Sym^2\!\cR_{k+2}^\vee)$ 
 {and~$\lambda_k$ induces an isomorphism}
\begin{equation*}
\lambda_k \colon 
\G_k(X) \ssm \P_{\F_{k+1}(X)}(\Sym^2\!\cR_{k+2}^\vee)
\isomto\chG_k(X) \ssm \F_{k+1}(X).
 \end{equation*}
  {In particular, $\lambda_k$} is an isomorphism over the open subscheme~\mbox{$\chG^0_k(X) \subset \chG_k(X)$}.
 \end{prop}

\begin{proof}
 Let~$\langle \Sigma \rangle \eqqcolon \P(R_{k+2})$.\ 
Since~$\Sigma \subset X$, the restriction to~$R_{k+2}$ of any quadratic form 
from the space~$\bq(V_6) \subset \Sym^2\!W^\vee$  of quadratic equations of~$X \subset \P(W)$
must be proportional to the equation of~$\Sigma$, 
hence the image of the map~$V_6 \xrightarrow{\ \bq\ } \Sym^2\!W^\vee \xrightarrow{\ \ } \Sym^2\!R_{k+2}^\vee$
is at most $1$-dimensional;
therefore~$[R_{k+2}] \in \chG_k(X)$, hence~\eqref{eq:pik-composition} factors through~$\chG_k(X)$.\ 

More precisely, let~{$\cQ_k(X) \subset \G_k(X) \times X$} be the universal quadric; 
it is the pullback along the morphism~$\G_k(X) \to \G_k(\P(W)) \cong \P_{\Gr(k+2,W)}(\Sym^2\!\cR_{k+2}^\vee)$ 
of the universal divisor of bidegree~$(1,2)$ in the fiber product
\begin{equation*}
\P_{\Gr(k+2,W)}(\Sym^2\!\cR_{k+2}^\vee) \times_{\Gr(k+2,W)} \P_{\Gr(k+2,W)}(\cR_{k+2}).
\end{equation*}
Denoting by~$\cL$ the pullback from~$\P_{\Gr(k+2,W)}(\Sym^2\!\cR_{k+2}^\vee)$ to~$\G_k(X)$ of the tautological line subbundle,
we conclude that~$\cQ_k(X)$ is a divisor in~$\P_{\G_k(X)}(\cR_{k+2})$ whose equation is a global section of~$\cL^\vee(2)$.\ 
 Since the morphism
\begin{equation*}
\cQ_k(X) \hookrightarrow 
\P_{\G_k(X)}(\cR_{k+2}) \to
\P_{\Gr(k+2,W)}(\cR_{k+2}) =
\Fl(1,k+2,W) \to 
\P(W)
\end{equation*}
  factors through~$X$, which is the zero locus of~$\bq \colon V_6 \otimes \cO \to \cO(2)$,
the pullback of~$\bq$ to~$\P_{\G_k(X)}(\cR_{k+2})$ must factor as
\begin{equation*}
V_6 \otimes \cO \to \cL \to \cO(2),
\end{equation*}
where the second arrow is the equation of~$\cQ_k(X)$.\ 
The morphism~$V_6 \otimes \cO \to \Sym^2\!\cR_{k+2}^\vee$ on~$\G_k(X)$ 
therefore  factors as~$V_6 \otimes \cO \to \cL \hookrightarrow \Sym^2\!\cR_{k+2}^\vee$,
which implies that~$\lambda_k(\G_k(X)) \subset \chG_k(X)$.

The equalities~\eqref{ggg} follow from the definition of~$\sigma$- and $\rtx$-quadrics  {(Definition~\ref{def:st-quadrics})}.
 
Assume now~$[R_{k+2}] \in \chG_k(X)$.\
If~$[R_{k+2}] \in \F_{k+1}(X)$, we have~$\P(R_{k+2}) \subset X$, 
hence any quadric in~$\P(R_{k+2})$ is contained in~$X$.\ This gives an isomorphism  
\begin{equation*}
\lambda_k^{-1}(\F_{k+1}(X)) \cong \P_{\F_{k+1}(X)}(\Sym^2\!\cR_{k+2}^\vee).
\end{equation*}
Furthermore, over~$\chG_k(X) \ssm \F_{k+1}(X)$,
the morphism~$V_6 \otimes \cO \to \Sym^2\!\cR_{k+2}^\vee$ has rank~$1$, 
hence its image is a line subbundle in~ $\Sym^2\!\cR_{k+2}^\vee$.\ 
This line subbundle defines  {a morphism
\begin{equation*}
\chG_k(X) \ssm \F_{k+1}(X) \to \P_{\Gr(k+2,W)}(\Sym^2\!\cR_{k+2}^\vee) 
\end{equation*}
which is obviously inverse to the restriction of~$\lambda_k$ to~$\G_k(X) \ssm \P_{\F_{k+1}(X)}(\Sym^2\!\cR_{k+2}^\vee)$.}
 \end{proof}

Combining Propositions~\ref{prop:pik} and~\ref{prop:fk-gm}, we obtain the following result.

\begin{coro}
\label{cor:pik}
Let~$X$ be a GM variety of dimension~$n$ satisfying Property~\eqref{hh} and let~$k$ be a nonnegative integer.\ 
If
\begin{itemize}
\item either $2k + 2 > n$, 
\item or~$2k + 2 = n  {{} \ne 4}$ and~$\sY^3_{A,V_5} = \vide$,
\item or $k =  {1}$, $n =  {4}$, and~$\sY^3_{A,V_5} =\sZ^4_{A,V_5} = \vide$,
\end{itemize}
then~{$\F_{k+1}(X) = \vide$ and}~$\lambda_k \colon \G_k(X) \to \chG_k(X)$ is an isomorphism.
\end{coro}

We will also need the following general observation.\ 
Recall that, by Lemma~\ref{lem:fkx-fkmx}, 
 {the complement of~$\F_{k+1}(X)$ in~$\F_{k+1}(M_X)$ is dense in~$\F_{k+1}(M_X)$.}
 
\begin{lemm}
\label{lem:fmx-strict-transform}
Let~$X$ be a GM variety of dimension~$n$ satisfying Property~\eqref{hh} and let~$k$ be a nonnegative integer.\ 
The strict transform  of the subscheme~$\F_{k+1}(M_X) \subset \chG_k(X)$ in~$\G_k(X)$ 
is isomorphic to~$\Bl_{\F_{k+1}(X)}(\F_{k+1}(M_X))$.
\end{lemm}

\begin{proof}
 {If~$k \ge 3$, then~$\F_{k+1}(X) = \vide$ and Corollary~\ref{cor:pik} implies the claim, so assume~\mbox{$k \le 2$}.}\ 
By definition of the embedding~$\F_{k+1}(M_X) \subset \chG_k(X)$  {(see Lemma~\ref{lem:fkmx-dkx})}, 
the restriction of the morphism~\eqref{eq:v6-s2r} to~$\F_{k+1}(M_X)$ factors as
\begin{equation*}
V_6 \otimes \cO_{\F_{k+1}(M_X)} \to (V_6 / V_5) \otimes \cO_{\F_{k+1}(M_X)} \to \Sym^2\!\cR_{k+2}^\vee,
\end{equation*}
where the right arrow is given by any non-Pl\"ucker quadric containing~$X$ and its zero locus is the subscheme~$\F_{k+1}(X) \subset \F_{k+1}(M_X)$.\ 
 {The dual of the second arrow is an epimorphism~$\Sym^2\!\cR_{k+2} \twoheadrightarrow \cI_{\F_{k+1}(X)}$ on~$\F_{k+1}(M_X)$;
it induces a closed embedding}
 \begin{equation*}
\Bl_{\F_{k+1}(X)}(\F_{k+1}(M_X)) \hookrightarrow 
\P_{\F_{k+1}(M_X)}(\Sym^2\!\cR_{k+2}^\vee) \subset 
\P_{\Gr(k+2,W)}(\Sym^2\!\cR_{k+2}^\vee)  {{} = \G_k(\P(W))}.
\end{equation*}
The open subset~$\F_{k+1}(M_X) \ssm \F_{k+1}(X)$ is dense in~$\F_{k+1}(M_X)$ by Lemma~\ref{lem:fkx-fkmx}
 and its image under the above map is contained in~$\G_k(X)$;
 therefore, the image of~$\Bl_{\F_{k+1}(X)}(\F_{k+1}(M_X))$ is equal to the strict transform of~$\F_{k+1}(M_X)$.
\end{proof}

Using this observation, we obtain a description of the schemes~$\Gs_k(X)$ and~$\Gt_k(X)$.

\begin{coro}
\label{cor:gstk}
Let~$X$ be a GM variety satisfying Property~\eqref{hh} and let~$k$ be a nonnegative integer.\ 
We have
\begin{equation*}
 \G^\star_k(X) = \Bl_{\F^\star_{k+1}(X)}(\F^\star_{k+1}(M_X)) \cup \P_{\F^\star_{k+1}(X)}(\Sym^2\!\cR_{k+2}^\vee),
\end{equation*}
where~$\star$ stands for~$\sigma$ or~$\tau$.\ 
 {In particular, if~$\F^\star_{k+1}(X) = \vide$, then~$\G^\star_k(X) = \F^\star_{k+1}(M_X)$.}
\end{coro}

\begin{proof}
 {By~\eqref{ggg}}, the scheme~$\G^\star_k(X)$ is the preimage of~$\F^\star_{k+1}(M_X)$ under the map~$\lambda_k$,
hence it is the union of the strict transform of~$\F^\star_{k+1}(M_X)$ and of the full transform of~$\F^\star_{k+1}(X)$,
described in Lemma~\ref{lem:fmx-strict-transform} and Proposition~\ref{prop:pik},  {respectively}.
\end{proof}

\subsection{{The morphism from the orthogonal Grassmannian}}

We now construct a morphism~$\OGr_B(k+2,\cQ) \to \chG_k(X)$.\ 
For this, using  {the embedding~$\cW_B \hookrightarrow W \otimes \cO_B$ obtained by pulling back}~\eqref{defwk2}, 
we consider the composition 
 \begin{equation}
\label{eq:def-pre-hpik}
\OGr_B(k+2,\cQ) \hookrightarrow 
\Gr_B(k+2,\cW_B) \hookrightarrow 
B \times \Gr(k+2,W) \to 
\Gr(k+2, W).
\end{equation}
The subvariety~$\OGr_E(k+2,\cQ_E) \subset \OGr_B(k+2,\cQ)$  was defined in~\eqref{eq:ogre}.

\begin{prop}
\label{prop:hpik}
Let~$X$ be a GM variety satisfying Property~\eqref{hh}.\ 
For any~$k \ge 0$, the composition~\eqref{eq:def-pre-hpik} factors through a morphism
\begin{equation}\label{defpik}
g_k \colon \OGr_B(k+2,\cQ) \lra \chG_k(X)
\end{equation}
such that
\begin{equation}
\label{eq:hpik-preimage}
g_k^{-1}(\F_{k+1}(M_X)) = g_k^{-1}(\F_{k+1}(X)) \cup \OGr_E(k+2,\cQ_E).
\end{equation}
Moreover,~$g_k$ induces  an isomorphism 
\begin{equation*}
g_k^{-1}(\chG^0_k(X)) \xrightiso{\ g_k\ } \chG^0_k(X) = \chG_k(X) \ssm \F_{k+1}(M_X).
\end{equation*}
 \end{prop}

\begin{proof}
 We gave in  Lemma~\ref{lem:ogrb-triples} a description of~$\OGr_B(k+2,\cQ)$.\ 
Since the composition 
\begin{equation*}
\cU_5 \lhra V_6 \otimes \cO_{\OGr_B(k+2,\cQ)} \xrightarrow{\ \bq\ }\Sym^2\!\cR_{k+2}^\vee
\end{equation*}
vanishes, the pullback~$g_k^*(\bq)$ factors through the line bundle~$(V_6 \otimes \cO_{\OGr_B(k+2,\cQ)})/\cU_5$.\ 
Therefore, $g_k^*(\bq)$ has rank at most~1 on~$\OGr_B(k+2,\cQ)$, hence~$g_k$ factors through~$\chG_k(X)$.

As the proof of Lemma~\ref{lem:fmx-strict-transform} shows, 
the kernel of the morphism~$V_6 \otimes \cO \to \Sym^2\!\cR_{k+2}^\vee$ on~$\F_{k+1}(M_X) \ssm \F_{k+1}(X)$
is equal to~$V_5 \otimes \cO$.\ 
Moreover, we saw  above that this kernel contains~$\cU_5$ as a subbundle.\ 
Therefore, on~$g_k^{-1}(\F_{k+1}(M_X) \ssm \F_{k+1}(X))$, we have the equality~$\cU_5 = V_5 \otimes \cO$
of subbundles of~$V_6 \otimes \cO$, which means that~$g_k^{-1}(\F_{k+1}(M_X) \ssm \F_{k+1}(X)) $ is contained in~$\OGr_E(k+2, \cQ_E)$,
hence~\eqref{eq:hpik-preimage} holds.

Similarly, on~$\chG^0_k(X)$, both morphisms~$V_6 \otimes \cO \to \Sym^2\!\cR_{k+2}^\vee$ and~$V_5 \otimes \cO \to \Sym^2\!\cR_{k+2}^\vee$
have rank~$1$, hence their kernels
\begin{equation*}
\cU_5 \coloneqq \Ker(V_6 \otimes \cO \to \Sym^2\!\cR_{k+2}^\vee)\quad\textnormal{and}\quad 
\cU_4 \coloneqq \Ker(V_5 \otimes \cO \to \Sym^2\!\cR_{k+2}^\vee) 
\end{equation*}
are locally free  and define a morphism~${\chG^0_k(X)} \to B$.\ 
It remains to show that together with the tautological subbundle~$\cR_{k+2} \subset W \otimes \cO$,
they define a morphism~$\chG^0_k(X) \to \OGr_B(k+2,\cQ)$.\ 
By definition of~$\cU_4$ and~$\cU_5$, and Lemma~\ref{lem:ogrb-triples}, 
this is equivalent to the factorization of the tautological embedding~$\cR_{k+2} \hookrightarrow W \otimes \cO$
through the  {pullback of the} subbundle~$\cW_B \subset W \otimes \cO$.\ 

To prove this factorization, we consider the morphism
\begin{equation*}
\phi \colon \P_{\chG^0_k(X)}(\cR_{k+2}) \to B_4 \times \P(W)
\end{equation*}
induced by the embeddings~$\cU_4 \hookrightarrow V_5 \otimes \cO$ and~$\cR_{k+2} \hookrightarrow W \otimes \cO$.\ 
The pullback along~$\phi$ of the morphism~$\bq \colon \pr_1^*\cU_4 \to \pr_2^*(\cO_{\P(W) \ssm M_X}(2))$ from Lemma~\ref{lemma:u4-vanishes}
is identically zero on~$\P_{\chG^0_k(X)}(\cR_{k+2})$ (by definition of the subbundle~$\cU_4$).\ 
Therefore, the second claim of Lemma~\ref{lemma:u4-vanishes} implies  
\begin{equation*}
\phi(\P_{\chG^0_k(X)}(\cR_{k+2})) \subset \P_{B_4}(\cW) \cup (B_4 \times M_X).
\end{equation*}
Furthermore, the assumption that the rank of~$V_5 \otimes \cO \to \Sym^2\!\cR_{k+2}^\vee$ is~1 
 {means that 
\begin{equation*}
\phi(\P_{\chG^0_k(X)}(\cR_{k+2})) \cap (B_4 \times M_X) = \vide,
\end{equation*}
hence~$\phi(\P_{\chG^0_k(X)}(\cR_{k+2})) \subset \P_{B_4}(\cW)$}, as required.\ 

The constructed morphism~$\chG^0_k(X) \to \OGr_B(k+2, \cQ)$ is obviously the inverse to the restriction of~$g_k$.
\end{proof}

Combining Propositions~\ref{prop:pik} and~\ref{prop:hpik}, we obtain the following corollary.

\begin{coro}
\label{cor:gk0}
There is an isomorphism~$\G^0_k(X) \cong \OGr_{B \ssm E}(k + 2, \cQ) \ssm g_k^{-1}(\F_{k+1}(X))$.
\end{coro}

The following proposition describes the nontrivial fibers of~$g_k$. 

\begin{prop}
\label{lem:hpik-fibers}
Let~$X$ be a GM variety of dimension~$n$ satisfying Property~\eqref{hh}.\  
The morphism~$g_k \colon \OGr_B(k+2,\cQ) \to \chG_k(X)$ defined in~\eqref{defpik} is
\begin{itemize}
\item   
over $\F_{k+1}(X)\subset \chG_k(X)$,  a Zariski-locally trivial fibration with fibers
\begin{itemize}
\item 
$\vide$ over~$\F_3(X)$,
\item 
$\P^1$ over~$\Fs_2(X)$,  
\item 
$\Bl_{\bp_X}(\P^2)$ over~$\Ft_2(X)$ and~$\F_1(X)$,
\end{itemize}
\item
over~$\F_{k+1}(M_X) \ssm \F_{k+1}(X)\subset \chG_k(X)$, a stratified Zariski-locally trivial fibration 
whose possible fibers are in the top line of the following table and the loci where they occur are the entries 
\textup(for clarity, we omitted~``${}\ssm\F_{k+1}(X)$'' everywhere\textup)
\begin{equation*}
\setlength\extrarowheight{5pt}\qquad 
\begin{array}{c|c|c|c|c}
& \vide & \P^0 & \P^1 & \P^2
\\\hline
k = 0 & \cellcolor{lightgray} & \cellcolor{lightgray} & \Fo_1(M_X) & \F_1(M_X) \ssm \Fo_1(M_X) 
\\
k = 1 & \cellcolor{lightgray} & \Fo_2^\sigma(M_X) & \F_2(M_X) \ssm \Fo_2^\sigma(M_X) & \cellcolor{lightgray}
\\
k = 2 & \Fo_3(M_X) & \Fs_3(M_X) \ssm \Fo_3(M_X) & \Ft_3(M_X) \ssm \Fo_3(M_X) & \cellcolor{lightgray}
\\
k = 3 & \F_4(M_X)  & \cellcolor{lightgray} & \cellcolor{lightgray} & \cellcolor{lightgray}
\end{array}
\end{equation*}
\end{itemize}
 In particular, $g_k$ has connected fibers and is surjective when~$k\le 1$ or~$n\le 3$. 

Moreover, the restriction of the morphism~$g_k$ to~$\OGr_E(k+2,\cQ_E) \subset \OGr_B(k+2,\cQ)$
is a stratified Zariski-locally trivial fibration with   fibers described in the above table;
this time, the descriptions hold without subtracting~$\F_{k+1}(X)$.
\end{prop}

\begin{proof}
 {Let~$[R_{k+2}] \in \chG_k(X)$.\ 
By Lemma~\ref{lem:ogrb-triples}, we have
\begin{equation*}
g_k^{-1}([R_{k+2}]) = \{ (U_4,U_5) \in B \mid U_5 \subset \Ker(V_6 \to \Sym^2\!R_{k+2}^\vee),\ R_{k+2} \subset \cW_{[U_4]} \}.
\end{equation*}
In other words, $g_k^{-1}([R_{k+2}])$ is the zero locus of the morphisms
\begin{equation}
\label{eq:fiber-gk}
\cU_5 \to \Sym^2\!R_{k+2}^\vee \otimes \cO_B
\qquad\text{and}\qquad 
R_{k+2} \otimes \cO_B \to \cU_4(H_4),
\end{equation}
where the second morphism is the composition of~$R_{k+2} \otimes \cO_B \to W \otimes \cO_B$ with the second morphism in~\eqref{defwk2}.}

Assume~$[R_{k+2}] \in \F_{k+1}(X)$.\ 
It follows from~\cite[Section~4.1]{DK2} that there is a flag  {of vector subspaces}~$V_{i_1} \subset V_{i_2} \subset V_5$, 
where~$0 \le i_1 \le 1$ and~$3 \le i_2 \le 5$, 
such that the projection from the vertex of~$\CGr(2,V_5)$ induces an isomorphism
\begin{equation*}
\P(R_{k+2}) \cong \{ [U_2] \in \Gr(2,V_5) \mid V_{i_1} \subset U_2 \subset V_{i_2} \} \subset \Gr(2,V_5).
\end{equation*}
In other words, $R_{k+2} = \bw{i_1}V_{i_1} \wedge \bw{2-i_1}V_{i_2}$.\ 
The vanishing of the second morphism in~\eqref{eq:fiber-gk} is equivalent to the condition~$V_{i_2} \subset U_4$.\ 
Moreover, the first morphism in~\eqref{eq:fiber-gk} vanishes identically.\ 
Taking the definition of~$B$ into account, obtain
\begin{equation*}
g_k^{-1}([R_{k+2}]) = \{ (U_4,U_5) \in B \mid V_{i_2} \subset U_4 \subset V_5\ \text{and}\ U_4 \subset U_5 \subset V_6 \}.
\end{equation*}
It remains to note that this subscheme of~$B$ is empty if~$i_2 = 5$ (because there is no room for~$U_4$),
isomorphic to~$\P^1$ if~$i_2 = 4$, and isomorphic to the blowup of~$\P^2$ if~$i_2 = 3$.\ 
Zariski-local triviality is obvious, 
because the subspaces~$V_{i_2} \subset V_5$ form a subbundle in~$V_5 \otimes \cO_{\F_{k+1}(X)}$.

When~$[R_{k+2}] \in \F_{k+1}(M_X) \ssm \F_{k+1}(X)$, the description is similar,
the only difference being that this time,  {the vanishing of the first morphism in~\eqref{eq:fiber-gk} implies that}~$U_5$ must be equal to~$V_5$.\ 
Therefore, the fiber is empty if~$i_2 = 5$, a point if~$i_2 = 4$, a line if~$i_2 = 3$, and a plane if~$i_2 = 2$.

Similarly, when we consider the fibers of~$g_k\vert_{\OGr_E(k+2,\cQ_E)}$, we must fix~$U_5 = V_5$,
but the rest of the description is the same.
\end{proof}

\begin{coro}
\label{cor:fkm-connected}
 If~$k \ge 2$, the subscheme~$\G_k^0(X) \subset \G_k(X)$ is closed,  {so that}
\begin{equation*}
\G_k(X) = \G_k^0(X) \sqcup \Gs_k(X) \sqcup \Gt_k(X)
\end{equation*}
 {is a disjoint union of closed subschemes} and~$\G_k^0(X) \cong \OGr_{B \ssm E}(k+2, \cQ)$.
 \end{coro}

\begin{proof}
If~$k \ge 2$, Proposition~\ref{lem:hpik-fibers} implies~$g_k^{-1}(\F_{k+1}(X)) = \vide$, 
hence Proposition~\ref{prop:hpik} implies~{$g_k^{-1}(\F_{k+1}(M_X)) = \OGr_E(k+2,\cQ_E)$  and shows that} the morphism 
\begin{equation*}
g_k \colon \OGr_{B \ssm E}(k + 2, \cQ) \lra \chG^0_k(X) 
\end{equation*}
is an isomorphism.\ 
Furthermore,  {setting~$p = k + 2 \ge 4$ and applying Corollary~\ref{cor:ogre-component},  
we conclude} that~$\OGr_{B \ssm E}(k + 2, \cQ)$ is closed in~$\OGr_B(k + 2, \cQ)$.\ 
Since~$g_k$ is proper, we conclude that~$\chG_k^0(X) = g_k(\OGr_{B \ssm E}(k + 2, \cQ))$ is closed in~$\chG_k(X)$,
and therefore~$\G^0_k(X) = \lambda_k^{-1}(\chG_k^0(X))$ is closed in~$\G_k(X)$.\ 
The disjoint union then follows from~\eqref{eq:g0k} because~$\Fst_{k+1}(M_X) = \vide$ for~\mbox{$k \ge 2$,}
and the isomorphism for~$\G_k^0(X)$ follows from Corollary~\ref{cor:gk0}.
\end{proof}

Corollary~\ref{cor:fkm-connected} shows that the scheme structure of~$\G_k(X)$ for~$k \ge 2$ is rather simple;
in the next lemma, we discuss the case~$k = 1$.

\begin{lemm}
\label{lem:g1-x56}
Let~$X$ be a smooth GM variety of dimension~$n \ge 3$.\ 
The Hilbert scheme~$\G_1(X)$ of conics on~$X$ is a Cohen--Macaulay scheme
 of pure dimension~$3n - 7$.\ 
 \end{lemm}

\begin{proof}
By definition,~$X \subset \CGr(2,V_5) \ssm \{\bv\}$ is the zero locus of a global section of the vector bundle~$\cO(1)^{\oplus (6 - n)} \oplus \cO(2)$.\ 
If~$\cC(X) \subset X \times \G_1(X)$ is the universal conic and~$p_X \colon \cC(X) \to X$ and~$p_\G \colon \cC(X) \to \G_1(X)$ are the projections,
then~$\G_1(X) \subset \G_1(\CGr(2,V_5) \ssm \{\bv\})$ is the zero locus of a global section 
of the vector bundle
\begin{equation*}
  {p_\G}_*p_X^*(\cO(1)^{\oplus (6 - n)} \oplus \cO(2)),
\end{equation*}
of rank~$3(6 - n) + 5 = 23 - 3n$.\ 
Since~$\G_1(\CGr(2,V_5) \ssm \{\bv\})$ is smooth of dimension~$16$ by Lemma~\ref{lem:g1-cgr},
$\G_1(X)$ has everywhere dimension at least $16 - (23 - 3n) = 3n - 7$.\ 
To prove that it is Cohen--Macaulay of that dimension, one needs to check the bound~$\dim(\G_1(X)) \le 3n - 7$. 

For this, we set~$p = k + 2 = 3$,  so that~\mbox{$\ell = 2p - n - 1 = 5 - n \le 2$},  ({see~\eqref{eq:ell-p}).\ 
Then} Corollary~\ref{cor:gk0}, Proposition~\ref{prop:corank-n-k}\ref{it:ogr-bme} (in the case~$n = 6$, use Remark~\ref{rema:l-neg}), 
and~\eqref{eq:nnp} imply  
\begin{equation*}
\dim(\G_1^0(X)) = \dim(\OGr_{B \ssm E}(3,\cQ)) = N(n,3) + 5 = (3n - 12) + 5 = 3n - 7
\end{equation*}
and~$\G_1^0(X)$ is normal and irreducible.\ 
Furthermore, by Proposition~\ref{prop:pik}, we have
 \begin{equation*}
\dim(\lambda_1^{-1}(\F_2(X))) = 
\dim(\F_2(X)) + 5 \le 
(3n - 12) + 5 = 3n - 7,
\end{equation*}
where the inequality is explained in Lemma~\ref{lem:fkx-fkmx}\ref{it:f2}  {and Proposition~\ref{prop:fk-gm}}.\ 
Finally, by Proposition~\ref{prop:pik} and Corollary~\ref{cor:dim-fkmx}, we have
\begin{equation*}
\dim(\lambda_1^{-1}(\F_2(M_X)\ssm \F_2(X)))) = 
\dim( \F_2(M_X)\ssm \F_2(X))) \le 
\dim(\F_2(M_X)) \le 
3n - 7.
\end{equation*}
Since $\G_1(X) = \G_1^0(X) \cup  {\lambda_1^{-1}(\F_2(X)) \cup \lambda_1^{-1}(\F_2(M_X) \ssm \F_2(X))}$, 
this proves that~$\G_1(X)$ is a Cohen--Macaulay scheme of pure dimension~$3n - 7$.
\end{proof}

 {The following lemma proves an additional property of~$\G_1(X)$ when~$X$ is general.}

\begin{lemm}
\label{lem:g1-x56-extra}
Let~$X$ be a smooth GM variety of dimension~$n \ge 3$.\ 
If
\begin{equation}
\label{newassu}
\dim(\F_2(X)) \le 3n - 14
\qquad\text{and}\qquad 
\dim(\F_2(M_X)) \le 3n - 8
\end{equation}
 {\textup(this holds if~$X$ is general or $n = 6$\textup)}
the scheme~$\G_1(X)$ is normal and integral.
\end{lemm}

\begin{proof}
The inequalities~\eqref{newassu} hold for general~$X$ by Lemma~\ref{lem:f2-general} and Corollary~\ref{cor:dim-fkmx}.

If~\eqref{newassu} holds for $X$,  we obtain,
arguing as in the proof of Lemma~\ref{lem:g1-x56}, 
improved bounds for the dimensions of~$\lambda_1^{-1}(\F_2(X))$ and~$\lambda_1^{-1}(\F_2(M_X) \ssm \F_2(X))$, 
which imply that~$\G_1^0(X)$ is dense in~$\G_1(X)$, which is therefore irreducible.

To prove normality, it remains to show that~$\G_1(X)$ is nonsingular in codimension~$1$.\ 
As we already know this on~$\G_1^0(X)$, and since  {the improved bound implies that}~$\lambda_1^{-1}(\F_2(X))$ has codimension at least~$2$,
we only need  to check that~$\G_1(X)$ is nonsingular at the generic point 
of each component of~\mbox{$\lambda_1^{-1}(\F_2(M_X)\ssm \F_2(X))$}.\ 
This is proved in Corollary~\ref{cor:g1x-f2mx}.
\end{proof}

\begin{coro}
\label{cor:g1-connected}
If~$X$ is a smooth GM variety of dimension~$n \ge 3$,
the Hilbert scheme~$\G_1(X)$ of conics on~$X$ is connected.\ 
When $X$ is general,~$\G_1(X)$ is moreover smooth.
\end{coro}

\begin{proof}
Consider the nested Hilbert scheme~$\cH$ parameterizing pairs
\begin{equation*}
\Sigma \subset X \subset \CGr(2,V_5),
\end{equation*}
where~$\Sigma$ is a conic and~$X$ is a smooth GM variety of dimension~$n \ge 3$.\ 
The fibers of the projection~$\cH \to \G_1(\CGr(2,V_5) \ssm \{\bv\})$ parameterize smooth GM varieties passing through a given conic~$\Sigma$.\ 
Since a GM variety of dimension~$n$ is a complete intersection in~$\CGr(2,V_5)$ of~$6 - n$ hyperplanes and a quadric,
these fibers are open subsets in a $\P^{(n+5)(n+6)/2 - 11}$-bundle over~$\Gr(6 - n, 8)$.\ 
Using Lemma~\ref{lem:g1-cgr}, we conclude that~$\cH$ is smooth and connected. 

Consider the projection from~$\cH$ to the moduli space of GM varieties of dimension~$n$.
Its fiber over~$[X]$ is the Hilbert scheme~$\G_1(X)$.\ 
By Lemma~\ref{lem:g1-x56-extra},   general fibers are connected,
and since~$\cH$ itself is smooth and connected, a Stein factorization argument implies that every fiber is connected.\ 
The last statement of the corollary follows from generic smoothness.
\end{proof}

\subsection{The scheme~$\Gp_k(X)$}

In some cases, it is helpful to factor the morphism~$g_k$ 
into a composition of two simpler morphisms; we do this in this section.\ 
We will concentrate on the case where~$X$ is ordinary, $k \in \{0,1\}$, and~$2k + 1 \le n \le 2k + 3$.\ 

\begin{prop}
\label{prop:gp}
Let~$X$ be an ordinary GM variety of dimension $n$ satisfying Property~\eqref{hh}.\ 
 Let $k \in \{0,1\}$ be such that~$2k + 1 \le n \le 2k + 3$.\ 
There  is a normal integral Cohen--Macaulay  {projective} variety~$\Gp_k(X)$ and an embedding
\begin{equation*}
\Ft_{k+1}(M_X) \lhra \Gp_k(X)
\end{equation*}
as a smooth subscheme of codimension~$2$ contained in the smooth locus of~$\Gp_k(X)$ such that
\begin{equation*}
\OGr_B(k + 2, \cQ) \cong \Bl_{\Ft_{k+1}(M_X)}(\Gp_k(X))
\end{equation*}
and the morphism~$g_k$ defined in~\eqref{defpik} admits a factorization
\begin{equation*}
\OGr_B(k + 2, \cQ) \xrightarrow{\ g'_k\ } \Gp_k(X) \xrightarrow{\ g''_k\ } \chG_k(X)
\end{equation*}
where~$g'_k$ is the blowup morphism
 {and the composition~$\Ft_{k+1}(M_X) \lhra \Gp_k(X) \xrightarrow{\ g''_k\ } \chG_k(X)$
coincides with the embedding defined in Lemma~\textup{\ref{lem:fkmx-dkx}}.}

 If~$k = 1$, the embedding~$\Fs_{k+1}(M_X) \hookrightarrow \chG_k(X)$ lifts along~$g''_k$ to an   embedding 
\begin{equation*}
\Fs_{k+1}(M_X) \lhra \Gp_k(X)
\end{equation*}
as a smooth Cartier divisor contained in the smooth locus of~$\Gp_k(X)$.
\end{prop}

\begin{proof}
By Proposition~\ref{prop:hpik} there is a morphism~$g_k\colon \OGr_E(k + 2, \cQ_E) \to \F_{k+1}(M_X)  \subset \chG_k(X) $.\ 
Consider the subscheme~$\Ft_{k+1}(M_X) \subset \F_{k+1}(M_X)$;
if~$k = 1$, it is a connected component, and if~$k = 0$, the embedding is an equality (see Section~\ref{subsec:ls-qu}).\ 
In either case,  {its} preimage
\begin{equation*}
\OGr^\tau_E(k + 2, \cQ_E) \coloneqq \OGr_E(k + 2, \cQ_E) \cap g_k^{-1}(\Ft_{k+1}(M_X))
\end{equation*}
is a union of connected components of~$\OGr_E(k + 2, \cQ_E)$;
in particular, $\OGr^\tau_E(k + 2, \cQ_E)$ is a smooth Cartier divisor in~$\OGr_B(k + 2, \cQ)$ by Proposition~\ref{prop:corank-n-k}\ref{it:l2-ord}.\ 
 {Moreover, $\OGr_E(k + 2, \cQ_E)$ is the preimage under the morphism~$f$ of the Cartier divisor~$E \subset B$,
so that we have an isomorphism
\begin{equation*}
N_{\OGr^\tau_E(k + 2, \cQ_E) / \OGr_B(k + 2, \cQ)} \cong f^*N_{E/B}\vert_{\OGr^\tau_E(k + 2, \cQ_E)}.
\end{equation*}
By} Proposition~\ref{lem:hpik-fibers}, the morphism
\begin{equation*}
g_k\vert_{\OGr^\tau_E(k + 2, \cQ_E)} \colon \OGr^\tau_E(k + 2, \cQ_E) \lra \Ft_{k+1}(M_X)
\end{equation*}
is a $\P^1$-fibration and  {the normal bundle of~$\OGr^\tau_E(k + 2, \cQ_E)$}
is a line bundle of degree~$-1$ on each $\P^1$-fiber.\  
Applying the main theorem of~\cite{Mo},
 we obtain a normal algebraic space~$\Gp_k(X)$, a birational contraction 
\begin{equation*}
g'_k \colon \OGr_B(k + 2, \cQ) \lra \Gp_k(X),
\end{equation*}
and an embedding~$\Ft_{k+1}(M_X) \hookrightarrow \Gp_k(X)$ into the smooth locus of~$\Gp_k(X)$,
such that~$g'_k$ is the blowup of~$\Gp_k(X)$ along~$\Ft_{k+1}(M_X)$,
the exceptional divisor coincides with~$\OGr^\tau_E(k + 2, \cQ_E)$,
and the restriction of~$g'_k$ to the exceptional divisor coincides with~$g_k\vert_{\OGr^\tau_E(k + 2, \cQ_E)}$.

Since~$\OGr_B(k + 2, \cQ)$ is normal, integral, and Cohen--Macaulay (see Proposition~\ref{prop:corank-n-k}\ref{it:l2-ord}),
the same is true for~$\Gp_k(X)$.\ 
Since, moreover, $g'_k\vert_{\OGr^\tau_E(k + 2, \cQ_E)} = g_k\vert_{\OGr^\tau_E(k + 2, \cQ_E)}$,
the morphism~$g_k$ factors through~$g'_k$.\ 
It also follows that the restriction of~$g''_k$ to~$\Ft_{k+1}(M_X) \subset \Gp_k(X)$ is an isomorphism onto~$\Ft_{k+1}(M_X) \subset \chG_k(X)$.

 {The same argument proves that the composition~$\OGr_B(k + 2, \cQ) \to B \to B_5$ factors through~$\Gp_k(X)$,
and Proposition~\ref{lem:hpik-fibers} shows that the morphism~$\Gp_k(X) \to \chG_k(X) \times B_5$ is finite,
hence the algebraic space~$\Gp_k(X)$ is a projective variety.}

Finally, if~$k = 1$, the preimage 
\begin{equation*}
\OGr^\sigma_E(k + 2, \cQ_E) \coloneqq \OGr_E(k + 2, \cQ_E) \cap g_k^{-1}(\Fs_{k+1}(M_X))
\end{equation*}
is a smooth Cartier divisor, disjoint from~$\OGr^\tau_E(k + 2, \cQ_E)$
 and the proper morphism
\begin{equation*}
g_k\vert_{\OGr^\sigma_E(k + 2, \cQ_E)} \colon \OGr^\sigma_E(k + 2, \cQ_E) \lra \Fs_{k+1}(M_X)
\end{equation*}
is bijective by Proposition~\ref{lem:hpik-fibers}; since~$\Fs_{k+1}(M_X)$ is smooth by Lemma~\ref{lem:fk-mx}, it is an isomorphism.\ 
Since~$\OGr^\sigma_E(k + 2, \cQ_E)$ is disjoint from the exceptional divisor of~{$g'_k$},
we obtain an embedding of~$\Fs_{k+1}(M_X) \cong \OGr^\sigma_E(k + 2, \cQ_E)$ into~$\Gp_k(X)$ as a smooth Cartier divisor.\ 
Since it is smooth, it must be contained in the smooth locus of~$\Gp_k(X)$.
 \end{proof}

To conclude this section, we explain how the scheme~$\Gp_k(X)$ fits into the picture of schemes and maps   defined so far.\ The
 inclusion~\mbox{$\F_{k+1}(X)   \subset \chG_k(X)$} was seen in  Lemma~\ref{lem:fkmx-dkx}.

\begin{prop}
\label{prop:diagram-big}
 Let~$X$ be an ordinary GM variety of dimension $n$ satisfying Property~\eqref{hh}.\ 
 Let~$k \in \{0,1\}$ be such that~{$2k + 1 \le n \le 2k + 3$}.\
 There is a commutative diagram 
 \begin{equation}
\label{eq:diagram-big}
\vcenter{\xymatrix@C=2em{
 & \hG_k(X) \ar[dl]_{{\hg_k}} \ar[dr]^(0.45){{\hla_k}}|(0.67)\hole &&
 \OGr_B(k + 2,\cQ) \ar[dl]_{g'_k} \ar[dr]^{\tf}  \ar@/_2.8pc/[ddll]_(0.45){g_k} 
\\
\G_k(X) \ar[dr]^{\lambda_k} && 
\Gp_k(X) \ar[dl]_{g''_k} \ar[dr]^{f^+} &&
\Bl_{\widetilde\bp_X}(\wtY^{\ge \ell}_\Ap) \ar[dl]
 \\
& \chG_k(X) &&
\wtY^{\ge \ell}_\Ap,
}}
\end{equation}
where~{$\hG_k(X) = \Bl_{(g''_k)^{-1}(\F_{k+1}(X))}(\Gp_k(X))$, the morphism~$\hla_k$ is the blowup morphism, while}
 \begin{itemize}
\item 
if~$k = 0$, the morphism~$(g''_{0})^{-1}(\F_{1}(X)) \to \F_{1}(X)$ is a $\P^2$-bundle, and
\item 
if~$k = 1$,  the morphism~$(g''_{1})^{-1}(\F_{2}(X)) \to \F_{2}(X)$ is a $\P^1$-bundle 
over  {the component~$\Fs_2(X)$ of~$\F_2(X)$} and a~$\P^2$-bundle over~$\Ft_2(X)$.
\end{itemize}
Finally, if~$\F_{k+1}(X) = \vide$, the morphisms~$\lambda_k$, $g''_k$, $\hla_k$, and~$\hg_k$ are all isomorphisms;
in particular, we have~$\G_k(X) \cong \chG_k(X) \cong \Gp_k(X)$.
 \end{prop}

\begin{proof}
 To construct the right square, note that the composition
\begin{equation*}
\OGr_B(k + 2, \cQ) \xrightarrow{\  {\tf}\ } \Bl_{\widetilde\bp_X}(\wtY^{\ge \ell}_\Ap) \xrightarrow{\quad} \wtY^{\ge \ell}_\Ap
\end{equation*}
contracts the divisor~$\OGr_E(k + 2, \cQ_E) \subset \OGr_B(k + 2, \cQ)$ to the Pl\"ucker point~$\bp_X$,
hence it factors through the contraction~$g'_k$ (of a connected component) of this divisor.

To construct the left square, recall that~$\F_{k+1}(X)$ is the zero locus of the composition 
\begin{equation*}
V_6 \otimes \cO \to 
\Sym^2\!W^\vee \otimes \cO \to 
\Sym^2\!\cR_{k+2}^\vee
\end{equation*}
of morphisms of vector bundles on~$\chG_k(X)$, 
hence the scheme~$(g''_k)^{-1}(\F_{k+1}(X))$ is the zero locus of its pullback to~$\Gp_k(X)$.\ 
By definition of~$\OGr_B(k + 2, \cQ)$ (see Lemma~\ref{lem:ogrb-triples}) 
and the construction of  {the} morphism~$f^+ \colon  {\Gp_k(X) \to \wtY^{\ge \ell}_\Ap}$, 
this  {composition} vanishes on the subbundle~$(f^+)^*\cU_5 \subset V_6 \otimes \cO_{\Gp_k(X)}$,
hence induces a morphism
\begin{equation*}
(f^+)^*\cO(H_5) \cong (V_6 \otimes \cO_{\Gp_k(X)}) / (f^+)^*\cU_5 \to (g''_k)^*\Sym^2\!\cR_{k+2}^\vee.
\end{equation*}
Dualizing, we obtain a morphism~$(g''_k)^*\Sym^2\!\cR_{k+2} \to (f^+)^*\cO(- H_5)$
whose zero locus is equal to~$(g''_k)^{-1}(\F_{k+1}(X))$ and whose image is the ideal of~$(g''_k)^{-1}(\F_{k+1}(X))$ twisted by~$(f^+)^*\cO(-H_5)$.\ 
Therefore, on the blowup~$\hG_k(X)$ of~$(g''_k)^{-1}(\F_{k+1}(X))$, 
we obtain a surjective morphism
\begin{equation*}
\hla_k^*(g''_k)^*(\Sym^2\!\cR_{k+2}) \thlra \cO(-H_5 - \widehat{E}),
\end{equation*}
 {where~$\widehat{E} \subset \hG_k(X)$ is the exceptional divisor of~$\hla_k$.}\ 
Since~$\cR_{k+2}$ is the tautological bundle on~$\Gr(k+2,W)$, 
we obtain  {a composition of} morphisms
\begin{equation*}
\hG_k(X) \to 
\P_{\hG_k(X)}(\hla_k^*(g''_k)^*\Sym^2\!\cR_{k+2}^\vee) \to
\P_{\Gr(k+2,W)}(\Sym^2\!\cR_{k+2}^\vee) =
\G_k(\P(W)).
\end{equation*}
On the dense open subset~$\hla_k^{-1}((g_k'')^{-1}(\chG^0_k(X))) \subset \hG_k(X)$, 
this {composition} coincides by construction with~$\lambda_k^{-1} \circ g''_k \circ \hla_k$; 
in particular, its image is contained in~$\G_k(X) \subset \G_k(\P(W))$.\ 
Thus, we obtain the morphism~$\hg_k \colon \hG_k(X) \to \G_k(X)$ such that the left square commutes.\

If~$k = 0$, the morphism~$g_{0}$ is, by Proposition~\ref{lem:hpik-fibers}, 
a fibration over {the scheme}~$\F_{1}(X)$ with fibers~$\Bl_{\bp_X}(\P^2)$.\ 
Moreover,  {the argument of Proposition~\ref{lem:hpik-fibers} also shows that}
its fibers intersect the divisor~$\OGr_E(  2, \cQ_E)$ along the exceptional curve of~$\Bl_{\bp_X}(\P^2)$,
which   is also equal to the fiber of the morphism~$g'_{0}$.\ 
Therefore, the fiber of~$g''_{0}$ is obtained from~~$\Bl_{\bp_X}(\P^2)$ by contracting the exceptional curve,
hence it is isomorphic to~$\P^2$.

 {If~$k = 1$, the morphism~$g_{1}$ is, again by Proposition~\ref{lem:hpik-fibers}, a fibration over~$\Fs_{2}(X)$ with fibers~$\P^1$ and over~$\Ft_2(X)$ with fiber~$\Bl_{\bp_X}(\P^2)$.\ 
Moreover, the fibers over~$\Fs_2(X)$ do not intersect the exceptional divisor~$\OGr_E(2, \cQ_E)$ of~$g'_1$,
while the fibers over~$\Ft_2(X)$ intersect it along the exceptional curve of~$\Bl_{\bp_X}(\P^2)$,
which is also equal to the fiber of the morphism~$g'_{1}$.\ 
This gives the required description of the fibers of~$g''_1$ over~$\F_2(X)$.}

If~$\F_{k+1}(X) = \vide$, the morphism~$\lambda_k$ is an isomorphism by Corollary~\ref{cor:pik}
 {and~$\hla_k$ is an isomorphism by construction; in particular, $\hg_k = g''_k$}.\ 
 Moreover, by Proposition~\ref{lem:hpik-fibers}, 
the morphism~$g'_k$ contracts all positive dimensional fibers of the surjective morphism~$g_k$ 
 {and, since by Proposition~\ref{prop:gp} the scheme~$\Gp_k(X)$ is normal, the morphism~$g''_k$ is the normalization of~$\chG_k(X) \cong \G_k(X)$.\ 
But~$\G_0(X)$ is the Hilbert square of a smooth variety, hence it is smooth,
and~$\G_1(X)$ is normal by Lemma~\ref{lem:g1-x56-extra} 
{(the first inequality in~\eqref{newassu} holds by the assumption~$\F_2(X) = \vide$
and the second inequality holds by Corollary~\ref{cor:dim-fkmx} since~$X$ is ordinary)}.\ In both cases,~$g''_k$ is thus an isomorphism.}
 \end{proof}


\section{Explicit descriptions}
\label{sec:explicit}

In this section, we provide explicit descriptions of some of the Hilbert schemes~$\G_k(X)$ 
of quadrics of dimension~$k$ on a smooth GM variety~$X$ of dimension~$n$.\ 
As we will see, the complexity of the scheme~$\G_k(X)$ grows with~$n - 2k$; 
equivalently, it decreases with
\begin{equation*}
\ell = 2k + 3 - n.
\end{equation*}
We will therefore consider the cases in the corresponding order.

The subschemes~$\overline{\G^0_k(X)}$, $\Gs_k(X)$, and~$\Gt_k(X)$ of~$\G_k(X)$ were defined in Definition~\ref{def:st-quadrics}.\

\subsection{Quadrics of dimensions more than half}  
\label{ss:quadrics-big}

We start by considering the case~$n < 2k$, that is,~$\ell > 3$.\ 
We show that  the Hilbert scheme~$\G_k(X)$ is then empty with a single exception.

\begin{prop}
\label{prop:g3x5}
Let~$X$ be a smooth GM variety   of dimension $n\ge2$.\ 
If~$n < 2k$, the scheme~$\G_k(X)$ is empty 
unless~$k = 3$ and~$X$ is a special GM fivefold, 
in which case the scheme~$\G_3(X) = \Gs_3(X)$ is a point.\ 
In particular, $\G_k(X) = \vide$ for all~$X$ if~$k \ge 4$.
\end{prop}

\begin{proof}
Assume~$n < 2k$.\
Since~$n \ge 2$, we have~$k \ge 2$.

If~$k = 2$, then~$n \le 3$, so~$\G_2(X)$ is the Hilbert scheme of quadric surfaces on GM surfaces or threefolds;
in the case of surfaces, it is obviously empty and, in the case of threefolds, it is empty by~\cite[Corollary~3.5]{DK2}.

If~$k = 3$, then~$n \le 5$.\ 
For~$n \le 4$, the same arguments as above work, so assume~\mbox{$n = 5$}.\ 
Then~$\F_{k+1}(X) = \vide$ by Proposition~\ref{prop:fk-gm}, and~$\OGr_B(k + 2, \cQ) = \vide$ by Proposition~\ref{prop:corank-n-k}\ref{it:l4}.\ 
Therefore, by {Corollaries~\ref{cor:fkm-connected}} and~\ref{cor:gstk}, we obtain
\begin{equation*}
\G_3(X) = \Gs_3(X) \sqcup \Gt_3(X) = \Fs_4(M_X) \sqcup \Ft_4(M_X).
\end{equation*}
By Lemmas~\ref{lem:fk-mx} and~\ref{lem:fk-cm}, this is empty if~$X$ is ordinary, and  a point if~$X$ is special.\ 
Note that the corresponding linear subspace on~$M_X$ is a $\sigma$-space, hence the quadric is a $\sigma$-quadric.
 
Finally, assume~$k \ge 4$.\ 
As before, we have $\F_{k+1}(X) = \vide$ and~$\OGr_B(k + 2,\cQ) = \vide$, 
hence we have~$\G_k(X) = \F_{k+1}(M_X)$.\ This is empty by Lemmas~\ref{lem:fk-mx} and~\ref{lem:fk-cm}.
\end{proof}

\subsection{Half-dimensional quadrics}
\label{ss:quadrics-half}

The next case is~$n = 2k$, that is,~$\ell = 3$.\ 
 {By~\eqref{eq:px-dim}},
the Pl\"ucker point~$\bp_X$ lies in~$\sY^3_\Ap$ when~$k = 1$, and away from~$\sY^3_\Ap$ when~$k \ge 2$.

\begin{theo}
\label{thm:gk-x2k}
Let~$X$ be  {a   GM variety of even dimension~$n = 2k $ satisfying Property~\eqref{hh}},
with associated Lagrangian~$A$.\  
Then
\begin{equation*}
\G_k^0(X) = \overline{\G^0_k(X)},
\qquad 
\Gs_k(X) = \Fs_{k+1}(M_X),
\qquad
\Gt_k(X) = \Ft_{k+1}(M_X),
\end{equation*}
and the scheme~$\G_k(X)$ is a disjoint union  {of closed subschemes}
\begin{equation*}
\G_k(X) =  {\G^0_k(X)} \sqcup \Gs_k(X) \sqcup \Gt_k(X)
\end{equation*}
 {which} have the following form depending on~$k$, $n$, and the type of~$X$:
\begin{equation*}
\renewcommand\arraystretch{1.2}
\begin{array}{l|ccc}
&  {\G_k^{0}(X)} & \Gs_k(X) & \Gt_k(X) 
\\
\hline 
\text{$k = 1$, $n = 2$, $X$ is ordinary} & \sY_\Ap^3 \ssm \{\bp_X\} & \vide & \vide
\\
\text{$k = 2$, $n = 4$, $X$ is ordinary} & \wtY_\Ap^3 & \P^0 & \vide
\\
\text{$k = 2$, $n = 4$, $X$ is special} & \wtY_\Ap^3 & \P^1 & \P^0 
\\
\text{$k = 3$, $n = 6$, $X$ is special} & \wtY_\Ap^3 \times \P^1 & \P^4 & \vide
\end{array}
\end{equation*}
\end{theo}

\begin{proof}
We have~$\F_{k+1}(X) = \vide$ by Proposition~\ref{prop:fk-gm}.

Assume~$k \ge 2$.\ 
By Corollaries~\ref{cor:gstk} and~\ref{cor:fkm-connected}, we have~{$\G^\star_k(X) = \F^\star_{k+1}(M_X)$ and}
 \begin{equation*}
\G_k(X) = \OGr_{B \ssm E}(k + 2, \cQ) \sqcup \Gs_{k}(M_X) \sqcup \Gt_{k}(M_X).
\end{equation*}
By Proposition~\ref{prop:corank-n-k}\ref{it:l3}, there is a morphism 
\begin{equation*}
\OGr_{B \ssm E}(k + 2, \cQ) \lra \sY^3_\Ap \ssm \{\bp_X\}   = \sY^3_\Ap 
\end{equation*}
whose fiber over a point~$b$ is the Hilbert scheme~$\F_{k+1}(\cQ_b)$.\ 
When~$k = 2$, this Hilbert scheme parameterizes linear spaces~$\P^3$ on a quadric of dimension~$3$ and corank~$3$,
and, when~$k = 3$, it parameterizes linear spaces~$\P^4$ on a quadric of dimension~$5$ and corank~$3$.\ 
It is therefore a disjoint union of two points or two~$\P^1$, respectively.\ 
Finally, a combination of Lemmas~\ref{lem:fk-mx} and~\ref{lem:fk-cm} gives the required description 
of the components~$\Gs_k(X) = \Fs_{k+1}(M_X)$ and~\mbox{$\Gt_k(X) = \Ft_{k+1}(M_X)$}.

Assume now~$k = 1$, hence~$n = 2$ and~$X$ is ordinary.\ 
 {Corollary~\ref{cor:gstk} and Lemma~\ref{lem:fk-mx} imply}
\begin{equation*}
\Gs_1(X) = \Fs_2(M_X) = \vide,
\qquad 
\Gt_1(X) = \Ft_2(M_X) = \vide 
\end{equation*}
hence,~$\G_1(X) = \G^0_1(X) = \OGr_{B \ssm E}(3,\cQ)$ by  {Corollary~\ref{cor:gk0}}.\ 
 It remains to note that~$\cQ \to B$ is a conic bundle, hence
\begin{equation*}
 \OGr_{B \ssm E}(3,\cQ) \cong B^{\ge 3} \ssm E = \sY^3_\Ap \ssm \{\bp_X\},
\end{equation*}
where the last equality is proved in Proposition~\ref{proposition:b-stratification}.\ 
 \end{proof}

\subsection{Conics on GM threefolds}

The next case to consider is~$n = 2k + 1$,  {that is,~$\ell = 2$}.\ 

We begin with the case~$n = 3$, so that~$X$ is a smooth GM threefold with associated Lagrangian~$A$.\ 
Its Pl\"ucker point~$\bp_X$   is in~$\sY^2_\Ap$ if~$X$ is ordinary, and in~$\sY^3_\Ap$ if~$X$ is special (see~\eqref{eq:px-dim}).\ 
The integral normal {double dual} EPW surface~$\wtY^{\ge 2}_\Ap$ was introduced in Theorem~\ref{thm:tya}.\ 
When~$X$ is special, the point $\widetilde\bp_X \in \wtY^3_\Ap$ was defined in Lemma~\ref{lem:closure};
by Proposition~\ref{propsing}, it is a node of~$\wtY^{\ge 2}_\Ap$.

The following result includes Theorem~\ref{thm:g1x3-intro} as a special case 
(compare with~\cite[Section~3]{lo} and \cite[Section~6]{dim} in the ordinary case and~\cite[Section~2]{ili1} in the special case).

\begin{theo}
\label{thm:g1-x3}
Let~$X$ be a smooth GM threefold with associated Lagrangian~$A$.\
The Hilbert scheme~$\G_1(X)$ of conics on~$X$ is a connected Cohen--Macaulay surface and
 \begin{equation*}
\G_1(X) \cong 
\begin{cases}
\Bl_{\bp'_X}(\wtY_{\Ap}^{\ge 2}) & \text{if~$X$ is ordinary,}
\\
\Bl_{\widetilde\bp_X}(\wtY_{\Ap}^{\ge 2}) \cup \P^2 & \text{if~$X$ is special,}
\end{cases}
\end{equation*}
where
\begin{itemize}
\item 
If~$X$ is ordinary,  {so that~$\bp_X \in \sY^2_\Ap$}, the point
$\bp'_X$ is one of the two preimages in~$\wtY_{\Ap}^{\ge 2}$ of~\mbox{$\bp_X$}.\ 
Moreover,~$\Gs_1(X) \cong \P^1$ is the exceptional curve over~$\bp'_X$ 
and~$\Gt_1(X) =  {\{ \bp''_X \}}$ is the other preimage of~$\bp_X$.
\item  
If~$X$ is special, so that~$\bp_X \in \sY^3_\Ap$, the point~$\widetilde\bp_X \in \wtY^3_\Ap$ is the point over~$\bp_X$,
the component~$\Bl_{\widetilde\bp_X}(\wtY_{\Ap}^{\ge 2})$ is equal to~$\overline{\G^0_1(X)}$, 
while~$ \P^2 $ is~$\Gs_1(X) = \Gst_1(X) = \Gt_1(X)$.\ 
These components intersect transversely along a smooth rational curve:
the exceptional curve on the component~$\Bl_{\widetilde\bp_X}(\wtY_{\Ap}^{\ge 2})$
and a conic on the component~$\P^2$.
\end{itemize}
 \end{theo}

\begin{proof}
By Proposition~\ref{prop:fk-gm},~$\F_2(X) $ is empty,  {hence~$\G_1(X) \isomto \chG_1(X)$ by Corollary~\ref{cor:pik}}.

Assume~$X$ is ordinary.\ 
  Proposition~\ref{prop:corank-n-k}\ref{it:l2-ord} together with Lemma~\ref{lem:closure} then imply  
\begin{align*}
\OGr_B(3,\cQ) &\cong \Bl_{\bp'_X,\bp''_X}(\wtY^{\ge 2}_\Ap),\\
\OGr_E(3,\cQ_E) &= f^{-1}(\bp_X) =
 {\tf^{-1}(\bp'_X) \cup \tf^{-1}(\bp''_X) = {}}
\P^1 \sqcup \P^1,
\end{align*}
 {and} Proposition~\ref{prop:gp}  shows that the morphism~$g_1$ factors as 
\begin{equation*}
\OGr_B(3,\cQ) \xrightarrow{\ g'_1\ } \Gp_1(X) \xrightarrow{\ g''_1\ } \chG_1(X),
\end{equation*}
 {where~$g'_1$ is the contraction of one of the components, say~$\tf^{-1}(\bp''_X)$, of~$\OGr_E(3,\cQ_E)$.\ Taking the isomorphisms of Proposition~\ref{prop:diagram-big} into account, 
we obtain
\begin{equation*}
\G_1(X) \cong \chG_1(X) \cong \Gp_1(X) \cong \Bl_{\bp'_X}(\wtY^{\ge 2}_\Ap).
\end{equation*}
An identification of~$\Gs_1(X)$ and~$\Gt_1(X)$ as subschemes in~$\Bl_{\bp'_X}(\wtY^{\ge 2}_\Ap)$ also follows.}

If~$X$ is special, Proposition~\ref{prop:corank-special} and Remark~\ref{rem:p3n3-special} imply  {that} 
 \begin{equation*}
\OGr_B(3,\cQ) = \OGrm_B(3,\cQ) \cup \OGr_E(3,\cQ) 
\end{equation*}
and the components intersect along the divisor~$\OGrm_E(3,\cQ) \subset \OGrm_B(3,\cQ)$;
furthermore, we have
\begin{equation*}
\OGrm_B(3,\cQ) \cong \Bl_{\widetilde\bp_X}(\wtY^{\ge 2}_\Ap)  
\end{equation*}
where, as mentioned at the beginning of this section, $\widetilde\bp_X $ is an ordinary double point of the normal integral surface $\wtY^{\ge 2}_\Ap$, and
\begin{equation*}
\OGrm_E(3,\cQ_E) \cong \P^1
\end{equation*}
is the exceptional divisor of~$\Bl_{\widetilde\bp_X}(\wtY^{\ge 2}_\Ap)$.\ 
For the other component,  {since~$\cQ_E \to E$ is a fibration in quadric surfaces,~$\OGr_E(3,\cQ)$ is a double covering of~$E^{\ge 2}$ branched along~$E^{\ge 3}$.\ 
Since~$X$ is special, by Lemma~\ref{lem:b-strata-special}, the corank stratification of~$E$ is shifted by~1 
with respect to the stratification associated with the corresponding ordinary GM variety,
but the corresponding double covering is the same, and  Lemma~\ref{lem:tek-covers} implies}
 \begin{equation*}
\OGr_E(3,\cQ) \cong \Fl(1,2;3)  \subset \P^2 \times \P^2
\end{equation*}
 and~$\OGrm_E(3,\cQ_E) \cong \P^1$ is the intersection of~$\Fl(1,2;3)$ with the diagonal~${\Delta(\P^2) \subset {}} \P^2 \times \P^2$.\ 
{Therefore, $\OGrm_E(3,\cQ_E) \subset \Delta(\P^2)$ is a smooth conic.}

We now study the map~$g_1 \colon \OGr_B(3,\cQ) \to \chG_1(X)\isom \G_1(X) $.\ 
We have
 \begin{equation*}
 \G_1(X) = \overline{\G^0_1(X)} \cup \Gs_1(X) \cup \Gt_1(X) = \overline{\G^0_1(X)} \cup \Gst_1(X),
\end{equation*}
where the last equality follows from the equality~$\F_2(M_X) = \F_1(M'_X) = \Fst_1(M'_X)$ (see Remark~\ref{rem:f3m}).\ 
The map~$g_1 $ is compatible with these decompositions and induces maps 
\begin{equation}\label{eqog}
  \OGrm_B(3, \cQ) \lra \overline{\G^0_1(X)}
\qquad \textnormal{and}\qquad
  \OGr_E(3, \cQ_E) \lra \Gst_1(X) 
\end{equation}
which we investigate separately. 

We have~$\OGr_E(3, \cQ_E) \cong \Fl(1,2;3)$ and $\Gst_1(X) = \F_1(M'_X) \cong \P^2$ 
 {(by Corollary~\ref{cor:gstk} and Lemma~\ref{lem:fk-mx})}
and, since all fibers of~$g_1$ over~$\F_1(M'_X)$ are isomorphic to~$\P^1$ by Proposition~\ref{lem:hpik-fibers},
the second map in~\eqref{eqog} is one of the two standard contractions of the flag variety 
(induced by the projections of~$\P^2 \times \P^2$ to either factor);
in particular, its restriction to the  {curve}~$\OGrm_E(3, \cQ_E) \cong \P^1$ is an isomorphism 
and its image in~$\Gst_1(X) = \P^2$ is a smooth conic.

We now prove that the first map   in~\eqref{eqog}  is an isomorphism.\ 
By Proposition~\ref{prop:hpik}, its restriction to~$\OGrm_B(3, \cQ) \ssm \OGrm_E(3, \cQ_E)$ is an isomorphism  
and, as we saw above, its restriction to  {the conic}~$\OGrm_E(3, \cQ_E)$ is an isomorphism as well.\ 
Since $\OGrm_E(3, \cQ_E)$ is a smooth Cartier divisor in~$\OGrm_B(3, \cQ)$, 
it remains to show that the differential
\begin{equation*}
\mathrm{d} g_1 \colon \cN_{\OGrm_E(3,\cQ_E) / \OGrm_B(3,\cQ)} \lra \cN_{g_1(\OGrm_E(3,\cQ_E)) / \chG_1(X)}
\end{equation*}
 is injective.\ 
Note that the morphism~$f \colon \OGrm_B(3,\cQ) \to B$ is ramified along~$\OGrm_E(3,\cQ_E)$  {(see Proposition~\ref{prop:corank-special})} 
and, since~$\OGrm_B(3,\cQ) $ is contained in $ B \times \Gr(k + 2, W)$, the differential 
\begin{equation*}
\cN_{\OGrm_E(3,\cQ_E) / \OGrm_B(3,\cQ)} \lra \cN_{g_1(\OGrm_E(3,\cQ_E)) / \Gr(k + 2, W)}
\end{equation*}
of the morphism~$\OGrm_B(3,\cQ) \xrightarrow{\ g_1\ } \chG_1(X) \hookrightarrow \Gr(k + 2, W)$ must be injective, 
hence the differential of~$g_1$ is injective as well.

Summarizing all of the above and taking into account the fact that~$\G_1(X)$ is Cohen--Macaulay by Lemma~\ref{lem:g1-x56},
we obtain an equality of schemes
\begin{equation*}
\G_1(X) = 
 \Bl_{\widetilde\bp_X}(\wtY^{\ge 2}_\Ap) \cup \P^2,
\end{equation*}
where the components intersect transversely along a smooth rational curve
which is the exceptional curve of the first component  and a conic in the second component.
\end{proof}

\begin{rema}
\label{rem:involution-g1x3}
In the situation of Theorem~\ref{thm:g1-x3}, when~$X$ is ordinary, 
the canonical involution on~$\wtY^{\ge 2}_\Ap$ induces a rational involution on the Hilbert scheme~$\G_1(X)$  
which is well defined outside the point~$\bp''_X$.\ 
This involution was described geometrically (as least when~$X$ is general) in~\cite{lo} and~\cite[Section~6.2]{dim}: 
given a conic~${\Sigma} \subset X$ which is not a ${\tau}$-conic, there is a unique  {subspace}~$U_4\subset V_5$ 
such that~$\Sigma \subset \Gr(2,U_4)$ and the intersection~$X \cap \Gr(2,U_4)$ is a $1$-cycle of degree~4 that contains $\Sigma$; 
 it can be written as~$\Sigma + \Sigma'$, where~$\Sigma'\subset X$ is another conic contained in~$X$; 
the involution maps~$[\Sigma]$ to~$[\Sigma']$.

The fixed points of this involution  (which correspond to the points of the finite set~$\sY^3_\Ap$) 
were described in~\cite[Remark~6.2]{dim}.

If~$X$ is special, the involution of~$\G_1(X)$ induced by the involution of~$\tY^{\ge 2}_\Ap$ 
coincides with the involution induced by the involution of the double cover~$X \to M'_X$.\ 
The set of its fixed points is the union of~$\sY^3_\Ap$ (that parameterizes conics in~$X_0 \subset M'_X$) 
and the component~$\Gst_1(X) \cong \P^2$ (preimages of lines in~$M'_X$).\ 
The conic~$\overline{\G^0_1(X)} \cap \Gst_1(X)$ parameterizes \emph{special} lines on~$M'_X$,
that is, those, whose normal bundle is not globally generated (see~Remark~\ref{rem:special-lines}).
\end{rema}

\subsection{Quadric surfaces on GM fivefolds}

We continue with the next case~$n = 2k + 1$ and~$\ell = 2$, 
assuming now~$n = 5$, so that~$X$ is a smooth GM fivefold with associated Lagrangian~$A$.\ 
Its Pl\"ucker point~$\bp_X $ is in~$\sY^0_\Ap$ if~$X$ is ordinary, and in~$\sY^1_\Ap$ if~$X$ is special (see~\eqref{eq:px-dim});
 {in either case,~$\bp_X \notin \sY^{\ge \ell}_\Ap$.}

The following result includes Theorem~\ref{thm:g2x5-intro} as a special case.\ 
The notation  used in the description of schemes~$\Gs_2(X)$ and~$\Gt_2(X)$ can be found in Lemma~\ref{lem:fk-mx}.

\begin{theo}
\label{thm:g2-x5}
Let~$X$ be a smooth GM fivefold with associated Lagrangian~$A$.\ 
The Hilbert scheme~$\G_2(X)$ of quadric surfaces on~$X$ is a disjoint union of closed subschemes
\begin{equation*}
\G_2(X) = \G^0_2(X) \sqcup \Gs_2(X) \sqcup \Gt_2(X).
\end{equation*}
The  component~$\G^0_2(X)$ is  {a normal integral Cohen--Macaulay threefold} 
 with an \'etale-locally trivial $\P^1$-fibration~$\G^0_2(X) \to \wtY^{\ge 2}_\Ap$,
while
 \begin{align*}
\Gs_2(X)  \cong \Fs_{3}(M_X)&= 
\begin{cases}
\P(V_5) & \text{if~$X$ is ordinary,}\\
\P(\C \oplus (V_1^{{M}} \wedge V_5)^{{\vee}}) \cup \Bl_{[V_1^{{M}}]}(\P(V_5)) & \text{if~$X$ is special,}
\end{cases}
\\
\Gt_2(X) \cong \Ft_{3}(M_X) &= 
\begin{cases}
\vide & \text{if~$X$ is ordinary,}\\
\IGr(3,V_5) & \text{if~$X$ is special}.
\end{cases}
\end{align*}
If~$X$ is special, the components of~$\Gs_2(X)$ intersect along~$\P(V_1^M \wedge V_5)$,
which is identified with the hyperplane~$\P(V_1^M \wedge V_5)$ in the component~$ \P(\C \oplus (V_1^M \wedge V_5))$
and with the exceptional divisor in the other component~$\Bl_{[V_1^M]}(\P(V_5))$.
\end{theo}

\begin{proof}
The disjoint union of closed subschemes is proved in {Corollary~\ref{cor:fkm-connected}}.\ 
Moreover, $\F_3(X) $ is empty by Proposition~\ref{prop:fk-gm}, 
 hence~$\G^0_2(X) \cong \OGr_{B \ssm E}(4, \cQ)$ by  {Corollary~\ref{cor:gk0}}.\ By Proposition~\ref{prop:corank-n-k}\ref{it:ogr-bme},  it is a normal integral Cohen--Macaulay threefold 
with the required~$\P^1$-fibration.

The isomorphisms~$\G^\star_2(X) \cong \F^\star_3(M_X)$ follow from Corollary~\ref{cor:gstk},
and their explicit descriptions follow from Lemmas~\ref{lem:fk-mx} and~\ref{lem:fk-cm} 
(see also Remark~\ref{rem:f3m}).
\end{proof}

\begin{rema}
In the ordinary case, we have~$E^{\ge 2} = \vide$ by Lemma~\ref{lemma:e-stratification}, 
hence~\mbox{$\OGr_E(4,\cQ_E) = \vide$}.\ 
In the special case, we have~$E^{\ge 2} = \P^3$ and~$E^{\ge 3} = \vide$ 
{and, since any \'etale covering of~$\P^3$ is trivial,}
 $\OGr_E(4,\cQ_E)$ is a disjoint union of two ~$\P^1$-bundles over~$\P^3$.\ 
It is easy to see that one of these bundles maps isomorphically 
onto the irreducible component~\mbox{$\Bl_{[V_1^{{M}}]}(\P(V_5)) \subset \Fs_3(M_X)$}
and the other is a $\P^1$-bundle over~$\IGr(3,V_5) = \Ft_3(M_X)$.
 \end{rema}

\subsection{Hilbert squares of ordinary GM surfaces}

We now pass to the case~$n = 2k + 2$, that is,~$\ell = 1$.\ Here, the story has a \hk\ flavor.

We first consider the case~$k = 0$,  that is, 
 {$X$ is a smooth ordinary GM surface with smooth Grassmannian hull~$M_X$,}
and~$\G_0(X) \cong X^{[2]}$ is its Hilbert square, a smooth \hk\ fourfold.\ 
Proposition~\ref{prop:fk-gm} gives  {an identification~$\F_1(X) \cong \sY^3_{{A},V_5}$;
so this is a finite reduced scheme, empty for general~$X$}.\ 
Each point of~$\F_1(X) = \sY^3_{A,V_5}$
gives a Lagrangian plane~\mbox{$\P^2 \cong (\P^1)^{[2]}$} contained in~$ X^{[2]}$.\ 
We have~$\bp_X \in \sY^3_\Ap$ (see~\eqref{eq:px-dim}).\ The point in~$\wtY^{\ge 1}_\Ap$ over~$\bp_X$  is denoted by~$\widetilde\bp_X$ {(see Lemma~\ref{lem:closure})} 
and Proposition~\ref{propsing} shows that~\mbox{$\widetilde\bp_X \in \Sing(\wtY^{\ge 1}_\Ap)$} 
and describes the type of this singularity.

\begin{theo}
\label{thm:g0-x2}
If~$X$ is a  {smooth ordinary GM surface with smooth Grassmannian hull~$M_X$}, there is a diagram
\begin{equation*}
X^{[2]} = 
\G_0(X) \stackrel{\phi}{\dashrightarrow} 
\Gp_0(X) \xrightarrow{\ f^+\ } 
 \wtY_\Ap^{ {\ge 1}},
\end{equation*}
where~$\phi$ is the Mukai flop of all Lagrangian planes in~$X^{[2]}$ corresponding to points of~$\F_1(X)$,
and the morphism~$f^+$ is a symplectic resolution of singularities.\ 

In particular, $X^{[2]}$ is birationally isomorphic to~$\wtY_\Ap^{ {\ge 1}}$.
\end{theo}

\begin{proof}
Consider the scheme~$\OGr_B(2,\cQ)$.\ 
By Proposition~\ref{prop:corank-n-k}\ref{it:l2-ord}, the morphism
\begin{equation*}
\tf \colon \OGr_B(2,\cQ) \lra \Bl_{\widetilde\bp_X} {(\wtY^{\ge 1}_\Ap)}
\end{equation*}
is an isomorphism  {over the complement of}~$\sY^3_\Ap \ssm \{\bp_X\}$, 
while over every point~$b \in \sY^3_\Ap \ssm \{\bp_X\}$,  {we have~$\tf^{-1}(b) \cong \P^2$.\ 
Indeed}, the fiber~$\cQ_b$ of the conic bundle~$\cQ/B$ has corank~$3$, hence degenerates to a plane; 
since the fiber of~$\tf$ is the Hilbert scheme of lines on~$\cQ_b$, this is the dual plane.\ 

Furthermore, $\OGr_B(2,\cQ)$ is  {nonsingular} over the complement of~$\sY^3_\Ap \ssm \{\bp_X\}$, 
again by Proposition~\ref{prop:corank-n-k}\ref{it:l2-ord}.\ 
The following strengthening of the argument used  {there} 
also proves that it is nonsingular over any point~$b\in \sY^3_\Ap \ssm \{\bp_X\}$.\ 
Indeed, by~\cite[Definition~3.5 and Lemma~3.9]{DK3},  {whose notation we adopt,} to show this, 
it is enough to check that the family of quadrics~\mbox{$\cQ / B$} is $2$-regular at $b$.\ 
Since, as we have seen above, the conic~$\cQ_b$  vanishes, 
we have~$K_b = \cW_b$, and the  {morphism
\begin{equation*}
\rT_{B,b} \lra \Sym^2\!K_b^\vee = \Sym^2\!\cW_b^\vee
\end{equation*}
from~\cite[Lemma~3.9]{DK3}}
is just the tangent map to the family of quadrics.\ 
The description of the singularity of~$\sY^{\ge 1}_\Ap$ at~$b$ given in Proposition~\ref{propsing}
allows us to identify this map with the map~$\uptau \colon V_6/V_1 \to \Sym^2V_3$,
and implies that this map is injective 
and that its image is a hyperplane corresponding to a nondegenerate quadratic form on~$\cW_b^\vee$;
 {in particular, this form has no isotropic subspaces of dimension~2}.\ 
Thus, for any $2$-dimensional subspace~\mbox{$K \subset K_b = \cW_b$}, 
the composition~\mbox{$\rT_{B,b} \to \Sym^2\!\cW_b^\vee \to \Sym^2\!K^\vee$} is surjective.\ 
This proves the 2-regularity of~$\cQ/B$ {at~$b$}, which implies that~$\OGr_B(2,\cQ)$ is nonsingular over~$b$,  hence everywhere.

Consider now the diagram of Proposition~\ref{prop:diagram-big}.\ 
The smoothness of~$\OGr_B(2,\cQ)$ (proved above) and Proposition~\ref{prop:gp} 
imply that the scheme~$\Gp_0(X)$ is smooth.\ 
Moreover, it follows that the fiber of  {the morphism~$f^+ \colon \Gp_0(X) \to \wtY^{\ge 1}_\Ap$} 
over the point~{$\widetilde\bp_X \in \uvt_\Ap^{-1}(\sY^3_\Ap)  \subset \wtY^{\ge 1}_\Ap$}
is the plane~$\F_1(M_X) {{} = \Fs_1(M_X)}$  {(see Lemma~\ref{lem:fk-mx})},
and the fibers of~$f^+$ over~\mbox{$\wtY^{\ge 1}_\Ap \ssm \{\widetilde\bp_X\}$} are the same as the fibers of~$\tf$ 
{(as described at the beginning of the proof)},
hence~$f^+$ is a small resolution of singularities  {of~$\wtY^{\ge 1}_\Ap$}.\ 

{Next, we} show that the birational map
\begin{equation*}
\phi \coloneqq 
\hat\lambda_0 \circ \hat{g}_0^{-1} = (g''_0)^{-1} \circ \lambda_0 \colon \G_0(X) \dashrightarrow \Gp_0(X)
\end{equation*}
is a Mukai flop.\ 
Since~$\Gp_0(X)$ is smooth and~$g''_0$ contracts a finite union of planes~{$\P^2$} over~$\F_1(X) \subset \chG_0(X)$ (see Proposition~\ref{prop:diagram-big}),
it follows that~$\Sing(\chG_0(X)) = \F_1(X)$ and the maps~$g''_0$ and~$\lambda_0$ provide two small resolutions of~$\chG_0(X)$.\ 
It is  {therefore} enough to show that these resolutions are different over each point of~$  \F_1(X)$.\  
For this, we look at the plane~$\F_1(M_X) \subset \chG_0(X)$; 
it contains~$ \F_1(X)$  {(Lemma~\ref{lem:fkmx-dkx})}
and its strict transform in~$\Gp_0(X)$ is isomorphic to~$\F_1(M_X)$ (by Proposition~\ref{prop:gp}),
while its strict transform in~$\G_0(X)$ is isomorphic to the blowup of~$\F_1(M_X)$ along~$\F_1(X)$ (by Lemma~\ref{lem:fmx-strict-transform}).

 {Finally, since~$\G_0(X) = X^{[2]}$ is a hyper-K\"ahler fourfold and~\mbox{$\phi \colon \G_0(X) \dashrightarrow \Gp_0(X)$} is a Mukai flop,
we conclude that~$\Gp_0(X)$ is a hyper-K\"ahler fourfold.}
\end{proof}

\begin{rema}
 {The involution of~$\tY^{\ge 1}_\Ap$ induces on~$X^{[2]}$ a birational involution
analogous to the one described in Remark~\ref{rem:involution-g1x3}:
given a subscheme~$\xi \subset X$ of length~2  not lying on a line in~$M_X$,
there is a unique subspace~$U_4 \subset V_5$ such that~$\xi \subset X \cap \Gr(2,U_4)$.\ 
Then~$M_X \cap \Gr(2,U_4)$ is a conic on~$M_X$  and, if it is not contained in~$X$, 
one has~$X \cap \Gr(2,U_4) = \xi + \xi'$, where~$\xi' \subset X$ is another subscheme of length~2;
the involution takes~$\xi$ to~$\xi'$.}
\end{rema}

\begin{rema}
It is clear from the proof of the theorem that the map~$f^+$ contracts the projective plane~$\F_1(M_X)$  
to the singular point~$\widetilde\bp_X$ of~$\wtY_\Ap^{\ge 1}$ 
and, for each point of~$\sY^3_\Ap \ssm \{\widetilde\bp_X\}$, 
corresponding to a conic~$\Sigma \subset X$ (see Theorem~\ref{thm:gk-x2k}), 
the projective plane~$\langle \Sigma \rangle^\vee$ to that point.
\end{rema}

\begin{rema}
The proof also shows that~$\Gs_0(X) = \Gst_0(X) = \Gt_0(X) \cong \Bl_{\F_1(X)}(\F_1(M_X))$
 {so, in this case, it is a subscheme of~$\overline{\G_0^0(X)}$.}
 \end{rema}

\begin{rema}
 {The Mukai flop of Theorem~\ref{thm:g0-x2} coincides with the flop in~\cite[(4.2.4)]{og4}.}\ 
If~$\F_1(X) = \vide$, that is, if the surface~$X$ contains no lines, 
the diagram simplifies to a single symplectic resolution~$X^{[2]} \to \wtY_\Ap^{ {\ge 1}}$ 
 {which is induced by the isomorphism~\cite[(4.2.2)]{og4}.}
\end{rema}

\subsection{Conics on GM fourfolds}

We now consider the case~$k = 1$, $n = 4$, hence again~$\ell = 1$.
Recall that~$\bp_X \in \sY^1_\Ap$ if~$X$ is an ordinary fourfold, 
and~$\bp_X \in \sY^2_\Ap$  {if}~$X$ is a special fourfold (see~\eqref{eq:px-dim}).

The Hilbert scheme~$\F_{2}(X)=\F_2^\sigma(X)\sqcup \F_2^\rtx(X)$ of planes on the fourfold~$X$ 
was described in Proposition~\ref{prop:fk-gm}: it is finite,  {reduced},
 and empty when~$X$ is general.\
The subschemes~$\overline{\G^0_1(X)}$, $\Gs_1(X)$, 
and~$\Gt_1(X)$ of~$\G_1(X)$ were defined in Definition~\ref{def:st-quadrics}.\ 
The next theorem extends to the case of any ordinary~$X$ the description of $\G_1(X)$ obtained in~\cite{IM} when~$X$ is general.

\begin{theo}
\label{thm:g1-x4}
Let~$X$ be a smooth ordinary GM fourfold with associated Lagrangian~$A$.\  
The Hilbert scheme~$\G_1(X)$ of conics on~$X$  {is a connected Cohen--Macaulay scheme of dimension~$5$.}
 Moreover,
\begin{equation*}
\G_1(X) = \overline{\G^0_1(X)} \cup 
\Bigl( (\Fs_2(X)  \times \P^5) \sqcup  (\Ft_2(X) \times \P^5) \Bigr).
 \end{equation*}
 The subscheme $\overline{\G^0_1(X)}$ is an irreducible component of~$\G_1(X)$ and there is a diagram
\begin{equation*}
\xymatrix@R=1ex{
&&& \OGr_B(3,\cQ) \ar[dl]_{g'_1} \ar[dr]^{\tf}
\\
\overline{\G^0_1(X)} \ar@{-->}[rr]^{\phi} &&
 \Gp_1(X) 
 \ar[dr]_{f^+}&&
\Bl_{\bp'_X,\bp''_X}(\wtY_\Ap^{\ge 1}) \ar[dl]
\\
 &&&
\wtY_\Ap^{\ge 1},
}
\end{equation*}
where
\begin{itemize}
\item 
$\phi$ is a birational map,
 \item 
$g'_1$ is the blowup of the smooth $3$-dimensional quadric~$\Ft_2(M_X) \cong \IGr(3,V_5)$,
 \item 
the points~$\bp'_X,\bp''_X$ are the preimages of the Pl\"ucker point~$\bp_X \in \sY^1_\Ap$,   
\item 
the morphism~$\tf$ is an \'etale-locally trivial~$\P^1$-fibration over the complement of~$\sY^3_\Ap$
and has fibers~${\P^3 \cup_{\mathrm{pt}} \P^3}$ over~$\sY^3_\Ap$.
\end{itemize}

If~$\F_2(X) = \vide$, then~$\phi$ is an isomorphism and~$\G_1(X) = \overline{\G^0_1(X)} \cong \Gp_1(X)$ 
is a normal integral Cohen--Macaulay fivefold {with isolated singularities}.\ 
If moreover $\sY^3_\Ap = \vide$, then~$\G_1(X) \cong \Gp_1(X)$ is smooth and~$\tf$ is an \'etale-locally trivial $\P^1$-fibration.
 \end{theo}

\begin{proof}
The scheme~$\G_1(X)$ is Cohen--Macaulay of pure dimension~5 by Lemma~\ref{lem:g1-x56}.\ 
The morphism~$\lambda_1^{-1}(\F_2(X)) \to \F_2(X)$ is a $\P^5$-fibration by Proposition~\ref{prop:pik},
hence any component of the finite  {reduced} scheme~$\F_2(X)$ gives an irreducible component of~$\G_1(X)$.\ 
 {Furthermore, we know from Lemma~\ref{lem:fk-mx} that~$\dim \F_2(M_X) = 4$, 
hence the strict transforms in~$\G_1(X)$ of its components~$\Fs_2(M_X) \cong \Bl_{[V_1]}(\P(V_5))$ and~$\Ft_2(M_X) \cong \IGr(3,V_5)$
 are contained in~$\overline{\G^0_1(X)}$,
and we obtain the required description of~$\G_1(X)$ as a union.}

 {To prove that~$\G_1(X)$ is connected, it is enough to check that 
any irreducible component~$\P^5 \subset \G_1(X)$ corresponding to a point of~$\F_2(X)$ intersects with~$\overline{\G_1^0(X)}$.
This is proved in Corollary~\ref{cor:encapsulated-x4} (see Remark~\ref{rem:tdpi-gm4} for details).}

The diagram is a part of~\eqref{eq:diagram-big} in the case~$k = 1$, $n = 4$, 
with~$\phi \coloneqq \hla_1 \circ \hg_1^{-1}   = (g''_1)^{-1} \circ \lambda_1$.\ 
The description of~$g'_1$ is in Proposition~\ref{prop:gp}.\ 
The fibers of~$\tf$ are (the connected components of) the Hilbert schemes~$\F_2(\cQ_b)$ 
for the fibration~$\cQ/B$ into $3$-dimensional quadrics: 
its fibers over the complement of~$\sY^3_\Ap$ {are isomorphic to~$\P^1$ (see Proposition~\ref{prop:corank-n-k}\ref{it:l2-ord}}),
 while over any point~$b \in \sY^3_\Ap$, the quadric~$\cQ_b$ is a union of two~$\P^3$ (intersecting along a plane),
hence~$\F_2(\cQ_b)$ is the union of the two dual~$\P^3$ (intersecting at a point).\ 
 {It also follows that~$\OGr_B(3,\cQ)$ is nonsingular over the complement of~$\sY^3_\Ap$,
and the argument of Theorem~\ref{thm:g0-x2} shows that the only singularities of~$\OGr_B(3,\cQ)$
are the points~$[K_b]$ corresponding to kernel spaces~\mbox{$K_b \subset \cW_b$} 
of the quadrics~$\cQ_b$ of corank~3 with~$b \in \sY^3_\Ap$.\ Thus~$\OGr_B(3,\cQ)$ has isolated singularities.}

If~$\F_2(X) = \vide$,   {we have~$\G_1(X) = \chG_1(X) = \Gp_1(X)$ by Proposition~\ref{prop:diagram-big}.\ 
Moreover, Proposition~\ref{prop:gp} implies that this scheme has only isolated singularities.}

Finally, if in addition~$\sY^3_\Ap = \vide$, then~$\OGr_B(3,\cQ)$ is smooth by Proposition~\ref{prop:corank-n-k}\ref{it:l2-ord}, 
hence the scheme~$\G_1(X) = \chG_1(X)$ is also smooth.
\end{proof}

\begin{rema}
{The fibers of  {the morphism}~$f^+$ over the complement of the points~$\bp'_X$ and~$\bp''_X$ are the same as those of~$\tf$.\ 
Over one of these points, say~$\bp'_X$, the fiber of~$f^+$ is isomorphic to~\mbox{$\Fs_2(M_X) \cong \Bl_{[V_1^M]}(\P(V_5))$}
and, over the other, to~$\Ft_2(M_X) \cong \IGr(3,V_5)$.}
\end{rema}

\begin{rema}
\label{rem:tdpi-gm4}
By Corollary~\ref{cor:encapsulated-x4}, if~{$\F_2(X) \ne \vide$, and}~$\Pi \eqqcolon \P(R_3) \subset X$ is a plane, 
the intersection of the irreducible component~$\P(\Sym^2\!R_3^\vee) = \P^5 \subset \G_1(X)$ with~$\overline{\G_1^0(X)}$
is  {a} determinantal fourfold~$D_\Pi$  {of degree~$2$ or~$3$ if~$\Pi$ is a $\sigma$- or~$\tau$-plane, respectively.\ 
Moreover,}
 one can check that if~$\Pi$ is a $\sigma$-plane, $D_\Pi$ is a cone with vertex~$\P^1$ over~$\P^1 \times \P^1$,
and  {if~$\Pi$ is a $\tau$-plane}, $D_\Pi$ is a symmetric determinantal cubic (singular along a Veronese surface).\ 
 {Finally, the map}~$\phi$ is the composition of 
\begin{itemize}
\item 
the blowup~$\hg_1 \colon \hG_1(X) \to \overline{\G_1^0(X)}$ 
of the Weil divisor~$D_\Pi$ in~$\overline{\G_1^0(X)}$,
so that the strict transform of~$D_\Pi$ is
\begin{equation*}
\tD_\Pi \cong \P_{\P^1}(\cO^{\oplus 2} \oplus \cO(-1)^{\oplus 2})
\qquad\text{or}\qquad 
\tD_\Pi \cong \P_{\P^2}(\Sym^2\!\cT_{\P^2}),
\end{equation*}
 {if~$\Pi$ is a~$\sigma$- or a~$\tau$-plane,} respectively, and
\item  
the morphism~$\hla_1 \colon \hG_1 (X)\to \Gp_1(X)$ that contracts~$\tD_\Pi$ onto a~$\P^1$ or a~$\P^2$, respectively.
 \end{itemize}
\end{rema}
 
\subsection{Quadric surfaces on GM sixfolds}

{We now consider} the case~$k = 2$, $n = 6$, hence again~$\ell = 1$.\ 
The variety~$X$ is then a special GM sixfold (so that~$\bp_X \notin \sY^{\ge 1}_\Ap$  {by~\eqref{eq:px-dim}}) 
and~\mbox{$M_X = \CGr(2,V_5)$}.\

 {The Hilbert scheme~$\F_3(M_X)$ was described in Lemma~\ref{lem:fk-cm} and Remark~\ref{rem:f3m};
it has two connected components of dimensions~$8$ and~$6$, 
and it is easy to describe them explicitly:}
\begin{equation*}
\Fs_3(M_X) \cong \P_{\P(V_5)}(\cO \oplus \Omega^1_{\P(V_5)}(1))
\qquad\text{and}\qquad 
\Ft_3(M_X) \cong \Ft_2(M'_X) \cong \Gr(3,V_5).
\end{equation*}
The Hilbert scheme~$\F_{3}(X)$ was described in Proposition~\ref{prop:fk-gm}:
it is finite,  {reduced}, and empty for general~$X$;
moreover, $\F_3(X) {{} = \Fs_3(X)} \subset \Fs_3(M_X)$.\

\begin{theo}
\label{thm:g2-x6}
Let~$X$ be a smooth special GM sixfold with associated Lagrangian $A$.\ 
The Hilbert scheme~$\G_2(X)$ of quadric surfaces on~$X$ is a disjoint union of closed subschemes
\begin{equation*}
\G_2(X) = \G^0_2(X) \sqcup \Gs_2(X) \sqcup \Gt_2(X).
\end{equation*}
 The scheme~$\G^0_2(X)$ is  {normal integral and Cohen--Macaulay of pure} dimension~$7$ 
and there is a morphism~$\tf \colon \G^0_2(X) \to \wtY_\Ap^{\ge 1}$
which is an \'etale-locally trivial~$\P^3$-fibration over the complement of~$\sY_\Ap^3$,
while
\begin{equation*}
\Gs_2(X) = 
\Bl_{\F_3(X)}(\Fs_3(M_X)) \cup \Bigl( \F_3(X) \times \P^9 \Bigr)
\qquad\text{and}\qquad 
\Gt_2(X) = \Ft_3(M_X). 
\end{equation*}
\end{theo}

\begin{proof}
The disjoint union is proved in  {Corollary~\ref{cor:fkm-connected}};
 {it was also shown there that}
\begin{equation*}
\G^0_2(X) \cong \OGr_{B \ssm E}(4, \cQ).
\end{equation*}
 Therefore, Proposition~\ref{prop:corank-n-k}\ref{it:ogr-bme} 
 {proves all the properties of~$\G^0_2(X)$ and}
gives the required~$\P^3$-fibration.\ 
The descriptions of the other components are given by a combination of Corollary~\ref{cor:gstk} and Lemmas~\ref{lem:fk-mx} and~\ref{lem:fk-cm}.
\end{proof}

\begin{rema}
The fibers of~$\tf$ over the points of~$\sY^3_\Ap$ 
are the Hilbert schemes of $3$-dimensional linear spaces on $5$-dimensional quadrics of corank~$3$; 
they are nonnormal schemes obtained by gluing along~$\P^1 \times \P^1$ two copies of a $\P^4$-fibration over~$\P^1$.\ 
 {One can also check that the singular locus of~$\G^0_2(X)$ is the union of the surfaces~$\P^1 \times \P^1$, one over each point of~$\sY^3_\Ap$.}
\end{rema}

\begin{rema}
By Proposition~\ref{prop:corank-n-k} and Lemma~\ref{lemma:e-stratification},
the scheme~$\OGr_E(4,\cQ_E) \subset \OGr_B(4,\cQ)$ 
is the disjoint union of two~$\P^3$-fibrations over~$E \cong \P(V_5^{{\vee}})$.\ 
It is easy to check that 
{its components are~$\Fl(1,4;V_5)$ and~$\Fl(3,4;V_5)$}
and that the map~$g_2$ takes the first component isomorphically onto {the scheme}~$\Fs_2(M'_X) \subset \Fs_3(M_X)$
and contracts the second component onto~\mbox{$\Gr(3,V_5) = \Ft_2(M'_X) = \Ft_3(M_X)$}.
\end{rema}

\subsection{Summary}
\label{ss:summary}

In this section, we summarize our results about the Hilbert scheme~$\G_k(X)$
of quadrics of dimension~$k \in \{0,1,2\}$ on a smooth GM variety~$X$ of dimension~{$n \ge 2k + 1$} with associated Lagrangian subspace~$A$.\ 
Most of these results are proved in Section~\ref{sec:explicit}  
and the others can be deduced in a similar way from our results in the previous sections.

For the cases with~$k \ge 3$  {or}~$n \le 2k$, see Sections~\ref{ss:quadrics-big} and~\ref{ss:quadrics-half}.

\subsubsection{$k = 0$}
The Hilbert square~$\G_0(X) = X^{[2]} \cong \Bl_{\Delta(X)}(X \times X) / \fS_2$ 
is smooth and irreducible of dimension~$2n$ for any smooth GM variety of dimension~$n$.\ 
Moreover, the subscheme~$\lambda_0^{-1}(\F_1(X)) \subset \G_0(X)$ is a $\P^2$-bundle over~$\F_1(X)$ 
(see Proposition~{\ref{prop:pik}}).
\begin{itemize}[wide]
\item 
If~$X$ is a smooth {ordinary} GM surface with smooth Grassmannian hull,~$\G_0(X)$ is a hyper-K\"ahler fourfold 
and, after the Mukai flop of a union of disjoint planes indexed by the finite Hilbert scheme of lines~$\F_1(X)$,
it turns into a small resolution of the double dual EPW sextic~$\wtY^{\ge 1}_\Ap$ (see Theorem~\ref{thm:g0-x2}).
\item 
If~$X$ is a smooth GM threefold,~$\G_0(X)$ is birational to the scheme~$\Gp_0(X)$ 
which admits a morphism~$ {f^+ \colon {}} \Gp_0(X) \to \wtY^{\ge 0}_\Ap$,  with general fibers~$\P^1$, 
to the double cover of~$\P(V_6^\vee)$ ramified over the dual EPW sextic~$\sY^{\ge 1}_\Ap$.\ 
The birational map~$\G_0(X) \dashrightarrow \Gp_0(X)$ 
{is the relative Atiyah flop in the~$\P^2$-bundle~$\lambda_0^{-1}(\F_1(X)) \subset \G_0(X)$ over the curve~$\F_1(X)$.}
\item 
If~$\dim(X) \in \{4,5,6\}$, there is a similar  {relative Atiyah} flop~$\G_0(X) \dashrightarrow \Gp_0(X)$  
 {in the~$\P^2$-bundle~$\lambda_0^{-1}(\F_1(X)) \subset \G_0(X)$ over the scheme~$\F_1(X)$ of dimension~$2n - 5$}
and a morphism~$\Gp_0(X) \to \P(V_6^\vee)$ with general fibers~$\P^3$ (when~$\dim(X)=4$), 
$\Fl(1,3;4)$ (when~$\dim(X)=5$), 
and~$\OGr(2,7)$ (when~$\dim(X)=6$).
\end{itemize}

\subsubsection{$k = 1$}
The Hilbert scheme~$\G_1(X)$ of conics on~$X$ is a  {connected} Cohen--Macaulay scheme of pure dimension~$3n - 7$ 
(see Lemma~\ref{lem:g1-x56}  {and Corollary~\ref{cor:g1-connected}}).\ Moreover,
\begin{itemize}[wide]
\item 
If~$X$ is a smooth GM threefold, the connected surface~$\G_1(X)$ is described in Theorem~\ref{thm:g1-x3}: 
it is smooth   if and only if~$X$ is ordinary and~$\sY^3_\Ap = \vide$; 
if~$X$ is special, the main component~$\overline{\G_1^0(X)}$ is smooth if and only if~$\sY^3_\Ap = \{\bp_X\}$.
 \item 
If~$X$ is a smooth GM fourfold, the main component~$\overline{\G^0_1(X)}$ is birational to the irreducible scheme~$\Gp_1(X)$ 
which admits a morphism~$\Gp_1(X) \to \wtY^{\ge 1}_\Ap$, with general fibers~$\P^1$, to the double dual EPW sextic.\ 
If~$X$ is ordinary and~$\sY^3_\Ap = \vide$, this morphism is \'etale-locally trivial
and, if moreover $\F_2(X) = \vide$, then~$\G_1(X) \cong \Gp_1(X)$  is smooth and irreducible (see Theorem~\ref{thm:g1-x4}).
\item 
If~$X$ is a smooth GM fivefold, the main component~$\overline{\G^0_1(X)}$ is birational to the scheme~$\Gp_1(X)$
 which admits a morphism~$\Gp_1(X) \to \wtY^{\ge 0}_\Ap$, with general fibers~$\P^3$, to the double cover of~$\P(V_6^\vee)$
ramified over the dual EPW sextic~$\sY^{\ge 1}_\Ap$.\ 
If moreover~$X$ is ordinary, $\sY^3_\Ap = \vide$, and the schemes~$\Fs_2(X)$ and~$\Ft_2(X)$   
are smooth of respective (expected) dimensions~$1$  and~$0$,
then~$\G_1(X)$ is smooth  (see {Corollaries~\ref{cor:g1x-f2mx} and~\ref{cor:encapsulated-x}}).
 \item 
If~$X$ is a smooth GM sixfold, ~$\G_1(X)$ is irreducible and birational to the scheme~$\Gp_1(X)$ 
which admits a morphism~$\Gp_1(X) \to \P(V_6^\vee)$ with general fibers smooth 6-dimensional quadrics.\ 
If moreover~$\sY^3_\Ap = \vide$  and  the schemes~$\Fs_2(X)$ and~$\Ft_2(X)$ are smooth, 
then~$\G_1(X)$ is smooth  (see {Corollaries~\ref{cor:g1x-f2mx} and~\ref{cor:encapsulated-x}}).
 \end{itemize}

\subsubsection{$k = 2$}\label{ss783}

The Hilbert scheme~$\G_2(X)$ of quadric surfaces on~$X$ is a disjoint union
\begin{equation*}
\G_2(X) = \G_2^0(X) \sqcup \Gs_2(X) \sqcup \Gt_2(X)
\end{equation*}
of closed subschemes (see  {Corollary~\ref{cor:fkm-connected}}).\ 
 \begin{itemize}[wide]
\item 
If~$X$ is a smooth GM fivefold,~$\G^0_2(X)$ is irreducible of dimension~$3$
and there is an \'etale-locally trivial $\P^1$-fibration~$\G^0_2(X) \to \wtY^{\ge 2}_\Ap$.\
The other components of~$\G_2(X) $ depend  on the type of~$X$ and were described in Theorem~\ref{thm:g2-x5}.\ 
This description implies that  the component~$\G_2^0(X) \subset \G_2(X)$ is smooth if and only if~$\sY^3_\Ap = \vide$, 
the $4$-dimensional component~$\Gs_2(X)$ is smooth if and only if~$X$ is ordinary, 
and the component~$\Gt_2(X)$ is always smooth (and empty when~$X$ is ordinary).

\item 
If~$X$ is a smooth GM sixfold, the main component~$\G^0_2(X)$ is irreducible of dimension~$7$ 
and there is an \'etale-locally trivial $\P^3$-fibration~$\G^0_2(X) \to \wtY^{\ge 1}_\Ap$ (see Theorem~\ref{thm:g2-x6});
it is smooth if and only if~$\sY^3_\Ap = \vide$.\ 
 {The other components are also described in Theorem~\ref{thm:g2-x6}.}
\end{itemize}


\appendix

\section{Encapsulated conics}
\label{sec:encapsulated}

We will say that a quadric~{$\Sigma$} is {\sf encapsulated} into a projective variety~$Z \subset \P^N$ 
if its linear span~{$\langle \Sigma \rangle$} is contained in~$Z$.\ 
For instance, $\sigma$- or~$\rtx$-quadrics in~$\Gr(2,V_5)$ 
are quadrics encapsulated into~$\Gr(2,V_5) \subset \P(\bw2V_5)$ (see Definition~\ref{def:st-quadrics}).\ 
Similarly, if, as usual,~$X$ is a GM variety with Grassmannian hull~$M_X$, the subvarieties
\begin{equation*}
\P_{\F_{k+1}(X)}(\Sym^2\!\cR_{k+2}^\vee) \subset 
\P_{\F_{k+1}(M_X)}(\Sym^2\!\cR_{k+2}^\vee) 
 \end{equation*}
of $\P_{\Gr(k+2,W)}(\Sym^2\!\cR_{k+2}^\vee) = \G_k(\P(W))$ (see~\eqref{eq:gk-pw} for this equality)
parameterize the quadrics encapsulated into~$X$ and~$M_X$, respectively.\ 

In this appendix, we study conics encapsulated into~$M_X$ or into~$X$.

\subsection{Conics encapsulated into~$M_X$}

For a conic~\mbox{$\Sigma \subset \P(W)$}  {encapsulated into~$M_X$,}
we denote by~$\Pi = \langle \Sigma \rangle \subset M_X \subset \CGr(2,V_5)$ its linear span,
 {isomorphic to~$  \P^2$}.\ 
When~$\Pi$ does not contain  {the vertex~$\bv$ of~$\CGr(2,V_5)$} (this is the case for example when~$\Pi\subset X$), 
we set
\begin{equation}
\label{eq:t-pi}
\st(\Pi) \coloneqq 
\begin{cases}
0  & \text{if~$\Pi \in \Fs_2(M_X)$,}\\
1  & \text{if~$\Pi \in \Ft_2(M_X)$,}
\end{cases}
\end{equation} 
for the {\sf type} of~$\Pi$.\ We denote by~$\gamma \colon \CGr(2,V_5) \ssm \{\bv\} \to \Gr(2,V_5)$ the natural projection.\

\begin{lemm}
\label{lem:cn-pi-m}
Let~$M \subset \CGr(2,V_5)$ be a dimensionally transverse linear section of dimension~$n + 1$, 
smooth away from the vertex~$\bv$ of the cone.\ 
If~$\Pi \subset M$ is a $\sigma$- or~$\tau$-plane not containing~$\bv$, there is an exact sequence
\begin{equation*}
 {0 \to \cN_{\Pi/M} \to 
\cT_\Pi(-1)^{\oplus (\st(\Pi) + 1)} \oplus \cO_\Pi^{\oplus (1 - \st(\Pi))} \oplus \cO_\Pi(1)^{\oplus (2 - \st(\Pi))} \to 
\cO_\Pi(1)^{\oplus (6 - n)} \to 0.}
\end{equation*}
 \end{lemm}

\begin{proof}
If~$V_1 \subset V_5$, consider the punctured $4$-space~$\P(\C \oplus  (V_1 \wedge V_5)) \ssm \{\bv\} = \gamma^{-1}(\P(V_1 \wedge V_5))$.
It is easy to see that its normal bundle in~$\CGr(2,V_5)$ is
\begin{equation*}
\cN_{(\P(\C \oplus (V_1 \wedge V_5)) \ssm \{\bv\})/ \CGr(2,V_5)} \cong 
\gamma^*(\cN_{\P(V_1 \wedge V_5) / \Gr(2,V_5)}) \cong
\gamma^*\cT_{\P(V_1 \wedge V_5)}(-1).
\end{equation*}
 Therefore, for a $\sigma$-plane~$\Pi \subset \P(\C \oplus (V_1 \wedge V_5))$ not containing~$\bv$, we have an exact sequence
\begin{equation*}
0 \to \cO_\Pi(1)^{\oplus 2} \to \cN_{\Pi / {\CGr}(2,V_5)} \to (\gamma^*\cT_{\P(V_1 \wedge V_5)}(-1))\vert_\Pi \to 0.
\end{equation*}
Its last term splits as~$\cT_\Pi(-1) \oplus \cO_\Pi$ and it is easy to see that 
the $\Ext^1$-group corresponding to the above sequence is zero.\ The sequence therefore splits and we obtain an isomorphism
\begin{equation*}
\cN_{\Pi / \CGr(2,V_5)} \cong \cT_\Pi(-1) \oplus \cO_\Pi \oplus \cO_\Pi(1)^{\oplus 2}.
\end{equation*}
This easily implies the lemma for the $\sigma$-plane~$\Pi$.

 If~$V_3 \subset V_5$, 
consider the punctured $3$-space~$\P(\C \oplus \bw2V_3) \ssm \{\bv\} = \gamma^{-1}(\P(\bw2V_3))$.
Its normal bundle in~$\CGr(2,V_5)$ is
\begin{equation*}
\cN_{(\P(\C \oplus \wedge^2V_3) \ssm \{\bv\}) / \CGr(2,V_5)} \cong 
\gamma^*(\cN_{\P(\wedge^2V_3)/\Gr(2,V_5)}) \cong
\gamma^*(\cT_{\P(\wedge^2V_3)}(-1) \oplus \cT_{\P(\wedge^2V_3)}(-1)).
\end{equation*}
 This easily implies the lemma for the $\tau$-plane~$\Pi= \P(\bw2V_3)$.
\end{proof}

The exact sequence of Lemma~\ref{lem:cn-pi-m}   allows us to compute the cohomology of~$\cN_{\Pi/M}$.
 
\begin{coro}
\label{cor:cn-pi-m}
 Let~$M \subset \CGr(2,V_5)$ be a dimensionally transverse linear section of dimension~$n + 1$, smooth away from the vertex~$\bv$ of the cone.\ 
If~$\Pi \subset M$ is a $\sigma$- or~$\tau$-plane not containing~$\bv$, then
\begin{equation*}
h^i(\Pi, \cN_{\Pi/M})  =
\begin{cases} 
3n - 8 -  {\st(\Pi)} & \text{if~$i = 0$,}\\
0 & \text{if~$i \ne 0$,}
\end{cases}
\qquad
h^i(\Pi, \cN_{\Pi/M}(-2))  =
\begin{cases} 
 \st(\Pi) + 1 & \text{if~$i = 1$,}\\
0 & \text{if~$i \ne 1$.}
\end{cases}
\end{equation*}
 \end{coro}

\begin{proof}
Planes~$\Pi$ as in the lemma are parameterized by the schemes~$\Fo^\sigma_2(M)$ and~$\Fo^\tau_2(M)$,
which are smooth of respective dimensions~\mbox{$3n - 8$} and~\mbox{$3n - 9$}  
(see Lemmas~\ref{lem:fk-mx} and~\ref{lem:fk-cm}).\ 
Therefore,  {we have}~$h^0(\Pi, \cN_{\Pi/M})  = 3n - 8 -  {\st(\Pi)}$ 
(because the space~$H^0(\Pi, \cN_{\Pi/M}) $ is the Zariski tangent space to~$\F_2(M)$ at~$[\Pi]$).\ 
The exact sequence of Lemma~\ref{lem:cn-pi-m} then gives the vanishing of~$h^1(\Pi, \cN_{\Pi/M})$ and~$h^2(\Pi, \cN_{\Pi/M})$.

 {To compute}~$h^i(\Pi, \cN_{\Pi/M}(-2))$, we twist this exact sequence by~$\cO_\Pi(-2)$
and note that the line bundles~$\cO_\Pi(-1)$ and~$\cO_\Pi(-2)$ have zero cohomology,
while the twisted tangent bundle~$\cT_\Pi(-3) \cong \Omega_\Pi$ has $1$-dimensional cohomology in degree~$1$.
\end{proof}

Using this, we can prove the following result (see~\eqref{defmpx} for the definition of~$M'_X$). 
 
\begin{prop}
\label{prop:sigma-mx}
Let~$X$ be a smooth GM variety of dimension~$n \ge 3$  and 
let~$\Sigma \subset X$ be a conic on~$X$ encapsulated into~$M_X$,   but not into~$X$.\ 
Then~$\G_1(X)$ is smooth of dimension~$3n-7$ at~$[\Sigma]$, 
except if~$X$ is special, $\bv \in \langle \Sigma \rangle$, and~$\gamma(\Sigma) \subset M'_X$ is a line whose normal bundle is not globally generated, 
in which case~{$\G_1(X)$} is singular at~$[\Sigma]$.
\end{prop}

\begin{proof}
We saw in Lemma~\ref{lem:g1-x56} that~$\G_1(X)$ has pure dimension~\mbox{$3n-7$}; 
in particular, by standard deformation theory, if~$\Sigma \subset X$ is a conic, 
the scheme~$\G_1(X)$ is smooth (of dimension~\mbox{$3n-7$}) at the point~$[\Sigma]$ 
if and only if~$h^0(\Sigma, \cN_{\Sigma/X}) =   3n - 7$  or, equivalently, if~$h^1(\Sigma, \cN_{\Sigma/X}) =  0$. 

 {As before}, set~$\Pi = \langle \Sigma \rangle   \cong \P^2$.\ 
Since~$\Sigma$ is not encapsulated into~$X$, we have~$\Sigma = \Pi \cap X$,
and therefore~$\cN_{\Sigma/X} \cong \cN_{\Pi/M_X}\vert_\Sigma$.\ 
If~$X$ is ordinary, or~$X$ is special and~$\bv \not\in \Pi$, we use the exact sequence
\begin{equation*}
0 \to \cN_{\Pi/M_X}(-2) \to \cN_{\Pi/M_X} \to \cN_{\Sigma/X} \to 0
\end{equation*}
and Corollary~\ref{cor:cn-pi-m}; 
they imply~$h^0(\Sigma, \cN_{\Sigma/X}) =   (3n - 8 - \st(\Pi)) + (\st(\Pi) + 1) = 3n - 7$ 
and the higher cohomology groups vanish.

 {If~$X$ is special and~$\bv \in \Pi$},
we argue as in Lemma~\ref{lem:g1-cgr}.\ 
First, we have an isomorphism
 \begin{equation*}
\cN_{\Sigma/X} \cong \cN_{\Pi/M_X}\vert_\Sigma \cong \gamma^*\cN_{L/M'_X},
\end{equation*}
where~{$L = \gamma(\Sigma)$}.\ Applying the projection formula, we obtain
\begin{equation*}
\mathbf{R}\gamma_*\cN_{\Sigma/X} \cong
\mathbf{R}\gamma_*\gamma^*\cN_{L/M'_X} \cong
\cN_{L/M'_X} \oplus \cN_{L/M'_X}(-1).
\end{equation*}
The right side has no higher cohomology if and only if the bundle~$\cN_{L/M'_X}$ is globally generated, 
in which case the Riemann--Roch theorem implies~$h^0(\Sigma, \cN_{\Sigma/X}) ={3n - 7}$.

In both cases, the proposition follows.
\end{proof}

 {If~$X$ is a special GM variety, a line on~$M'_X$ is called {\sf special} if its normal bundle is not globally generated;
the corresponding subscheme of~$\F_1(M'_X)$ is denoted by~$\F_1^{\mathrm{spe}}(M'_X)$.}

\begin{rema}
\label{rem:special-lines}
 If~$X$ is a smooth special GM variety of dimension~$n \in \{3,4,5\}$, 
the subscheme~$\F_1^{\mathrm{spe}}(M'_X) \subset \F_1(M'_X)$ of special lines is 
\begin{itemize}[wide]
\item 
a smooth conic 
  for~$n = 3$,
\item 
a Hirzebruch surface in the exceptional divisor of~$\Bl_{\upkappa(\P(W_0^\perp))}(\P(V_5))$ for~$n = 4$,
\item 
the quadric~$\IGr(3,V_5)$ in the exceptional divisor of~$\Bl_{\IGr(3,V_5)}(\Gr(3,V_5))$ for~\mbox{$n = 5$},
\end{itemize}
where in all cases we use the description of~$\F_1(M'_X)$ from Lemma~\ref{lem:fk-mx}.
\end{rema}
 
 {In the following corollary, we identify the scheme~$\F_2(M_X) \ssm \F_2(X)$ with its preimage in~$\G_1(X)$
under the morphism~$\lambda_1$ (see Proposition~\ref{prop:pik}).}

\begin{coro}
\label{cor:g1x-f2mx}
If~$X$ is a smooth GM variety of dimension~$n \ge 3$, the scheme~$\G_1(X)$ is smooth along~$\F_2(M_X) \ssm \F_2(X)$
unless~$X$ is special and~$  n \le 5$, in which case
\begin{equation*}
\Sing(\G_1(X)) \cap (\F_2(M_X) \ssm \F_2(X)) = \F_1^{\mathrm{spe}}(M'_X).
\end{equation*}
 In particular, $\G_1(X)$ is smooth of dimension~$3n-7$ at general points of every irreducible component of~$\F_2(M_X) \ssm \F_2(X)$.
\end{coro}

\begin{proof}
This is just a restatement of Proposition~\ref{prop:sigma-mx}, together with  the fact that~$\F_1^{\mathrm{spe}}(M'_X)  $ is empty  if~$n = 6$,
and that~${\F_2(M_X)} \ssm \F_1^{\mathrm{spe}}(M'_X)$ is dense in~${\F_2(M_X)}$ for~$n \le 5$.
\end{proof}

\subsection{Conics encapsulated into~$X$}

 {If~$\Sigma$ is a conic encapsulated into~$X$}, the span~$\Pi \coloneqq \langle \Sigma \rangle$ is contained in~$X$,
hence it does not contain the vertex~$\bv$ of~$\CGr(2,V_5)$ (because~$\bv\notin X$).

\begin{prop}
\label{prop:encapsulated-filtration}
Let~$\Sigma \subset X$ be a  conic encapsulated into a smooth GM variety~$X$  {of dimension~$n \ge 3$}
and let~$\Pi \coloneqq \langle \Sigma \rangle$ be its linear span.\
There are exact sequences of sheaves
\begin{align*}
0 \to \cO_\Pi(2)\vert_\Sigma \to \cN_{\Sigma/X} &\to \cN_{\Pi/X}\vert_\Sigma \to 0,\\
 0 \to \cN_{\Pi/X}(-2) \to \cN_{\Pi/X} &\to \cN_{\Pi/X}\vert_\Sigma \to 0.
\end{align*}
In particular, if~
\begin{equation*}
 {H^1(\Pi, \cN_{\Pi/X}) = 
H^2(\Pi, \cN_{\Pi/X}(-2)) = 0},
\end{equation*}
then~$\G_1(X)$ is smooth at~$[\Sigma]$. 
\end{prop}

\begin{proof}
The first exact sequence follows from the isomorphism~$\cN_{\Sigma/\Pi} \cong \cO_\Pi(2)\vert_\Sigma$ and the second is obvious.\ 
Using the second sequence and the hypotheses, 
we obtain~{$h^1(\Sigma, \cN_{\Pi/X}\vert_\Sigma) = 0$}.\ 
Since~{$h^1(\Sigma, \cO_\Pi(2)\vert_\Sigma) = 0$},
we obtain from the first sequence the vanishing~{$h^1(\Sigma, \cN_{\Sigma/X}) = 0$}, 
hence~$\G_1(X)$ is smooth at~$[\Sigma]$.\ 
\end{proof}

In the corollary below, we use notation~\eqref{eq:t-pi} (recall that $\bv\notin\Pi$).

\begin{coro}
\label{cor:encapsulated-x}
Let~$X$ be a smooth GM variety of dimension~$n \ge 5$.\ 
Let~$\Sigma \subset X$ be a conic encapsulated  {into}~$X$ 
and let~$\Pi \coloneqq \langle \Sigma \rangle\subset X$ be its linear span.\ 
If the scheme~$\F_2(X)$ is smooth of dimension~$3n - 14 -  {\st(\Pi)}$ at~$[\Pi]$, 
then~$\G_1(X)$ is smooth. 
 \end{coro}

\begin{proof}
Since~$\Pi \subset X \subset M_X$, we have an exact sequence
\begin{equation*}
0 \to \cN_{\Pi/X} \to \cN_{\Pi/M_X} \to \cO_\Pi(2) \to 0.
\end{equation*}
By assumption, we have~$h^0(\Pi, \cN_{\Pi/X})=3n - 14 - {\st(\Pi)}$ 
({because} the space~$H^0(\Pi, \cN_{\Pi/X}) $ is the Zariski tangent space to~$\F_2(X)$ at~$[\Pi]$).\ 
Moreover, since~$\Pi$ does not contain the vertex~$\bv$ of~$M_X$,
Corollary~\ref{cor:cn-pi-m} describes the cohomology of the middle term.\ 
Finally, the cohomology of the right term is~$\C^6$ in degree~$0$ and {zero} in positive degrees.\ 
All this implies that~{$H^1(\Pi, \cN_{\Pi/X})  $ vanishes}.

Similarly, twisting the above sequence by~$\cO_\Pi(-2)$ and using Corollary~\ref{cor:cn-pi-m} again, 
we obtain~$h^1(\Pi, \cN_{\Pi/X}(-2)) = {\st(\Pi)} + 2$ and the other cohomology groups  {of~$\cN_{\Pi/X}(-2)$} vanish.

With these computations and Proposition~\ref{prop:encapsulated-filtration},
we see that~$\G_1(X)$ is smooth at~$[\Sigma]$.
 \end{proof}

Finally, we  {consider} conics encapsulated {into} a smooth GM fourfold~$X$.\ 
 {The following result adds some details to the description of Theorem~\ref{thm:g1-x4}, see also Remark~\ref{rem:tdpi-gm4}.}\ 
We know from Proposition~\ref{prop:fk-gm} that the scheme~$\F_2(X)$ is finite  {and reduced}, and empty for~$X$ general.

\begin{coro}
\label{cor:encapsulated-x4}
Let~$X$ be a smooth GM fourfold and let~$\Pi = \P(R_3)$ be a plane on~$X$.\ 
 Then, the space~$\P(\Sym^2\!R_3^\vee) \isom \P^5 $ of conics contained in~$\Pi$ is  {an irreducible} component of~$\G_1(X)$.\ 
 {It~intersects the main component~$\overline{\G^0_1(X)}$}
 along a determinantal hypersurface~$D_\Pi \subset \P(\Sym^2\!R_3^\vee)$ of degree~$ {\st(\Pi)} + 2$.
\end{coro}

\begin{proof}
{Since~$\F_2(X)$ is finite and reduced by Proposition~\ref{prop:fk-gm},}
we have~$h^0(\Pi, \cN_{\Pi/X})=0$.\ 
 {In particular, $N_{\Pi/X}$ is a stable rank~$2$ vector bundle on~$\Pi$ with   first Chern class~$-1$.\ 
Furthermore,} the   proof of Corollary~\ref{cor:encapsulated-x} shows that in this case, we have 
\begin{equation*}
h^1(\Pi, \cN_{\Pi/X}(-2)) = h^1(\Pi, \cN_{\Pi/X}) = \st(\Pi) + 2
 \end{equation*}
 and the other cohomology groups of~$\cN_{\Pi/X}$ and~$\cN_{\Pi/X}(-2)$ vanish.
 
Let~$\Sigma\subset \Pi$ be a conic.\ 
 {Since~$\dim(\G_1(X)) = 5$ (Lemma~\ref{lem:g1-x56}), 
the scheme~$\G_1(X)$ is smooth at~$[\Sigma]$
if and only if~$h^0(\Sigma, \cN_{\Sigma/X}) = 5$.\ 
The first exact sequences of Proposition~\ref{prop:encapsulated-filtration} shows 
that this is equivalent to the vanishing of~$H^0(\Sigma, \cN_{\Pi/X}\vert_\Sigma)$,
and the second exact sequence of Proposition~\ref{prop:encapsulated-filtration} 
identifies~$H^0(\Sigma, \cN_{\Pi/X}\vert_\Sigma)$ with the kernel of the map}
 \begin{equation*}
\C^{ {\st(\Pi)} + 2} = H^1(\Pi, \cN_{\Pi/X}(-2)) \xrightarrow{\ {}\cdot\Sigma\ } H^1(\Pi, \cN_{\Pi/X}) = \C^{{\st(\Pi)} + 2}.
\end{equation*}
Letting~$[\Sigma]$ vary in~$\P(\Sym^2\!R_3^\vee)$, we see that this map is the fiber at~$[\Sigma]$ of a morphism 
\begin{equation*}
\C^{ {\st(\Pi)} + 2} \otimes \cO_{\P(\Sym^2\!R_3^\vee)}(-1) \lra 
\C^{ \st(\Pi)  + 2} \otimes \cO_{\P(\Sym^2\!R_3^\vee)}
\end{equation*}
of vector bundles.\ Denoting by~$D_\Pi$ its degeneracy locus, 
we obtain
\begin{equation*}
 {\Sing(\G_1(X)) \cap \P(\Sym^2\!R_3^\vee) = D_\Pi.}
\end{equation*}
 Note that~$D_\Pi \ne \P(\Sym^2\!R_3^\vee)$ because, for a general smooth conic~$\Sigma \subset \Pi$, 
we have  {and isomorphism}~$N_{\Pi/X}\vert_\Sigma \cong \cO_{\Sigma}(-1)^{\oplus 2}$ 
by the Grauert--M\"ulich Theorem (see~\cite[Th\'eor\`eme]{H})
and the stability of~$N_{\Pi/X}$.

 {Since~$\G_1(X)$ is Cohen--Macaulay of dimension~5 by Lemma~\ref{lem:g1-x56},
so that~$\P(\Sym^2\!R_3^\vee)$ is a component of~$\G_1(X)$,
we conclude that} another component of~$\G_1(X)$ must intersect~$\P(\Sym^2\!R_3^\vee) $ along~$D_\Pi$.\
By Theorem~\ref{thm:g1-x4}, this must be the main component~$\overline{\G^0_1(X)}$.
 \end{proof}

\end{document}